\documentclass[openany, oneside]{book}

\pdfoutput=1

\usepackage{amsmath, mathtools, comment}
\usepackage{amsfonts}
\usepackage{natbib}
\usepackage{amssymb}
\usepackage{amsthm}
\usepackage{tikz-cd}
\usetikzlibrary{decorations.pathmorphing}
\usepackage{tikz}
\usepackage{amsxtra}
\usepackage{wrapfig}
\usepackage{color}
\usepackage{url}
\usepackage[hidelinks]{hyperref}
\usepackage[percent]{overpic}
\usepackage{mathrsfs}
\usepackage{stmaryrd}

\usepackage{setspace, cleveref, rotating, adjustbox}

\newtheorem{thm}{Theorem}[chapter]
\newtheorem{prop}[thm]{Proposition}

\newtheorem{cor}[thm]{Corollary}

\newtheorem{lem}[thm]{Lemma}
\theoremstyle{definition}
\newtheorem{defn}[thm]{Definition}

\newtheorem{exer}[thm]{Exercise}
\newtheorem*{thma}{Theorem}

\theoremstyle{remark}
\newtheorem{ex}[thm]{Example}
\newtheorem{rmk}[thm]{Remark}

\newcommand{\cat}[1]{{\mathbf{#1}}}
\newcommand{\p}{\paragraph{}}
\newcommand{\spec}{\operatorname{Spec}}
\newcommand{\onto}{\twoheadrightarrow}
\newcommand{\into}{\hookrightarrow}
\newcommand{\from}{\leftarrow}
\newcommand{\lot}{\otimes^{\mathbb{L}}}
\newcommand{\per}{{\ensuremath{\cat{per}}}\kern 1pt}

\newcommand{\Mod}{\cat{Mod}\text{-}}

\DeclareMathOperator{\id}{id}

\let\ker\relax\DeclareMathOperator{\ker}{ker}

\let\hom\relax\newcommand{\hom}{\mathrm{Hom}}
\newcommand{\enn}{\mathrm{End}}
\DeclareMathOperator{\tor}{Tor}
\DeclareMathOperator{\ext}{Ext}

\DeclareMathOperator{\map}{Hom}

\newcommand{\A}{\mathcal{A}}
\newcommand{\B}{\mathcal{B}}
\newcommand{\C}{\mathcal{C}}
\newcommand{\D}{\mathcal{D}}

\newcommand{\Q}{\mathbb{Q}}
\newcommand{\bbC}{\mathbb{C}}
\newcommand{\Z}{\mathbb{Z}}
\newcommand{\N}{\mathbb{N}}
\renewcommand{\P}{\mathbb{P}}
\newcommand{\R}{{\mathrm{\normalfont\mathbb{R}}}}

\newcommand{\rmap}{\R\kern -1.5pt \map}

\newcommand{\tbtm}[4]{\ensuremath{\begin{pmatrix}#1&#2\\#3&#4\end{pmatrix}}}
\newcommand{\stbtm}[4]{\ensuremath{\left(\begin{smallmatrix}#1&#2\\#3&#4\end{smallmatrix}\right)}}

\newcommand{\dq}{A/^{\mathbb{L}}\kern -2pt AeA}

\newcommand{\dco}{D^\mathrm{co}}

\numberwithin{equation}{section}

\usepackage[utf8]{inputenc}
\usepackage[T1]{fontenc}
\usepackage{textcomp}

\usepackage[nolabel]{showlabels}

\title{Lectures on singularity categories}
\author{Matt Booth}

\begin{document}
\frontmatter
	
	\begin{titlepage}
	\maketitle
	\end{titlepage}

	\chapter*{Introduction}

   Let $A$ be a (not necessarily commutative) noetherian ring. Associated to $A$ are various invariants which tell us about the geometry, homological algebra, representation theory, etc.\ of $A$. Often one thinks of an invariant as being a number, or maybe a vector space, or at least something reasonably small and concrete. But there are several natural categories attached to $A$, which one can view as invariants in their own right. These categorical invariants contain a lot more structure and information than more classical invariants, but at the downside of being more complicated objects. Examples include:
	
	\begin{enumerate}
		\item The abelian categories $\mathbf{mod}\text{-}A$ and $\Mod A$. These are rather strong invariants: Morita theory tells us that when $A$ and $A'$ have equivalent module categories, then $Z(A)\cong Z(A')$. In particular these are complete invariants for commutative rings. In the non-affine setting, one has the Gabriel-Rosenberg theorem, which says that if $X$ is a noetherian scheme then $X$ can be recovered from $\mathbf{Coh}(X)$ or $\mathbf{QCoh}(X)$.
        
		\item The triangulated category $D(\Mod A)$, defined to be the localisation of the category of chain complexes of $A$-modules at the quasi-isomorphisms (homology isomorphisms). This category is the natural home for the homological algebra of $A$-modules. Variants include  $D^b(\mathbf{mod}\text{-}A)$, where one restricts to bounded complexes of finitely generated modules, and $\per(A)$, where one further restricts to finitely generated projective modules. These categories are looser invariants that still know about the homological properties of $A$. For an example of the utility of this approach, $\mathbb{P}^1$ is derived equivalent to the path algebra of the Kronecker quiver $\bullet \rightrightarrows\bullet$, although their abelian categories of sheaves/modules are not equivalent. So homological algebra over $\mathbb{P}^1$ is the same as homological algebra over $\bullet \rightrightarrows\bullet$. Recovery theorems are more difficult to come by here; one famous example in the geometric setting is the Bondal-Orlov reconstruction theorem, which says that if $X$ is a smooth projective variety with (anti)ample canonical bundle then $D^b(X)$ recovers $X$ (for a nice exposition of this see \cite{caldararu}).
		
		\item Differential graded enhancements of the above triangulated categories; loosely these keep track of the derived hom-complexes rather than just the Ext groups. One is typically interested in dg enhancements since triangulated categories have bad formal properties:
        \begin{itemize}
        \item Mapping cones are not functorial. 
            \item The category of triangulated categories doesn't have internal homs.
            \item Triangulated categories don't satisfy any reasonable form of geometric descent.
            \item Invariants like Hochschild, cyclic, or periodic cyclic co/homology are defined using the dg structure, and it is unknown if they can be recovered from the triangulated structure alone.
        \end{itemize}
         All of these problems are solved by working with pretriangulated dg categories instead of triangulated categories.

		\item If $A$ is reasonably commutative (for example, an enveloping algebra or an $E_2$-algebra), then one can equip $D(A)$ or $\per(A)$ with a monoidal structure given by a derived tensor product. This is the starting point for the subject of tensor triangular geometry.
		\end{enumerate}
	
	By definition, $\per(A)$ is a subcategory of $D^b(\mathbf{mod}\text{-}A)$. One natural question to ask is - what sort of difference is there between these two categories? For an example, let $k$ be a field and consider the ring $A\coloneqq k[x]/x^2$. The module $k$ has a  projective resolution given by $$\cdots \xrightarrow{x}A \xrightarrow{x}A \xrightarrow{x}A \xrightarrow{x}A$$and using this one can compute that $\ext^i(k,k)\cong k$ for all $i\geq 0$. Since perfect complexes must have bounded self-Exts, $k$ cannot be perfect. Hence $k\in D^b(A)$ is not an object of the subcategory $\per(A)$. More formally, trying to measure the difference between the triangulated category $D^b(\mathbf{mod}\text{-}A)$ and its triangulated subcategory $\per(A)$ leads us to consider the Verdier quotient $D^b(\mathbf{mod}\text{-}A) / \per(A)$. This triangulated category is known as the \textbf{singularity category} of $A$ and will be the main object of our study. In particular, we will be interested in the sort of information that the singularity category $D_{\text{sg}}(A)\coloneqq D^b(\mathbf{mod}\text{-}A) / \per(A)$ knows about $A$. The above calculation shows that $D_{\text{sg}}(k[x]/x^2)$ does not vanish --- in fact, it is $D^b(k)$ (as long as $k$ is algebraically closed). The foundational result in this field, and the reason for the name ``singularity category'' (coined by Orlov), is the following result:
	\begin{thma}
		Let $A$ be a commutative noetherian ring of finite Krull dimension. Then $A$ is regular if and only if $D_{\text{sg}}(A)$ vanishes.
		\end{thma}
	Intuitively, a commutative ring is regular when the corresponding scheme is smooth. So for example, the ring $\mathbb{C}[x_1,\cdots,x_n]$ is regular, since it is the coordinate ring of complex $n$-space. The rings $\mathbb{C}[x,y]/xy$, $\mathbb{C}[x,y]/(x^2-y^3)$, and $\mathbb{C}[x,y]/(x^3+x^2-y^2)$ are not smooth, since they are respectively the coordinate rings of the coordinate axes in the plane, the cuspidal cubic, and the nodal cubic, all of which have singular points. With this in mind we regard the singularity category of $A$ as a homological invariant that detects the singularities of $A$, even when $A$ is noncommutative. Along the way to proving the above theorem, we will see a purely homological characterisation of smoothness in terms of global dimension, which measures the lengths of projective resolutions of $A$-modules; it is intuitively clear that the existence of finitely generated modules without a bounded projective resolution is an obstruction to the vanishing of $D_{\text{sg}}(A)$.
 	
 	After we prove the above theorem, we will see two important alternate descriptions of the singularity category, valid in various levels of generality. The first is as follows:

    \begin{thma}[Buchweitz]
        Suppose that $A$ is (Iwanaga)--Gorenstein; i.e.\ $A$ has finite injective dimension over itself. Then there is a triangle equivalence between $D_{\text{sg}}(A)$ and the stable category of maximal Cohen--Macaulay modules $\underline{\mathbf{MCM}}(A)$.
    \end{thma}

 	We say that a finitely generated $A$-module $X$ is MCM if $\ext_A^i(X,A)$ vanishes for $i>0$; in the commutative setting, there is an equivalent characterisation in terms of depth. Loosely, the stable category of MCM modules is what one gets by taking the category of MCM modules and quotienting out the projectives. The shift of an MCM module is its (inverse) syzygy. As originally noticed by Eisenbud, when $A$ is a complete local hypersurface singularity, the syzygy functor is its own inverse, which makes $\underline{\mathbf{MCM}}(A)$ into a 2-periodic triangulated category. In this setting there is a simpler description that makes use of the 2-periodic structure, which gives us our second alternate description:

  \begin{thma}
      Suppose that $R=k\llbracket x_1,\ldots, x_n\rrbracket / f$ is a complete local isolated hypersurface singularity. Then there is a triangle equivalence between $D_{\text{sg}}(R)$ and the (homotopy) category of matrix factorisations $\mathbf{MF}(k\llbracket x_1,\ldots, x_n\rrbracket , f)$.
      \end{thma}
      A matrix factorisation of $f$ is a pair of free finite rank $k\llbracket x_1,\ldots, x_n\rrbracket$-modules $M$ and $N$ together with `differentials' $d:M \to N$ and $d:N \to M$ satisfying $d^2 = f$. One can alternately think of this as 2-periodic `complex', where the differentials do not square to zero, but rather to $f$; indeed there is an interpretation here in terms of modules over a certain curved dg algebra. One can organise the collection of matrix factorisations into a category, and $\mathbf{MF}(k\llbracket x_1,\ldots, x_n\rrbracket , f)$ is given by quotienting this category by a suitable notion of homotopy.

      Armed with these three alternate descriptions, we will then proceed to enhance them to the level of dg categories. This will allow us to talk about invariants like the Hochschild co/homology of singularity categories, which we will compute in some special cases. In particular we will see the following theorem:

      \begin{thma}
      Let $R=k\llbracket x_1,\ldots, x_n\rrbracket / f$ be a complete local isolated hypersurface singularity. The 2-periodic Hochschild cohomology of $D_\mathrm{sg}(R)$ is the Milnor algebra
          $$M_f\coloneqq \frac{R}{{\left(\frac{\partial f}{\partial x_1},\ldots,\frac{\partial f}{\partial x_n}\right)}}$$concentrated in even parity. Moreover, the non-2-periodic $\mathrm{HH^0}(D_\mathrm{sg}(R))$ is the Tjurina algebra $M_f/f$.
      \end{thma}
      In particular, one can recover both the Milnor and Tjurina algebras of $f$ -- classical objects from singularity theory -- from the singularity category $D_\mathrm{sg}(R)$ (possibly with some extra 2-periodicity data attached).

      Following this, we will consider relative singularity categories, and especially concrete models for them in terms of derived quotients. As an aside, we will see some applications of this technology to the classification of isolated hypersurface singularities. Following work of Kalck and Yang, we will then use these concrete models to give explicit computations of singularity categories of noncommutative deformations of Kleinian singularities originally due to Crawford. Finally, we will see that various duality results for singularity categories can be placed within the context of Koszul duality, which is a duality theory for dg co/algebras. In particular, our earlier Hochschild cohomology computation can be viewed as an instance of an easier computation `across Koszul duality'.
  
 	\vfill\pagebreak

\section*{How to use these notes}
These notes are adapted from a seminar course that I ran in Lancaster in 2022/2023, and from an online lecture course at Imperial that I ran for the Taught Course Centre in 2025. Each Chapter is roughly supposed to correspond to two hours of lecture material. In the interests of brevity, many proofs are hence sketches, given in varying levels of detail. The main references used for each Chapter are given at the start, rather than being given inline; any unreferenced results are to be attributed to those citations and in particular I make no claims to originality of any of the material presented here.

These notes are aimed at early graduate students; the reader should have a basic knowledge of rings, algebras, and their modules, and in particular knowledge of tensor products and the hom-tensor adjunction. Knowledge of category theory is a must, and it would be helpful to be comfortable with abelian categories. Basic knowledge of affine algebraic geometry is useful but not strictly required. The first five chapters are more `classical' and well-known material, and the second five are harder and more oriented towards recent research. A brief `further reading' section is provided at the end, in part to mention topics not included in the main text.

There are several scattered exercises to test the reader's understanding which I encourage attempting. A star indicates a harder exercise; these should be thought of as optional.

    \section*{Acknowledgments}
    I have benefited over the years from many helpful conversations in and around singularity categories with Xiao-Wu Chen, John Greenlees, Martin Kalck, Bernhard Keller, Andrey Lazarev, David Pauksztello, Jonathan Pridham, Michael Wemyss, Nicholas Williams, and Dong Yang. I would like to thank all the participants of the Lancaster seminar and the Imperial course, for directly or indirectly contributing improvements to the present text.

	\tableofcontents
\mainmatter		 
\chapter{Derived categories}
We begin with our basic object of study, the derived category of a ring. For far more comprehensive treatments than this section provides, see \cite{weibel} or \cite{yekbook}. If you are geometrically inclined, you will enjoy \cite{thomasworking}.

\section{Chain complexes and quasi-isomorphisms}

Let $A$ be a ring. All modules are right modules unless otherwise specified. A \textbf{cochain complex} is a $\Z$-indexed sequence of $A$-modules $\{M^n\}_{n\in \Z}$ together with differentials $d:M^n \to M^{n+1}$ such that $d^2=0$. A \textbf{chain complex} is defined similarly, except that the differentials lower the degree; in this case we typically write the indices as subscripts. One can convert between homological and cohomological notation by setting $M^n = M_{-n}$. For us, the term \textbf{complex} will always mean a cochain complex. Observe that the complexes concentrated in degree zero are precisely the $A$-modules.

If $M$ is a complex then its \textbf{shift} $M[1]$ is the complex with $M[1]^n = M^{1+n}$ and differential $d_{M[1]}=-d_{M}$. Graphically, this corresponds to shifting the complex \textit{left}. There are analogous shifts $M[n]$ for all integers $n$, and we have $M[i][j]\cong M[i+j]$.

\begin{rmk}
    The \textbf{Koszul sign rule} says that when an object of degree $p$ moves past an object of degree $q$, then a sign change of $(-1)^{pq}$ is required. All of the sign conventions in this book can be worked out with a careful (if, sometimes, non-obvious!) use of the Koszul sign rule.
\end{rmk}

A \textbf{morphism} of complexes $M\to N$ is simply a collection of maps $M_n \to N_n$, with no compatibilities. A \textbf{morphism of degree} $d$ is a morphism $M[d] \to N$. A morphism $f:M \to N$ is a \textbf{chain map} if it is compatible with the differentials in the sense that $d_N f =(-1)^{\deg f}fd_M$.

An $n$-\textbf{cocycle} in a complex $M$ is an element $m\in M^n$ such that $dm=0$. We denote the $A$-module of $n$-cocycles by $Z^n(M)$. An $n$-\textbf{coboundary} is an element $m\in M^n$ of the form $m=dn$. We denote the module of $n$-coboundaries by $B^n(M)$. Clearly we have $B^nM \subseteq Z^nM$, and the quotient $Z^n(M)/B^n(M)$ is the $n^\text{th}$ \textbf{cohomology} module $H^n(M)$. We often assemble the cohomology of $M$ into a complex $H(M)$ with zero differential.

\begin{exer}[Mapping complexes]\hfill
\begin{enumerate}
    \item If $M,N$ are complexes, show that there is a natural complex of $\Z$-modules $\hom_A(M,N)$ which in degree $n$ consists of the morphisms of degree $n$, and whose differential is defined by $f \mapsto d_N f -(-1)^{\deg f}fd_M$. Show that the $n$-cocycles are precisely the degree $n$ chain maps.
    \item If $X$ is a complex and $f,g$ are two endomorphisms of $X$, their \textbf{graded commutator} is $[f,g]=fg - (-1)^{\deg(f)\cdot\deg(g)}gf$. Show that the graded commutator $[d_X,-]$ makes the graded abelian group $\enn_A(X)$ into a complex.
    \item Show that the obvious inclusion $\hom_A(M,N)\to \enn_A(M\oplus N)$ is a chain map. This explains the sign in the differential of the hom-complex.
    \item*If $B$ is another ring, $Y$ a complex of right $B$-modules and $Z$ a complex of left $B$-modules, show that the graded abelian group $Y\otimes_BZ$ which in degree $i$ is given by $\bigoplus_{i=p+q}Y^p\otimes_B Z^q$ is a complex under the differential $d(y\otimes z) = dy\otimes z + (-1)^{\deg(y)}y\otimes dz$. If $M$ is a $B$-$A$-bimodule and $N$ an $A$-module, show that the evaluation pairing $\hom(M,N)\otimes_B M \to N$ is a chain map.
    \item *Formulate and prove a hom-tensor adjunction for complexes.
\end{enumerate}
\end{exer}

\begin{exer}
    Compute the cohomology of the two-term complex $$0\to A \xrightarrow{a\cdot}A\to 0.$$ 
\end{exer}

If $f:M \to N$ is a chain map, it induces maps $H^if: H^iM \to H^iN$ on cohomology groups. A chain map $f$ is a \textbf{quasi-isomorphism} if each $H^if$ is an isomorphism. Two chain complexes are \textbf{quasi-isomorphic} if there is a (finite length) zig-zag of quasi-isomorphisms between them.
\begin{exer}
    Let $M$ be the complex of abelian groups $\Z \xrightarrow{n}\Z$, with the rightmost $\Z$ placed in degree zero. Show that the projection $M \to \Z/n$ is a quasi-isomorphism.
\end{exer}

\begin{exer}
    If $k$ is a field, and $M$ is a complex of $k$-vector spaces, show that $M$ is quasi-isomorphic to $H(M)$. (Hint: if $U\into V$ is a subspace, then $U$ admits a complement $U^\perp \into V$ such that $U\oplus U^\perp = V$.)
\end{exer}

\begin{exer}
    Let $A$ be the ring $\mathbb{C}[x,y]$, and consider the two complexes $$M = A\oplus A \xrightarrow{(x,y)}A$$ and $$M' = A\xrightarrow{x\mapsto 0, \, y\mapsto 0} \mathbb{C}$$Show that $H(M)\cong H(M')$. Show that there is no ($A$-linear!) quasi-isomorphism $M\to M'$. *Show that $M$ and $M'$ are not quasi-isomorphic.
\end{exer}

A complex $M$ is \textbf{exact at $n$}\, if $H^n(M)\cong 0$. A complex is \textbf{acyclic} (or a \textbf{long exact sequence}) if it is quasi-isomorphic to the zero complex. 
\begin{exer}
    If $M$ is a complex, show that $M$ is acyclic if and only if $M$ is exact at all $i\in \Z$.
\end{exer}

\begin{exer}\hfill
    \begin{itemize}
        \item When is a two-term complex $0\to M_0 \xrightarrow{f}M_1\to 0$ acyclic?
        \item When is a three-term complex $0\to M_0 \xrightarrow{f}M_1\xrightarrow{g}M_2\to 0$ acyclic?
    \end{itemize}
\end{exer}
If $f:M \to N$ is a chain map, the \textbf{mapping cone} $\mathrm{cone}(f)$ is the complex with $\mathrm{cone}(f)_i = M_{i+1} \oplus N_i$ and differential given by the upper-triangular matrix $\stbtm{d_M}{f}{0}{-d_N}$.
\begin{exer}[Mapping cones]
    If $f:M \to N$ is a chain map, show that there are natural chain maps $N \to \mathrm{cone}(f) \to M[1]$. Show that $f$ together with these maps induces a long exact sequence $$\cdots \to H^i(M) \xrightarrow{H^i(f)} H^i(N) \to H^i(\mathrm{cone}(f)) \to H^{i+1}(M)\to\cdots$$ Deduce that $f$ is a quasi-isomorphism if and only if $\mathrm{cone}(f)$ is acyclic.
\end{exer}

\begin{exer}\label[exer]{tStructureExer1}
    Let $M= 0 \to M_n \to \cdots \to M_m \to 0$ be a complex with finitely many nonzero terms. Show that $M$ can be obtained from the finite set $\{M^i: 0<i<m\}$ via a finite sequence of shifts and mapping cones.
\end{exer}

A complex $M$ is \textbf{strictly bounded above} if $M^i\cong 0$ for all $i\gg 0$ and \textbf{strictly bounded below} if $M^i\cong 0$ for all $i\ll 0$. A complex $M$ is \textbf{strictly bounded} if $M^i\cong 0$ for all but finitely many $i$; this is equivalent to being strictly bounded above and strictly bounded below. We say that $M$ is \textbf{bounded} if the complex $H(M)$ is bounded, and similarly for above/below.

Synonyms for bounded above in the literature include \textbf{right bounded} and \textbf{eventually connective}; bounded below is also referred to as \textbf{left bounded} or \textbf{eventually connective}. Sometimes we will use \textbf{cohomologically bounded} as a synonym for bounded, for emphasis.

If $A$ is a ring then we let $\mathbf{Ch}(A)$ denote the category of chain complexes of $A$-modules, with morphisms given by the chain maps. 
\begin{defn}
    Let $A$ be a ring. The \textbf{derived category} of $A$ is the localisation $D(A)\coloneqq \mathbf{Ch}(A)[\text{quasi-iso}^{-1}]$ given by formally inverting $\mathbf{Ch}(A)$ at the quasi-isomorphisms. This has subcategories $D^+(A)$, $D^-(A)$, $D^b(A)$ on those complexes which are bounded below, bounded above, and bounded, respectively. If $A$ is noetherian, we write $D(\mathbf{mod-}A)$ for the full subcategory of $D(A)$ on the bounded complexes of finitely generated $A$-modules.
\end{defn}

\begin{exer}
    Show that $D(A)$ is also the quotient $\mathbf{Ch}(A)/(\text{acyclic modules})$ where we identify all acyclic modules with zero. (Hint: mapping cones.)
\end{exer}
For the rest of this section, we will try to understand $D(A)$. In generality, this is often too difficult, and we restrict ourselves to $D^b(A)$ instead.

\section{Projective resolutions and chain homotopies}
Recall that an $A$-module is \textbf{projective} if it is a summand of a free module.

\begin{exer}
    Let $P$ be a module. Show that $P$ is projective if and only if for all maps $f:P\to N$ and surjections $\pi:M\onto N$, there exists a lift of $f$ through $\pi$, i.e.\ a map $\tilde f : P \to M$ such that $\pi\tilde f = f$.
\end{exer}
\begin{exer}
    If $p,q$ are distinct primes, show that $\Z/p$ is a projective $\Z/(pq)$-module.
\end{exer}

\begin{exer}
A \textbf{short exact sequence} is an exact complex of the form $0\to M_1 \to M_2 \to M_3 \to 0$. Let $P$ be a projective module. Show that the functor $\hom_A(P,-):\mathbf{mod-}A \to \mathbf{Ab}$ is \textbf{exact}; i.e.\ preserves short exact sequences.
\end{exer}

If $M$ is a module, a \textbf{projective resolution} of $M$ is a complex $$P \quad=\quad \cdots \to P_{2} \to P_1 \to P_0$$ of projectives together with a quasi-isomorphism $P\xrightarrow{\simeq }M$.
\begin{exer}
    If $P \quad=\quad \cdots \to P_{2} \to P_1 \to P_0$ is a complex of projectives, show that $P$ is a projective resolution of $M$ if and only if $H^i(P)\cong 0$ for $i\neq 0$, and $H^0(P)\cong M$.
\end{exer}

\begin{prop}
  Projective resolutions exist.  
\end{prop}
\begin{proof}
    Given an arbitrary module $M$, we can find a free module $P_0$ and a surjection $P_0\onto M$; for example, take $P_0$ to be the free module on the set $M$. This yields a short exact sequence $$0 \to K_0 \to P_0 \to M \to 0$$where $K_0$ denotes the kernel. Find a free module $P_1$ and a surjection $P_1 \onto K_0$; by composition this yields an exact sequence $$0 \to K_1 \to P_1 \to P_0 \to M \to 0$$where $K_1$ denotes the kernel of $P_1 \to K_0$. Continuing inductively produces the desired resolution.
\end{proof}

\begin{exer}
    If $A$ is noetherian and $M$ is a finitely generated module, show that $M$ has a resolution by finitely generated projectives.
\end{exer}

\begin{prop}
   Let $f:M\to M'$ be a map of modules. If $P\to M$ and $P' \to M'$ are projective resolutions, then $f$ lifts to a map $P \to P'$.
\end{prop}
\begin{proof}
    Iteratively use the lifting property of projective modules.
\end{proof}
\begin{exer}[Truncations]\label[exer]{truncExer}
    If $M$ is a right bounded complex, show that there is a strictly right bounded complex $M'$ with a quasi-isomorphism $M\simeq M'$. Show that if $M$ is a complex of projectives, then one can take $M'$ to be a complex of projectives. (Hint: if $n$ is the largest integer for which $H^n(M)$ is nonzero, then one can arrange for $M'$ to agree with $M$ in degrees $<n$ and be zero in degrees $>n$.)
\end{exer}

If $M$ is a right bounded complex, a \textbf{projective resolution} of $M$ is a right bounded complex of projectives $P$ with a quasi-isomorphism $P \to M$.
\begin{prop}
    Projective resolutions of right bounded complexes exist.
\end{prop}
\begin{proof}
The construction we give is known as the \textbf{Cartan--Eilenberg resolution}. By \Cref{truncExer} we may assume that $M_i\cong 0$ for $i\gg 0$. First resolve each $M_n$ individually to obtain a projective resolution $P_n$. The differentials in $M$ lift to maps $P_n \to P_{n+1}$ and this yields a double complex $P$. Totalising $P$ gives the desired resolution.
\end{proof}
\begin{rmk}
    One can also use this construction to give, for every complex $M$, a complex of projectives $P$ with a quasi-isomorphism $P\to M$. In general, such a complex $P$ is not `nice enough' from a homotopical viewpoint - for evidence towards this see \Cref{unbExer} below. The correct generalisation of `right bounded complex of projectives' to the unbounded setting is given by the concept of \textbf{h-projective} complexes, which we will not discuss further.
\end{rmk}

If $f,g:M \to N$ are two maps of degree $n$, a \textbf{chain homotopy} from $f$ to $g$ is a degree $n-1$ map $h:M \to N$ such that $\partial h = f-g$, where $\partial h$ denotes the differential of $h$ in the mapping complex $\hom_A(M,N)$. We write $f\simeq g$.

Say that two complexes $M,N$ are \textbf{chain homotopy equivalent} if there exists a pair of maps $f:M\to N$ and $g:N \to M$ such that $fg\simeq \id_N$ and $gf\simeq \id_M$. A complex is \textbf{nullhomotopic} if it is chain homotopy equivalent to $0$.

\begin{exer}
    Show that a complex $M$ is nullhomotopic if and only if there exists a degree $-1$ map $h:M \to M$ such that $\partial h = \id_M$.
\end{exer}

\begin{exer}\label[exer]{chqiEx}
    Show that if $M,N$ are chain homotopy equivalent complexes then they are quasi-isomorphic.
\end{exer}

\begin{prop}
Let $f:M \to N$ be a map of modules. If $g,g':P \to Q$ are two lifts of $f$ to projective resolutions, then $g$ and $g'$ are chain homotopic.
\end{prop}
\begin{proof}
    The homotopy is constructed iteratively using the lifting property.
\end{proof}
\begin{cor}\label[cor]{nhcor}
    If $P$ is a right bounded acyclic complex of projectives, then $P$ is nullhomotopic.
\end{cor}

\begin{exer}\label[exer]{unbExer}
    Let $A$ be the ring $k[\varepsilon] / \varepsilon^2$ of dual numbers. Let $M$ be the unbounded complex of projective modules $$\cdots\xrightarrow{\varepsilon}A\xrightarrow{\varepsilon}A\xrightarrow{\varepsilon}A\xrightarrow{\varepsilon}\cdots$$
    Show that $M$ is acyclic. Show that $M$ is not nullhomotopic. In particular, \Cref{nhcor} is false if one drops the right bounded hypothesis - essentially because one needs a rightmost module at which to begin the induction.
\end{exer}

\begin{prop}\label[prop]{rbprop}
    Let $P,\ Q$ be two right bounded complexes of projectives. A map $f:P\to Q$ is a quasi-isomorphism if and only if it is a chain homotopy equivalence.
\end{prop}
\begin{proof}By \Cref{chqiEx}, we only need to prove the forwards implication, so assume that $f$ is a quasi-isomorphism. Let $C$ be the cone of $f$, so that $C$ is acyclic. Hence \Cref{nhcor} yields a nullhomotopy of $C$. But a chain homotopy equivalence $C\simeq 0$ is the same thing as a chain homotopy equivalence $P\simeq Q$.
\end{proof}

\section{Homotopy categories and Ext}

\begin{defn}
    Let $A$ be a ring. The \textbf{homotopy category} of $A$ is the category $K^-(\mathbf{Proj-}A)$ whose objects are the strictly right bounded complexes of projectives. The morphisms are given by homotopy classes of chain maps. We may also replace $K^-$ by $K^b$ to consider only strictly bounded complexes, or $\mathbf{Proj}$ by $\mathbf{proj}$ to consider only those complexes of finitely generated projectives (which is only well-behaved when $A$ is noetherian).
\end{defn}
Observe that the morphisms in $K^-(\mathbf{Proj-}A)$ are given by taking $H^0$ of the corresponding mapping complexes.
\begin{thm}\label[thm]{KProjthm}
    There is an equivalence $K^-(\mathbf{Proj-}A) \to D^-(A)$ which is the identity on objects.
\end{thm}
\begin{proof}
If $\mathbf{Ch}^-(\mathbf{Proj-}A)$ denotes the category of strictly right bounded complexes of projectives (without taking homotopy classes), then there is an obvious functor $\mathbf{Ch}^-(\mathbf{Proj-}A) \to D(A)$ which sends an object to its equivalence class in the quotient. Since quasi-isomorphic objects in the source are chain homotopy equivalent, it descends to a functor $K^-(\mathbf{Proj-}A) \to D(A)$. Since every right bounded complex admits a projective resolution, and these can be taken to be strictly bounded by \Cref{truncExer}, this gives an essentially surjective functor $K^-(\mathbf{Proj-}A) \to D^-(A)$. Fullness and faithfulness of this functor both follow from \Cref{rbprop}.
\end{proof}
Similarly, we obtain an equivalence $K^-(\mathbf{proj-}A) \simeq D^-(\mathbf{mod-}A)$.
\begin{defn}
    If $X,Y$ are two objects of $D^-(A)$, then we write $$\ext_A^i(X,Y)\coloneqq \hom_{D(A)}(X,Y[i]).$$
\end{defn}
By \Cref{KProjthm}, we may compute Ext as follows. First find projective resolutions $P,Q$ of $X,Y$ respectively. We put $\R\hom_A(X,Y)\coloneqq \hom_A(P,Q)$, and it follows that $\ext_A^i(X,Y)\cong H^i\R\hom_A(X,Y)$. The complex $\R\hom_A(X,Y)$ is called the (total) \textbf{derived hom-complex} from $X$ to $Y$.

\begin{exer}
    Show that $\R\hom_A(X,Y)$ is well-defined as an object of $D(\Z)$, i.e.\ different choices of resolution lead to quasi-isomorphic derived hom complexes. *If $R$ is a commutative ring and $A$ is an $R$-algebra, show that $\R\hom_A(X,Y)$ is well-defined as an object of $D(R)$.
\end{exer}
\begin{exer}\label[exer]{exthalfexer}
Let $X,\ Y$ be two $A$-modules, and let $P\to X,\ Q\to Y$ be projective resolutions. Show that the natural map $\hom_A(P,Q) \to \hom_A(P,Y)$ is a quasi-isomorphism. Hence, one can compute Ext by only resolving in the first variable. Use this to show that if $M,N$ are $A$-modules, there is an isomorphism $\mathrm{Ext}^0_A(M,N)\cong \hom_A(M,N)$.
\end{exer}

\begin{defn}
    A complex is \textbf{perfect} if it is quasi-isomorphic to a strictly bounded complex of finitely generated projectives. We denote the category of perfect complexes by $\per(A) \into D(A)$.
\end{defn}
By \Cref{KProjthm}, there is an equivalence $K^b(\mathbf{proj-}A) \simeq \per(A)$.

\begin{exer}
    If $P$ is perfect and $X$ is bounded, show that $\R\hom_A(P,X)$ is bounded. Let $A$ be the dual numbers and let $M$ be the complex $$\cdots\xrightarrow{\varepsilon}A\xrightarrow{\varepsilon}A\xrightarrow{\varepsilon}A$$with the rightmost $A$ in degree zero. Show that $M$ is a projective resolution of $k$ and use this to compute $\ext^*_A(k,k)$. Deduce that $k$ is not a perfect $A$-module.
\end{exer}

\begin{exer}[Derived tensor products]
Let $X$ be a bounded above complex of right $A$-modules and $Y$ a bounded above complex of left $A$-modules. Let $P\to X$ and $Q\to Y$ be projective resolutions. Define the \textbf{derived tensor product} $X\lot_A Y$ to be the tensor product $P\otimes_A Q$. We write $\mathrm{Tor}^A_i(X,Y)\coloneqq H_i(X\lot_AY)$, using homological notation.
\begin{enumerate}
    \item Show that $X\lot_A Y$ is well-defined as an object of $D(\Z)$.
    \item *(Balancing of Tor). Let $M$ be a right $A$-module and $N$ a left $A$-module, with projective resolutions $P, \ Q$ respectively. Show that there are natural quasi-isomorphisms $M\otimes_A Q \from P\otimes_A Q \to P\otimes_A N$. (Hint: use the double complex $E^{pq} = P^p\otimes_A Q^q$). Hence, one may resolve in either (or both) variables to compute Tor.
    \item Show that if $M, \ N$ are modules, then there is a natural isomorphism $\mathrm{Tor}^A_0(M,N)\cong M\otimes_A N$. (Hint: resolve in only one variable).
    \item *Formulate and prove a hom-tensor adjunction for $\lot$ and $\R\hom$.
\end{enumerate}
\end{exer}

\section{Injective modules}\label[section]{injsect}

Dual to the notion of a projective module is an \textit{injective module}. In this section we formulate dual versions of our above results, culminating in a version of \Cref{KProjthm} for bounded below complexes of injective modules. Injective modules will be of use to us later in \Cref{Gorchapt} when we define Gorenstein rings; if desired the reader may skip this section until then. Injective objects more generally are of use in situations where we may not have enough projectives, for example in categories of sheaves arising in algebraic geometry.

\begin{defn}
    A module $I$ over a ring $A$ is \textbf{injective} if, whenever $I\into M$ is an injection, there exists a module $N\into M$ with $I\oplus N \cong M$.
\end{defn}
\begin{exer}
    Show that a module is injective if and only if it has the right lifting property with respect to injections: if $N\into M$ is an injection and $N\to I$ is any map then there exists an extension $M\to I$.
\end{exer}
\begin{exer}
    *Show that $\Q$ and $\Q/\Z$ are injective $\Z$-modules.
\end{exer}
\begin{exer}
    If $A$ is a semisimple ring, show that every $A$-module is injective.
\end{exer}
Just like every module has a projective resolution, every module $M$ has an \textbf{injective resolution}: a complex $I^0 \to I^1 \to I^2 \to \cdots$ which resolves $M$. Maps of modules lift to maps of injective resolutions, uniquely up to chain homotopy. One can construct Chevalley--Eilenberg type injective resolutions of bounded below complexes. By adapting the proof of \Cref{KProjthm}, one can show the following result:
\begin{thm}\label[thm]{kinjthm}
    Let $A$ be a ring. There is a triangle equivalence $$K^+(\mathrm{Inj}(A)) \simeq D^+(A)$$ which is the identity on objects.
\end{thm}
\begin{exer}
    Prove \Cref{kinjthm} by dualising the proof of \Cref{KProjthm}. You may assume the nontrivial fact that every module is a submodule of an injective module.
\end{exer}

\begin{exer}[Balancing of Ext]
*Let $X$ be a bounded above complex and $Y$ a bounded below complex. Let $P \to X$ be a projective resolution and $Y\to I$ be an injective resolution. Show that there is a natural quasi-isomorphism $\hom_A(P,Y) \to \hom_A(P,I)$. Deduce using \Cref{exthalfexer} that one can compute Ext by resolving in either variable separately, using projective or injective resolutions as appropriate.
\end{exer}

 	\chapter{Triangulated categories}
 	
 	The notion of a triangulated category is an axiomatisation of some of the properties that derived categories enjoy. As intimated above, triangulated categories will not be completely sufficient for our uses, so we only give a sketch of the ideas. Comprehensive discussions can be found in \cite{neemantxt, weibel,yekbook}.

    \section{Axioms}
 	
 	Let $k$ be a commutative ring. A $k$-linear \textbf{triangulated category} is a $k$-linear category $\C$ together with two extra pieces of data. The first piece of data is a linear autoequivalence $\Sigma$ of $\C$, which we call the suspension or the shift functor. A \textbf{triangle} in $\C$ is a sequence of three morphisms $$X \to Y \to Z \to \Sigma X$$ which we will frequently abbreviate by dropping the $\Sigma X$ term and letting the rightmost arrow point to nowhere. A \textbf{morphism} of triangles is a triple of morphisms which fits into the obvious commutative diagram. The second piece of data is a class of \textbf{exact} (or \textbf{distinguished}) triangles. The suspension and the shift should satisfy the following axioms:
 	
 	\begin{itemize}
 		\item TR0: Exact triangles are closed under isomorphisms and under $\Sigma$.
 		\item TR1: The triangle $X \xrightarrow{\id}X \to 0 \to $ is exact. Every morphism $f:X \to Y$ has a \textbf{cone} $Z$, which fits into an exact triangle of the form $X \to Y \to Z \to$. We caution that the cone \textit{need not be functorial}.
 		\item TR2: One can rotate triangles: the triangle $X \to Y \to Z \to$ is exact if and only if $Y \to Z \to \Sigma X \xrightarrow{-}$ is, where one has to flip the sign on the indicated map. 
 		\item TR3: Given a morphism $f \to g$ in the arrow category (i.e.\ a commutative square from $f$ to $g$!) then there is an induced morphism $\mathrm{cone}(f) \to \mathrm{cone}(g)$ fitting into a morphism of exact triangles. This morphism \textit{need not be unique}.
 		\item TR4: the famous octahedral axiom. Loosely this encodes a version of the third isomorphism theorem, if one thinks of cones as homotopy cokernels.
 		\end{itemize}
        \begin{prop}
    $D(A)$ and $D^b(A)$ are triangulated categories.
\end{prop}
\begin{proof}[Proof idea]
    The suspension is the shift $[1]$. The exact triangles are precisely those triangles isomorphic to triangles of the form $X \to Y \to \mathrm{cone}(f) \to$.
\end{proof}
  The intuition is that an exact triangle behaves like a rolled-up long exact sequence. Indeed, in our main example $D(A)$, given a morphism $f:X\to Y$ the induced morphism $\mathrm{cone}(f) \to X[1]$ corresponds precisely to the connecting morphisms in the associated long exact sequence.
       
\begin{exer}
    Show that cones are unique up to isomorphism. *Show that this isomorphism need not be unique.
\end{exer}
If $f:X\to Y$ is a morphism, with cone $C$, we will often refer to $\Sigma^{-1}C$ as the \textbf{cocone} of $f$. Observe that this fits into the exact triangle $\Sigma^{-1}C\to X \xrightarrow{f}Y \to$.
\begin{exer}
    Show that $X \to Y \to Z \to$ is exact if and only if the rotated triangle $\Sigma^{-1}Z \xrightarrow{-} X \to Y \to$ is. (Hint: what happens if you rotate a triangle three times?)
\end{exer}
\begin{exer}\label[exer]{zeroExer}
    Show that any two consecutive compositions in an exact triangle are zero.
\end{exer}

  A \textbf{triangle functor} between triangulated categories is a $k$-linear functor that commutes with $\Sigma$ and sends exact triangles to exact triangles.
  \begin{exer}
      If $X$ is a fixed object of $D^b(A)$, show that the functor $$\R\hom_A(X,-): D(A) \to D(\Z)$$ is a triangle functor. Deduce that if $Y\to Z \to W \to$ is an exact triangle in $D(A)$ then there is a long exact sequence $$\cdots\to \ext^i_A(X,Y) \to  \ext^i_A(X,Z) \to  \ext^i_A(X,W) \to  \ext^{i+1}_A(X,Y)\to\cdots$$of abelian groups.
  \end{exer}We generalise the previous exercise. If $\mathcal{T}$ is a triangulated category, we write $\ext^i_{\mathcal{T}}(X,Y)\coloneqq \mathcal{T}(X,\Sigma^iY)$.
  \begin{prop}\label[prop]{TriCatLES}Let $\mathcal{T}$ be a triangulated category. If $X$ is any object of $\mathcal{T}$ and $Y \to Z \to W \to$ is an exact triangle, then there is a long exact sequence $$\cdots\to \ext^i_\mathcal{T}(X,Y) \to  \ext^i_\mathcal{T}(X,Z) \to  \ext^i_\mathcal{T}(X,W) \to  \ext^{i+1}_\mathcal{T}(X,Y)\to\cdots$$of abelian groups.
  \end{prop}
  \begin{proof}
      By rotating the triangle and applying shifts it suffices to check exactness at $\ext^0_\mathcal{T}(X,Z)\cong \mathcal{T}(X,Z)$. The composition is zero by \Cref{zeroExer}, and exactness follows from completing a commutative square $$\begin{tikzcd}
          X \ar[r]\ar[d]& 0\ar[d] \\ Z\ar[r] & W
      \end{tikzcd}$$to a morphism of exact triangles.
  \end{proof}
\begin{rmk}
    There is a similar long exact sequence involving the $\ext_\mathcal{T}^i(-,X)$ functors. It can be derived from the previous proposition using the fact that the opposite of a triangulated category is itself triangulated.
\end{rmk}

  \begin{rmk} Say that a triangulated category $\mathcal{T}$ \textbf{has functorial cones} if there is a functor $C:\mathbf{Ar}(\mathcal{T}) \to \mathcal{T}$ such that for each $f$, the object $C(f)$ is a cone of $f$. Then a triangulated category $\mathcal{T}$ has functorial cones if and only if $\mathcal{T}$ is semisimple abelian. This fact goes back to Verdier's thesis \cite[1.2.13]{verdierthesis}, but Greg Stevenson has given a modern proof \cite{nofuncones}. The loose idea is that having functorial cones actually forces $\mathcal{T}$ to have kernels and cokernels. Then the claim follows because monos and epis split.
  	\end{rmk}
  
  \begin{rmk}For homotopy theorists: we will later say that a triangulated category $\mathcal{T}$ has an \textbf{enhancement} if it is the homotopy category of a pretriangulated dg category (so, roughly, if one can coherently assign derived hom-complexes $\R\hom_{\mathcal{T}}(X,Y)\in D(k)$ to every $X,Y\in \mathcal{T}$). A more general notion than this is a \textbf{topological enhancement}, namely that $\mathcal{T}$ is the homotopy category of a stable $\infty$-category (pretriangulated dg categories are precisely the $k$-linear stable $\infty$-categories). The classical stable homotopy category admits a topological enhancement, but not an enhancement. Suppose that $\C$ is a stable $\infty$-category, so that the homotopy category $h\C$ is canonically triangulated. In $\C$, one can make a functorial choice of cone, giving a morphism $\mathrm{Fun}(\Delta^1,\C) \to \C$. One then obtains by taking homotopy categories a functor $h\mathrm{Fun}(\Delta^1,\C) \to h\C$. There is a comparison map $h\mathrm{Fun}(\Delta^1,\C)  \to \mathbf{Ar}(h\C)$, but it fails to be an equivalence, as spelled out in \cite{386369}. In particular, the `functorial cone' does not factor through a morphism $\mathbf{Ar}(h\C) \to \C$, and so this does not prove that every enhanceable triangulated category is actually abelian.
  	\end{rmk}

  \section{Sums and subcategories}

  \begin{exer}
      An \textbf{extension} of $A$ by $B$ is an object $X$ fitting into an exact triangle $B\to X \to A \to$. Show that extensions are classified by $\ext^1(A,B)$.
  \end{exer}
  
\begin{prop}
    Triangulated categories have finite biproducts.
\end{prop}
\begin{proof}Let $\mathcal{T}$ be a $k$-linear triangulated category. It has a zero object, so we show that $\mathcal{T}$ has binary biproducts. We define $X\oplus Y$ to be the unique object fitting into the exact triangle $X\to X\oplus Y \to Y \xrightarrow{0}$. Equivalently, $X\oplus Y$ is the unique extension corresponding to $0\in \ext^1_\mathcal{T}(Y,X)$. By completing the diagram$$\begin{tikzcd}
    X\ar[r]\ar[d,"\id"]&X\oplus Y \ar[r]& Y \ar[r,"0"]\ar[d]& {}\\
    X\ar[r,"\id"]&X\ar[r]&0\ar[r]& {}
\end{tikzcd}$$to a morphism of exact triangles, we see that $X$ is a retract of $X\oplus Y$. Similarly we see that $Y$ is also a retract. One can further extend this diagram to a diagram with exact rows of the form $$\begin{tikzcd}
0\ar[r]\ar[d]&Y\ar[r,"\id"] \ar[d]& Y \ar[r]\ar[d,"\id"]& {}\\
    X\ar[r]\ar[d,"\id"]&X\oplus Y\ar[d] \ar[r]& Y \ar[r,"0"]\ar[d]& {}\\
    X\ar[r,"\id"]\ar[d]&X\ar[r]\ar[d,"0"]&0\ar[r]\ar[d]& {}\\
     {}&{}&{}& {}
\end{tikzcd}$$and the Nine Lemma now shows that the middle vertical triangle is exact; in other words we have $X\oplus Y \cong Y \oplus X$.

If $Z$ is a third object, apply the long exact sequence of \Cref{TriCatLES} together with the above splittings to deduce the existence of an isomorphism $\mathcal{T}(X\oplus Y,Z)\cong\mathcal{T}(X,Z)\oplus\mathcal{T}(Y,Z)$. The fact that $\oplus$ is a coproduct now follows from the Yoneda lemma. The argument for products is similar.
\end{proof}
\begin{defn}
     If $\D$ is a triangulated category, then a \textbf{triangulated subcategory} $\C$ is a full subcategory which contains $0$ and is closed under shifts and cones. The restriction of the triangulated structure to $\C$ makes $\C$ into a triangulated category and the inclusion into a triangle functor. 
\end{defn}
\begin{exer}
 Let $\D$ be a triangulated category and $\C\ni 0$ a full subcategory. 
 \begin{enumerate}
     \item Show that if $\C$ is closed under cones, then it is a triangulated subcategory that is moreover closed under finite direct sums.
     \item Show that if $\C$ is closed under shifts, then it is a triangulated subcategory if and only if it is closed under extensions.
 
 \end{enumerate}
\end{exer}

  Let $\D$ be a triangulated category. A triangulated subcategory $\C$ of $\D$ is called \textbf{thick} (or \textbf{\'epaisse}) if it is closed under direct summands. If $S$ is a set of objects of $\mathcal{T}$, the \textbf{thick closure} of $S$ is the smallest thick subcategory of $\mathcal{T}$ containing $S$. We denote it by $\mathbf{thick}_\D(S)$ or just $\mathbf{thick}(S)$ when the context is clear. When $S=\{X\}$ then we write $\mathbf{thick}_\D(X)$ or just $\mathbf{thick}(X)$.
  \begin{prop}\label[prop]{thickexistence}
      $\mathbf{thick}(S)$ exists.
  \end{prop}
  \begin{proof}
  One sets $\mathbf{thick}_1(S)$ to be the full subcategory on those objects which are direct summands of objects of the form $\bigoplus^n_{i=1}\Sigma^{t_i}S_i$ with $S_i\in S$. One then inductively sets $\mathbf{thick}_{r+1}(S)$ to be the full subcategory given by extending objects of $\mathbf{thick}_1(S)$ by objects of $\mathbf{thick}_r(S)$. Then we take $\mathbf{thick}(S)$ to be then union of all the $\mathbf{thick}_r(S)$.
  \end{proof}
  \begin{exer}[\Cref{tStructureExer1}, derived version]\label[exer]{tStructureExer2}
  Let $A$ be a ring and $M$ a bounded complex. Show that $M\in \mathbf{thick}\{H^i(M): i\in \Z\}$. (Hint: induct on $i$.)
  \end{exer}

  \begin{ex}
      If $A$ is a ring, the subcategory of $\per(A)$ on the strictly bounded complexes of finitely generated free $A$-modules is a triangulated subcategory which is not in general thick, since it need not be closed under summands.
  \end{ex}

  \begin{prop}
      Let $A$ be a noetherian ring. Then there are equalities $\per(A) = \mathbf{thick}_{D(A)}(A) = \mathbf{thick}_{D^b(A)}(A)$.
  \end{prop}
  \begin{proof}
      The second equality is easy to see so we concentrate on the first. Certainly $\per(A)$ is a triangulated subcategory of $D(A)$, and moreover it is thick since summands of finitely generated projective modules are finitely generated projective. So we need only show that $\per(A) \subseteq \mathbf{thick}_{D(A)}(A)$. Since the latter is closed under sums and summands we have $\mathbf{proj-}A \subseteq  \mathbf{thick}_{D(A)}(A)$. But by \Cref{tStructureExer1} we are done.
  \end{proof}

  \begin{rmk}
      In more general settings one often \textit{defines} $\per(A)\coloneqq \mathbf{thick}_{D(A)}(A)$.
  \end{rmk}

    \section{Verdier quotients and $D_\mathrm{sg}$}

    Suppose that $F:\C \to \D$ is a triangle functor. The \textbf{kernel} of $F$ is the full subcategory of $\C$ on those objects $x$ such that $F(x)\cong 0$. We denote the kernel by $\ker(F)$.
    \begin{exer}
        Show that $\ker(F)$ is a thick triangulated subcategory of $\C$. 
    \end{exer}

\begin{defn}
    Let $\D$ be a triangulated category and $\C\into \D$ a triangulated subcategory. The \textbf{Verdier quotient} is the universal triangulated category $\D/\C$ equipped with a functor $\pi:\D\to\D/\C$ such that $\C\subseteq \ker(\pi)$. In other words, if $F:\D\to \D'$ is a triangle functor which kills $\C$, then $F$ factors through $\pi$.
\end{defn}
Verdier quotients are also referred to as \textbf{Verdier localisation}s, since killing an object $X$ is equivalent to inverting $0\to X$. More generally, inverting a morphism $f$ is equivalent to killing its cone.
\begin{prop}
    Verdier quotients exist. The kernel of the natural projection $\D \to \D/\C$ is precisely $\mathbf{thick}_\D(\C)$.
\end{prop}
In particular, if $\C$ is thick, then $\C$ is precisely the kernel of the projection to the Verdier quotient.  
\begin{proof}[Proof sketch]The rough idea is to formally adjoin inverses to morphisms whose cone lies in $\C$. We follow the construction given in \cite[\S2.1]{neemantxt}. The objects of $\D/\C$ will simply be the objects of $\D$. For two such objects $X,Y$, let $\alpha(X,Y)$ denote the set of roofs $X \xleftarrow{f} Z \xrightarrow{g} Y$ where $f$ has cone in $\C$. We think of such a roof as a `fraction' $g/f$. A morphism of roofs $R \to R'$ is simply a morphism $Z\to Z'$ making the obvious diagram commute. Declare that two such roofs $R, R'$ are equivalent if they are dominated by a common roof $R \from R'' \to R'$. Define $(\D/\C)(X,Y)$ to be the set of equivalence classes of roofs. Composition is given by homotopy pullback; a concrete model for the homotopy pullback of $Z\to Y \from Z'$ is given by the cocone of the induced map $Z\oplus Z' \to Y$. This makes $\D/\C$ into a category, with an obvious functor $\pi:\D\to \D/\C$ which sends $f:X \to Y$ to the roof $X\xleftarrow{\mathrm{id}}X \xrightarrow{f}Y$. Next, one shows that if the cone of $f$ is in $\C$, then this roof is inverse in $\C/\D$ to the roof $Y\xleftarrow{f}X \xrightarrow{\mathrm{id}}X$. In particular, $\pi$ inverts all morphisms whose cone is in $\C$, and hence kills $\C$. One then shows that $\D/\C$ inherits a triangulated structure from $\D$ making $\pi$ into a triangle functor. The universal property follows from the construction. To identify the kernel, one direction is clear since kernels are always thick. For the other direction, one shows that $X\to 0$ becomes an isomorphism in $\D/\C$ if and only if $X$ is a summand of an object of $\C$. 
\end{proof}

  \begin{ex}
 Let $K(A)$ denote the chain homotopy category, where the objects are the chain complexes, and the morphisms are the chain maps taken up to chain homotopy equivalence. This is a triangulated category in the usual way. Let $K_\mathrm{ac}(A)$ denote the subcategory of acyclic complexes; this is a thick subcategory. The Verdier quotient $K(A)/K_\mathrm{ac}(A)$ is precisely the derived category $D(A)$.
  \end{ex}
 	
 	We can finally define the singularity category.
 	
 	\begin{defn}
 	Let $A$ be a noetherian ring. Then the singularity category is the Verdier quotient $$D_\text{sg}(A)\coloneqq \frac{D^b(\mathbf{mod}\text{-}A) }{\per(A)}$$
 	\end{defn}
 
 The singularity category comes equipped with a natural projection map $D^b(\mathbf{mod}\text{-}A) \to D_\text{sg}(A)$ whose kernel is precisely $\per(A)$.
 	
 	\chapter{Notions of regularity}

    In this section, we investigate various notions of regularity for rings and relate these notions to the singularity category. We will restrict ourselves here to two-sided noetherian rings, since their dimension theory avoids some pathologies present in the general case. Our references here are \cite{lamlect, eisenbud, matsumura}. 

    \section{Global dimension}
    
 	Let $A$ be a ring. If $M$ is an $A$-module, the \textbf{projective dimension} of $M$ is the minimal length of a projective resolution of $M$, where the length means the number of nonzero modules. We denote it by $\mathbf{pd}_A(M)$ or just $\mathbf{pd}(M)$. Evidently, the modules of projective dimension 0 are precisely the projective modules.
    \begin{exer}Let $A$ be noetherian. Show that a finitely generated module $M$ has finite projective dimension if and only if $M\in D(A)$ is a perfect complex.
    \end{exer}
     \begin{exer}
        If $0\to X \to Y \to Z\to 0$ is a short exact sequence, show that $\mathbf{pd}(Y) \leq \max(\mathbf{pd}(X),\mathbf{pd}(Z))$.
    \end{exer} 
    
    If $A$ is noetherian, the \textbf{global dimension} of $A$ is defined to be the supremum of the projective dimensions of all finitely generated $A$-modules. We denote the global dimension by $\mathbf{gldim}(A)$.

    \begin{exer}
        If $A$ is a principal ideal domain, show that $\mathbf{gldim}(A)\leq 1$.
    \end{exer}
    \begin{exer}
        *Show that a noetherian ring has infinite global dimension if and only if it has a module of infinite projective dimension.
    \end{exer}
    \begin{rmk}
        A priori, there is a left and a right notion of global dimension. However, for two-sided noetherian rings the two concepts agree, and we will use the two notions interchangeably.
    \end{rmk}

 	\begin{ex}
 		A ring $A$ has global dimension zero if and only if every module is projective. These are precisely the semisimple rings. A commutative semisimple ring is a finite direct product of fields.
 		\end{ex}
 	
 	\begin{ex}
 		If $A$ has global dimension $n$, then if $M,N$ are two finitely generated $A$-modules, we must have $\ext^i(M,N)\cong 0$ for $i>n$. In particular, when $k$ is a field the ring $k[x]/x^2$ must have infinite global dimension.
 		\end{ex}
 	
 	\begin{lem}
 		If $A$ has finite global dimension then $D_\mathrm{sg}(A)$ vanishes.
 		\end{lem}
 	\begin{proof}
 		Take a bounded complex $M=M_p \to \cdots \to M_q$ of finitely generated modules. By hypothesis, each $M_i$ has a bounded resolution $P_i$ by finitely generated projectives. Each differential $M_i \to M_{i+1}$ moreover lifts to a morphism $P_i \to P_{i+1}$ of complexes. It follows that $M$ is quasi-isomorphic to the totalisation of the double complex $P_p \to \cdots \to P_q$, which is clearly perfect. Hence $M$ is quasi-isomorphic to a perfect complex. So $\per(A)=D^b(A)$ and hence $D_\mathrm{sg}(A)$ vanishes.
 		\end{proof}

In the rest of this part we will look for a converse to the above lemma in the setting of commutative rings. To begin with, we will restrict ourselves to local rings.

        \section{Commutative rings I: the local setting}\label[section]{ABSsect}

If $R$ is a commutative ring, recall that the \textbf{Krull dimension} of $R$ is the supremum of the lengths of all chains of prime ideals in $R$. If $M$ is an $R$-module, we put $\dim(M)\coloneqq \dim(R/\mathbf{ann}(M))$. In particular, since taking quotients cannot increase the Krull dimension, we have $\dim(M)\leq \dim(R)$.

    If $(R,\mathfrak{m},k)$ is a commutative noetherian local ring with residue field $k$, then there is an inequality $\dim(R) \leq \dim_k(\mathfrak{m}/\mathfrak{m}^2)$, the dimension of the cotangent space of $R$. This is the same as the minimal number of generators of the ideal $\mathfrak{m}$, by Nakayama's Lemma. In particular, we can conclude that $R$ has finite Krull dimension.

    Say that $(R,\mathfrak{m},k)$ is \textbf{regular} if $\dim(R)=\dim_k(\mathfrak{m}/\mathfrak{m}^2)$. In other words, this means that the (co)tangent space of $R$ has the correct expected dimension.

 	\begin{thm}[``Auslander--Buchsbaum--Serre'']
 		Let $R$ be a commutative local noetherian ring. The following are equivalent:
 		\begin{enumerate}
 			\item $R$ is regular.
 			\item $\mathbf{gldim}(R)$ is finite. 
 			\item $D_\mathrm{sg}(R)$ vanishes.
 			\end{enumerate}
 		Moreover, if any of the above hold, the global dimension of $R$ is equal to its Krull dimension.
 		\end{thm}
        \begin{rmk}
             Here we make a historical remark. The `original' ABS theorem is $(1)\iff (2)$, which long predates the invention of singularity categories. The implication $(1)\implies (2)$ was first noticed by Buchsbaum, and the implication $(2)\implies (1)$ was independently proved by both Serre and Auslander--Buchsbaum. The equivalence of both statements with $(3)$ and the statement about the Krull dimension are in fact easy corollaries of the proof.
        \end{rmk}
Before we give a sketch proof of the ABS theorem, we will give an intermediate result about regular sequences and projective dimension. Recall that a \textbf{regular sequence} in a commutative ring $R$ is a sequence of elements $x_1,\ldots x_n$ such that $x_{i+1}$ is not a zerodivisor in the quotient $R/(x_1,\ldots, x_i)$ for $0<i<n$. Note that the $i=0$ case says that $x_1$ is not a zerodivisor on $R$.
        \begin{prop}\label[prop]{RegSeqProp}
    Let $R$ be a commutative noetherian local ring and $x_1,\ldots x_n$ a regular sequence contained in $\mathfrak{m}$. Then $\mathbf{pd}_R(R/(x_1,\ldots, x_n)) = n$.
\end{prop}
Before we give the proof we first recall some facts about the Koszul complex. If $R$ is a commutative ring and $f\in R$ then we let $K(R;f)$ denote the cochain complex of $R$-modules $R \xrightarrow{f} R$ concentrated in degrees $-1$ and $0$. If $f_1,\ldots, f_n$ is a sequence of elements, we put $$K(R;f_1,\ldots, f_n)\coloneqq K(R;f_1)\otimes_R\cdots \otimes_R K(R;f_n)$$ and call this the \textbf{Koszul complex} of the sequence $f_1,\ldots, f_n$.

\begin{exer}[Exterior power description of the Koszul complex]\label[exer]{KosExer}Take a rank $n$ free $R$-module $E=R^{\oplus n}$  and let $e_1,\ldots,e_n$ be the standard $R$-basis for $E$. Show that the Koszul complex $K(R;f_1,\ldots, f_n)$ is \textit{isomorphic} to the complex $$\wedge^nE \to \wedge^{n-1}E \to\cdots \to\wedge^1E\to \wedge^0E\cong R$$with differential induced by the map which sends $$e_{i_1}\wedge\cdots \wedge e_{i_j}\mapsto \sum_k (-1)^kf_{i_k}e_{i_1}\wedge\cdots\wedge \tilde e_{i_k} \wedge\cdots\wedge e_{i_j}$$where the tilde denotes omission of the corresponding factor.
\end{exer}
The key fact for us will be the following. We omit the proof.
\begin{thm}
    If $f_1,\ldots, f_n$ is a regular sequence then the Koszul complex $K(R;f_1,\ldots, f_n)$ is a resolution of the quotient $R/(f_1,\ldots, f_n)$.
\end{thm}

\begin{proof}[Proof of \Cref{RegSeqProp}]
    The Koszul resolution of $R/(f_1,\ldots, f_n)$ is a length $n$ projective resolution, so we see that $\mathbf{pd}_R(R/(x_1,\ldots, x_n)) \leq n$. It will suffice to prove that if $M$ is an $R$-module and $x\in \mathfrak{m}$ a nonzerodivisor on $M$ then we have an inequality $\mathbf{pd}(M/x)\geq 1+ \mathbf{pd}(M)$; we then obtain the desired inequality $\mathbf{pd}_R(R/(x_1,\ldots, x_n)) \geq n$ by induction. To prove this statement, first observe that since $x$ is not a zerodivisor on $M$ the sequence $M\xrightarrow{x}M\to M/x$ is exact. If $d$ denotes the projective dimension of $M$, then we can find another $R$-module $N$ such that $\mathrm{Ext}^d(M,N)$ is nonzero. Consider the associated long exact sequence $$\cdots\to \mathrm{Ext}^d(M,N)\xrightarrow{x}\mathrm{Ext}^d(M,N)\to \mathrm{Ext}^{d+1}(M/x,N)\to 0$$Nakayama's Lemma tells us that multiplication by $x$ on $\mathrm{Ext}^d(M,N)$ is not surjective, and so its cokernel $\mathrm{Ext}^{d+1}(M/x,N)$ does not vanish. Hence we have an inequality $\mathbf{pd}(M/x)> d$, as required.
\end{proof}
        
 	\begin{proof}[Sketch proof of Auslander--Buchsbaum--Serre]
    Let $\mathfrak{m}$ be the maximal ideal of $R$ and $k=R/\mathfrak{m}$ the residue field. The proof relies on the key equality $$\mathbf{gldim}(R)=\mathbf{pd}_R(k)$$
   which can be proved via an argument with Tor-dimension. We deduce that (2) and (3) are equivalent: we have already observed one direction of the proof, and the other follows since if $D_\mathrm{sg}(R)$ vanishes then certainly $k$ has finite projective dimension.

   Assume now that (1) holds, i.e. that $R$ is a regular local ring. Take a minimal set of generators $x_1,\ldots, x_d$ for the maximal ideal $\mathfrak{m}$. Because $R$ is regular, we have $d=\dim(R)$ by hypothesis. The $x_i$ in fact form a regular sequence on $R$, so by \Cref{RegSeqProp} we have $d=\mathbf{pd}_R(k)$. By the key equality, we see that both (2) and the statement about Krull dimension hold.
   
   We are left to show that (2) implies (1). This is the hard part of the proof; we omit the argument which, roughly, is an induction on $\mathbf{gldim}(R)$.
 		\end{proof}

        \begin{cor}
            If $R$ is a commutative noetherian regular local ring and $\mathfrak{p}$ is a prime ideal of $R$, then the localisation $R_\mathfrak{p}$ is also regular local.
        \end{cor}
        \begin{proof}
            Resolutions localise so we have $\mathbf{gldim}(R_\mathfrak{p})\leq \mathbf{gldim}(R)<\infty$.
        \end{proof}

    \section{Commutative rings II}

Now we move to the global setting.
 	
 	 	\begin{thm}[``Global Auslander--Buchsbaum--Serre'']
 		Let $R$ be a commutative noetherian ring. The following are equivalent:
 		\begin{enumerate}
 			\item The localisation $R_\mathfrak{m}$ is regular for every $\mathfrak{m}\in \mathrm{MaxSpec}(R)$.
            \item The localisation $R_\mathfrak{p}$ is regular for every $\mathfrak{p}\in \mathrm{Spec}(R)$.
 			\item Every finitely generated $R$-module has finite projective dimension.
            \item $D_\mathrm{sg}(R)$ vanishes.
 		\end{enumerate}
 		Moreover, if any of the above hold, the global dimension of $R$ is equal to its Krull dimension.
 	\end{thm}
    If $R$ satisfies any of the above equivalent conditions, we call $R$ \textbf{regular}.
 	\begin{proof}[Sketch proof of global ABS] The equivalence of (3) and (4) is clear. Since resolutions localise, if (4) holds then every  $D_\mathrm{sg}(R_\mathfrak{p})$ also vanishes, and hence (2) holds by ABS. Clearly (2) implies (1) so we only need to show that (1) implies (3). One first proves, by a compactness argument, that for every finitely generated $R$-module $M$ there exists a maximal ideal $\mathfrak{m}$ of $R$ such that $\mathbf{pd}_R(M) = \mathbf{pd}_{R_\mathfrak{m}}(M_\mathfrak{m})$. Hence if each $R_\mathfrak{m}$ is regular, then (3) holds by ABS again. The statement about Krull dimension follows from the equalities $$\mathbf{gldim}(R) = \sup_\mathfrak{m}\mathbf{gldim}(R_\mathfrak{m}) = \sup_\mathfrak{m}\dim(R_\mathfrak{m})=\dim(R)$$where in the second equality we are using ABS for one final time.
 		\end{proof}

 	Note that we have not proved that a commutative noetherian regular ring must have finite global dimension. In fact this is false! Nagata gave an example of a commutative noetherian regular ring $R$ with infinite Krull dimension (and hence, by global ABS, global dimension). Each localisation of $R$ must have finite - but arbitrarily large - global dimension. Although $D_\mathrm{sg}(R)$ vanishes, it does not vanish in a `uniform' way, in the sense that one cannot uniformly bound the projective dimension of all finitely generated modules.
 	
 	\begin{ex}[Nagata \cite{nagata}]
 		Let $I_n\subseteq \N$ denote the interval $[2^{n-1}, 2^n-1]$, which has length $2^{n-1}$. Let $A=\mathbb{C}[x_1,x_2,\cdots] $ be the infinite-dimensional polynomial ring and for each $n\in \N$ let $\mathfrak{p}_n$ denote the ideal generated by $\{x_i: i\in I_n\}$. Put $S\coloneqq A/\cup_n\mathfrak{p_n}$ and put $R\coloneqq A_S$ the localisation. Since $\mathfrak{p_n}$ has height $2^{n-1}$ in $A$ it follows that $R$ has infinite Krull dimension. To prove that it is regular, first use the Prime Avoidance Lemma to show that every maximal ideal of $R$ is of the form $\mathfrak{p}_nR$, so that we need to check that each $A_{\mathfrak{p}_n}$ is regular; this holds since it is a localisation of a regular ring. To prove that it is noetherian boils down to checking that each $A_{\mathfrak{p}_n}$ is noetherian; again this holds since it can be written as a localisation of a noetherian ring.
 		\end{ex}

        \begin{rmk}
            Suppose that $k$ is a field and $R$ is a finite type $k$-algebra. If $k$ is perfect (for example, characteristic zero, algebraically closed, or finite) then $R$ is regular if and only if $R$ is smooth over $k$. Without the perfectness assumption this is no longer true; for example, the variety $x^2-y^2=t$ defined over the field $\Z/2(t)$ is regular but not smooth.
        \end{rmk}

\chapter{Buchweitz's stable category}\label[chapter]{Gorchapt}

In this chapter we give an alternate description of the singularity category, which works for Gorenstein rings. The original (and very good) reference here is \cite{buchweitz}; see also the new and updated version \cite{buchnew}. If you are a representation theorist then see also \cite{beligiannis} for some significant generalisations.

\section{Gorenstein rings and depth}
Recall from \Cref{injsect} the notion of an injective module. Dually to projective dimension, one can define the \textbf{injective dimension} of a module $M$ as the minimal length of an injective resolution of $M$. We denote this number by $ \mathbf{id}_A(M)$ or just $\mathbf{id}(M)$ when relevant.
\begin{rmk}
    It is a nontrivial fact that $\mathbf{gldim}(A) = \sup_M \mathbf{id}_A(M)$, where the supremum is taken over \textbf{all} $A$-modules $M$.
\end{rmk}

\begin{defn}
    A two-sided noetherian ring $A$ is \textbf{Gorenstein} or \textbf{Iwanaga--Gorenstein} if the $A$-module $A$ has finite injective dimension over $A$, as both a left and a right module. 
\end{defn}
\begin{rmk}
    If $A$ is Gorenstein, then a theorem of Zaks states that the right injective dimension of $A$ must agree with the left injective dimension of $A$. In general, the injective dimension may be infinite on one side and finite on the other. See \cite{zaks} for further discussion.
\end{rmk}

In the remainder of this section, our goal is to give an alternate characterisation of commutative Gorenstein rings in terms of an invariant called depth. This will also be of use later when discussing matrix factorisations. For the rest of this section, let $(R,\mathfrak{m},k)$ be a commutative noetherian local ring. The \textbf{depth} of an $R$-module $M$ is the smallest number $i$ for which $\ext^i(k,M)$ is nonzero. Note that $\mathrm{depth}(M)$ may be infinite - clearly the zero module has $\mathrm{depth}(0)=\infty$.

\begin{ex}
    A module $M$ is depth zero precisely when there exists a nonzero map $k\to M$. This is equivalent to the existence of $x\in M$ with $x\mathfrak{m}=0$.
\end{ex}

\begin{thm}[Rees]
    If $M$ is finitely generated, then the depth of $M$ is the length of a maximal $M$-regular sequence $x_1,\ldots x_n$ with all $x_i\in \mathfrak{m}$.
\end{thm}
\begin{proof}[Proof idea]
This is by induction on the depth of $M$ - one shows that if $x\in \mathfrak{m}$ is a non-zerodivisor then $\mathrm{depth}(M/x) = \mathrm{depth}(M)-1$. This computation is similar to the proof of \Cref{RegSeqProp}.
\end{proof}
\begin{cor}
    If $M\neq 0$ is finitely generated, there is an inequality $$\mathrm{depth}(M)\leq \dim(M).$$
\end{cor}
\begin{proof}
    If $x$ is not a zerodivisor on $M$, we have $\dim(M/x) = \dim(M)-1$. An induction now shows that $\dim(M/(x_1,\ldots, x_n)) = \dim(M) - n$ for any $M$-regular sequence. The left hand side is at least zero, so we have $n\leq \dim(M)$. 
\end{proof}
Say that a module $M$ is \textbf{Cohen--Macaulay} (or just \textbf{CM}) if it satisfies  $\mathrm{depth}(M)=\dim(M)$. Say that $M$ is \textbf{maximal Cohen--Macaulay} (or \textbf{MCM}) if $\mathrm{depth}(M) = \dim(R)$. By convention, we also say that $0$ is an MCM module. Say that $R$ is \textbf{Cohen--Macaulay} if the $R$-module $R$ is CM (in which case it is necessarily MCM).
\begin{rmk}
    One can extend the above definitions to non-local rings by saying that $M$ is (M)CM whenever all of its localisations at primes are so.
\end{rmk}
\begin{ex}
    A regular local ring is CM, since in this case $\mathfrak{m}$ is generated by a regular sequence of length $\dim(R)$, which implies that $\mathrm{depth}(R)\geq \dim(R)$.
\end{ex}
\begin{ex}
    A commutative Artinian ring is CM, since it has Krull dimension zero. Every module is MCM.
\end{ex}
\begin{thm}[Auslander--Buchsbaum formula]
    If $M$ is a nonzero module of finite projective dimension, then there is an equality
    $$\mathrm{depth}(M)+ \mathbf{pd}_R(M) = \mathrm{depth}(R).$$
\end{thm}
\begin{proof}[Proof idea]
    Induct on the projective dimension of $M$. At the induction step one uses a characterisation of depth in terms of the Koszul complex for $\mathfrak{m}$.
\end{proof}
\begin{cor}
    Let $R$ be a Cohen--Macaulay ring and $M$ an MCM $R$-module. Then $M$ is either projective or has infinite projective dimension.
\end{cor}

\begin{ex}Let $R$ be a commutative Artinian ring. Then every $R$-module is either projective or has infinite projective dimension. For example, one can take $R=kG$ for $G$ a finite abelian group and $k$ a field; this fact is well known when $\mathrm{char}(k)$ does not divide the order of $G$, in which case $kG$ is semisimple and hence every module is projective.
\end{ex}

Now we are ready to link CM rings to Gorenstein rings:

\begin{thm}
    Let $R$ be a commutative noetherian local ring with residue field $k$, of Krull dimension $n$. The following are equivalent:
    \begin{enumerate}
        \item $R$ is Gorenstein.
        \item $R$ has injective dimension $n$.
        \item $R$ is CM and $\mathrm{Ext}_R^n(k,R)\simeq k$.
    \end{enumerate}
    \end{thm}
    We omit the proof, which uses an induction on the dimension and some facts about dualising modules. As an immediate corollary, we see that a commutative Gorenstein ring is CM.

    \begin{ex}
        A commutative noetherian ring $R$ is \textbf{lci} (local complete intersection) if for every prime $p$, the completion $\hat{R}_p$ is of the form $A/(a_1,\ldots,a_r)$ where $A$ is regular complete local and the $x_i$ are a regular sequence. The obvious class of examples of lci rings is given by the \textit{global} complete intersections, i.e. the rings of the form $R=k[x_1,\ldots,x_n]/(f_1,\ldots, f_r)$ where the $f_i$ form a regular sequence (which in this setting is equivalent to $\dim R = n-r$). In particular, hypersurfaces (the case $r=1$) are lci. All lci rings are Gorenstein: to see this, first note that being Gorenstein is a complete local property, so it suffices to check that a complete intersection is Gorenstein. This is a computation with the Koszul complex.
    \end{ex}
    \begin{rmk}
        In general there exist CM rings which are not Gorenstein, and Gorenstein rings which are not lci; examples can be found even in Artinian rings.
    \end{rmk}

We give a useful alternate characterisation of MCM modules:
	\begin{thm}
	    Let $R$ be a commutative Gorenstein local ring and $M$ a finitely generated $R$-module. Then the following are equivalent:
  \begin{enumerate}
      \item $M$ is MCM.
      \item There is a natural quasi-isomorphism $\R\hom_A(M,A)\simeq \hom_A(M,A)$.
      \item $\ext_A^i(M,A)\cong 0$ for $i>0$.
  \end{enumerate}
	\end{thm}
    \begin{proof}[Proof idea]
The equivalence of (2) and (3) is clear, so we focus on the equivalence of (1) and (2). For a local CM ring $R$ with finitely generated module $M$, local duality tells us that $M$ has maximal depth if and only if $\mathrm{Ext}^i(M,\omega)$ vanishes for all $i>0$. Here $\omega$ denotes a dualising module for $R$, which exists since $R$ is CM. (The \textbf{dualising complex} is then $\omega[\dim R]$). A Gorenstein ring is precisely a ring where $R$ is a dualising module for $R.$
    \end{proof}
In the noncommutative setting, we may use the second criterion as a \textit{definition} of what it means to be MCM:
\begin{defn}
     A module $M$ over a Gorenstein ring $A$ is \textbf{maximal Cohen--Macaulay} (or just \textbf{MCM} for short) if there is a natural quasi-isomorphism $\R\hom_A(M,A)\simeq \hom_A(M,A)$.
\end{defn}

        \section{The stable category}
 		\p Suppose from now on that $A$ is a Gorenstein ring. Let $M$ be a finitely generated module. A \textbf{syzygy} of $M$ is a module $\Omega M$ which fits into a short exact sequence of the form $$0 \to \Omega M \to P \to M \to 0$$where $P$ is a finitely generated projective. Observe that one can stitch together the syzygy exact sequences for all $\Omega^iM$ into a projective resolution of $M$.
\begin{exer}
    Show that if $\Omega M$ is a syzygy of $M$ then so is $Q\oplus \Omega M$ for any finitely generated projective module $Q$.
\end{exer} 
        \begin{lem}
            If $M$ is MCM then so is $\Omega M$.
        \end{lem}
        \begin{proof}
            The defining short exact sequence yields an exact triangle $$\R\hom_R(M, R) \to \R\hom_R(P,R) \to \R\hom_R(\Omega M,R) \to$$and examining the corresponding long exact sequence, one sees that $\Omega M$ must be MCM.
        \end{proof}

 		\p One can define the \textbf{stable category of MCM modules} 
        $\underline{\mathbf{MCM}}(A)$ to have objects the MCM $A$-modules, and morphisms given by $$\underline\hom(M,N)\coloneqq\frac{\hom(M,N)}{(\text{morphisms which factor through a projective module})}.$$ In particular, projective modules go to zero in $\underline{\mathbf{MCM}}(A)$. Observe that $\underline{\mathbf{MCM}}(A)$ is a subcategory of the larger stable category $\underline{\mathbf{mod}}(A)$, defined analogously.

\begin{lem}
    There is a well-defined functor $\Omega$ on $\underline{\mathbf{MCM}}(A)$ which sends an object to its syzygy.
\end{lem}
\begin{proof}
Schanuel's Lemma shows that any two syzygies $X,Y$ of a module $M$ are \textbf{stably equivalent}, in the sense that $X\oplus Q \cong Y\oplus Q'$ for some projective (or free) modules $Q,Q'$. Hence they define the same object in the stable category, and so $\Omega$ is well-defined. Functoriality is easy to see.
\end{proof}
        
        \begin{lem}
            There is a functor $\iota: \underline{\mathbf{MCM}}(A) \to D_\mathrm{sg}(A)$ which is the identity on objects. It sends the $\Omega$ functor to the inverse shift functor $[-1]$.
        \end{lem}
        \begin{proof}
            The projection ${\mathbf{MCM}}(A) \to D_\mathrm{sg}(A)$ clearly kills projective modules so factors uniquely through the stable category; $\iota$ is the factoring map. By definition of $\Omega$ there is an exact triangle $$\Omega M \to P \to M \to$$ in $D^b(A)$ which induces an exact triangle $$\iota\Omega M \to 0 \to \iota M \to$$ in $D_\mathrm{sg}(A)$. Hence the connecting map $\iota M \to (\iota \Omega M) [1]$ is an isomorphism.
        \end{proof}
 
 		\begin{thm}\label[thm]{buchthm}
 			$\iota$ is an equivalence. Hence $\underline{\mathbf{MCM}}(A)$ is a triangulated category, with shift functor the `inverse syzygy' $\Omega^{-1}$.
 			\end{thm} 
            The proof passes through an intermediate category $K_\mathrm{ac}(\mathrm{proj}A)$, the homotopy category of acyclic complexes of finitely generated projective modules.

\begin{prop}\label[prop]{buchprop}\hfill
    \begin{enumerate}
        \item There are functors $\Omega_i: K_\mathrm{ac}(\mathrm{proj}A) \to \underline{\mathbf{MCM}}(A)$ which send a complex $X$ to the cokernel of $d^i:X^{-i-1} \to X^{-i}$. We have $\Omega_i(X[j])\simeq \Omega_{i-j}X$. 
        \item There are triangle functors $\sigma_i: K_\mathrm{ac}(\mathrm{proj}A) \to D_\mathrm{sg}(A)$ which send a complex to its brutal truncation at $i$. In other words, $(\sigma_iX)^j$ is $X^j$ for $j\leq i$ and $0$ otherwise.
        \item We have $\iota \Omega_i (X) \cong \sigma_{-i}(X)[-i]$.
    \end{enumerate}
\end{prop}
\begin{proof}
First observe that we obviously have functors $\Omega_i: \mathrm{Ch}_\mathrm{ac}(\mathrm{proj}A) \to \underline{\mathbf{mod}}(A)$ which satisfy the required conditions. If $f:P \to Q$ is nullhomotopic then the map $\Omega_i f$ factors through a projective module, so that they descend to functors $\Omega_i: K_\mathrm{ac}(\mathrm{proj}A) \to \underline{\mathbf{mod}}(A)$. The compatibility with shifts is clear. To see that the images are MCM, observe that we have isomorphisms $$\ext^j_A(\Omega_iX,A)\cong \ext_A^{j+k}(\Omega_{i-k}X,A)$$ for all $j>0$ and $k\geq 0$. In particular we see by taking $k\gg 0$ that these Ext groups vanish, and hence $\Omega_iX$ is MCM. Claim (2) is easy to see, and claim (3) follows from the fact that for $X \in K_\mathrm{ac}(\mathrm{proj}A)$, the complex $\sigma_i(X)[i]$ is a projective resolution of $\Omega_{-i}X$.
\end{proof}
We remark that part (1) of the preceding proposition tells us that `high enough syzygies of an arbitrary module are MCM'.

\begin{prop}\label[prop]{sigma0}
    The functor $\sigma_0$ is a triangle equivalence.
\end{prop}
Before we prove this, let us show how this yields a proof of \Cref{buchthm}.

\begin{proof}[Proof of \Cref{buchthm}]
    By \Cref{sigma0} and \Cref{buchprop} it suffices to prove that $\Omega_0$ is an equivalence. It is faithful by \Cref{sigma0}, so it remains to check that $\Omega_0$ is full and essentially surjective. To do this we introduce the notion of the \textbf{complete resolution} of an MCM module $M$. This is glued together out of the data of
    \begin{enumerate}
        \item A projective resolution $P\to M$
        \item A projective resolution $Q \to M^\vee$
    \end{enumerate}
    where we denote $M^\vee \coloneqq \hom_A(M,A)$. Since $M$ is MCM, its dual $M^\vee$ is MCM, and moreover $M$ is reflexive: we have a natural isomorphism $M\to M^{\vee\vee}$. Dualising the projective resolution $Q$ hence yields a quasi-isomorphism $M \to Q^\vee$. Composing $P\to M \to Q^\vee$ hence yields a quasi-isomorphism. A \textbf{complete resolution} for $M$ is the corresponding acyclic complex $$\mathbf{CR}(M)\ \coloneqq\ \mathrm{cocone}(P\to Q^\vee)\ \in\   K_\mathrm{ac}(\mathrm{proj}A) $$We caution that this is merely notation, although by choosing functorial projective resolutions $\mathbf{CR}$ can be upgraded to a functor. Clearly we have an isomorphism $\Omega_0\mathbf{CR}(M)\simeq M$, and hence $\Omega_0$ is essentially surjective. To show fullness, let $f:M\to M'$ be a map between two MCM modules. By lifting both $f$ and $f^\vee$ to maps of projective resolutions, we can lift $f$ to a map of complete resolutions.
\end{proof}

\begin{proof}[Proof of \Cref{sigma0}]
We need to show that $\sigma_0$ is fully faithful and essentially surjective. For essential surjectivity, take an object $X$ of the singularity category, which we may assume is a bounded and strictly right bounded complex of finitely generated projectives. Observe that, for every $k$, the complexes $X$ and $\sigma_k X$ differ by a perfect complex and hence are isomorphic in the singularity category. For $k\ll0$, we see that $M\coloneqq H^{k}(\sigma_kX) \simeq \sigma_k X$ is MCM since high enough syzygies are MCM as in \Cref{buchprop}(1). Completing $\sigma_k X$ to a complete resolution of $M$ and then shifting yields an acyclic complex $Y$ with $\sigma_0Y\simeq \sigma_k Y \simeq \sigma_kX \simeq X$, as desired.

For fully faithfulness, we apply a theorem of Verdier stating that it is enough to show that if $P$ is a perfect complex and $X \in K_\mathrm{ac}(\mathrm{proj}A)$, then there exists a $k$ such that $\hom_{D^b(A)}(\sigma_kX,P)\simeq 0$. By taking shifts and cones we can reduce to the case that $P$ is a projective module concentrated in degree $-i$. Taking $k>i$, in this case, the claim amounts to checking that the Ext group $\ext^1_A(\Omega_{-i-1}X,P)$ vanishes. But this is the case since $P$ is finitely generated projective and the module $\Omega_{-i-1}X$ is MCM.
\end{proof}

         \section{Stable Ext}
Since it is a triangulated category, the stable category $\underline{\mathbf{MCM}}(A)$ admits a notion of Ext groups, which we denote by $\underline\ext$ and refer to as the \textbf{stable Ext groups}. One pleasing fact about them is the following:
 		 \begin{prop}\label[prop]{stabextprop}
 		 	Let $A$ be a Gorenstein ring and $M,N$ two MCM $R$-modules.
            \begin{enumerate}
                \item For $j>0$, there are natural isomorphisms $\underline{\ext}^j_A(M,N)\cong \ext_A^j(M,N)$.
                \item For $j<-1$ there are natural isomorphisms $\underline{\ext}^j_A(M,N)\cong \tor^A_{-j-1}(N,M^\vee)$.
                \item There is a four-term exact sequence
                $$0 \to\underline{\ext}_A^{-1}(M,N) \to N\otimes_A M^\vee \to \hom_A(M,N) \to \underline{\ext}^0(M,N)\to 0.$$
            \end{enumerate}
 		 	\end{prop}
            \begin{proof}
               In what follows we let $\mathbf{CR}(X)$ denote a complete resolution for $X$. Since $\underline{\mathbf{MCM}}(A)$ is equivalent to $K_\mathrm{ac}(\mathrm{proj} (A))$, we have isomorphisms $$\underline{\ext}^i(M,N)\cong H^i\hom_A( \mathbf{CR}(M),\mathbf{CR}(N))$$Let $P$ be a projective resolution for $N$ and $Q$ a projective resolution of $N^\vee$. The description of $\mathbf{CR}(N)$ as a mapping cocone of the natural morphism $P\to Q^\vee$ gives us an exact triangle $$\hom_A( \mathbf{CR}(M),\mathbf{CR}(N))\to\hom_A( \mathbf{CR}(M),P) \to\hom_A( \mathbf{CR}(M),Q^\vee)\to $$in $D(\Z)$. Since $A$ is Gorenstein, the complex $\hom_A( \mathbf{CR}(M),Q^\vee)$ is actually acyclic - morally this is since $Q^\vee$ is close to being an injective resolution of $N$. Hence we obtain a quasi-isomorphism$$\hom_A( \mathbf{CR}(M),\mathbf{CR}(N)) \simeq \hom_A( \mathbf{CR}(M),P)$$Now let $P'$ be a projective resolution of $M$ and $Q'$ a projective resolution of $M^\vee$. We obtain an exact triangle $$\hom_A(Q'^\vee,P) \to \hom_A(P',P)\to \hom_A(\mathbf{CR}(M),P) \to$$There is a natural map $$P\otimes_A Q'\cong P\otimes_A Q'^{\vee \vee}\to\hom_A(Q'^\vee,P)$$which is a quasi-isomorphism since $P$ and $Q'$ are bounded above complexes of finitely generated projectives. The exact triangle above becomes an exact triangle $$N\lot_A M^\vee \to \R\hom_A(M,N) \to \hom_A(\mathbf{CR}(M),N)\to$$which finally yields a long exact sequence
               $$\cdots \to \tor_{-i}^A(N,M^\vee) \to \ext^i_A(M,N) \to \underline{\ext}^i(M,N) \to \cdots$$which gives the desired statements.
            \end{proof}

         \begin{rmk}
             Buchweitz calls the stable Ext groups the \textbf{Tate cohomology} groups, since they generalise the following construction. Let $k$ be a field of characteristic zero and $G$ a finite group. If $M$ is a $G$-module then we can consider its group homology $H_\bullet(G,M)$ and its group cohomology $H^\bullet(G,M)$. There is a norm map $H_0(M,G) \to H^0(G,M)$ and the \textbf{Tate complex} $\hat{H}(G,M)$ is defined to be the cone of the induced map $H_\bullet(G,M) \to H^\bullet(G,M)$. In this case, the group ring $kG$ is Gorenstein, and we have $\hat{H}^i(G,M)\cong \underline{\mathrm{Ext}}^i_{kG}(k,M)$.
         \end{rmk}

 		\chapter{Matrix factorisations}
        Here we give our third main construction of the singularity category, which works for hypersurfaces. References here include \cite{dyck, symons, bjmat} or the original \cite{eisenbudper}. We will make use of the following fact from commutative algebra:
        \begin{prop}
            A projective module over a commutative local ring is free.
        \end{prop}

        \section{Periodicity}
        In this part, $k$ is a field. In what follows we will put $A\coloneqq k\llbracket x_1,\cdots,x_n \rrbracket$ and $R\coloneqq A/\sigma$ for some $\sigma \in\mathfrak{m}_A^2$. The condition that $\sigma$ has order two implies that the unique closed point of $R$ is a singular point. We call $R$ a \textbf{complete local hypersurface singularity}. Since $R$ is clearly lci it is a Gorenstein ring. 
        \begin{lem}\label[lem]{MFABFlem}
            An $R$-module $M$ is MCM if and only if $\mathrm{pd}_A(M)=1$.
        \end{lem}
        \begin{proof}
        Use the Auslander--Buchsbaum formula along with the fact that, since $R$ is lci, we have $\dim(A)=1+\dim(R)$.
        \end{proof}
        Eisenbud was the first to notice the following surprising fact:
 		\begin{thm}\label[thm]{perthm}
 			If $R$ is a complete local hypersurface singularity then $\Omega^2\cong \id$ as endofunctors of $\underline{\mathbf{MCM}}(R)$. 
 		\end{thm}
        In order to prove this, we introduce the main object of study of this part. 
         \begin{defn}
 	 	A \textbf{matrix factorisation} of $\sigma$ is a free finitely generated $\Z/2$-graded $A$-module $X$ together with an odd $A$-linear map $d:X \to X$ satisfying $d^2=\id_X\cdot \sigma$.
 	 	\end{defn}
        More concretely, a matrix factorisation is a pair $(F,G)$ of finitely generated free $A$-modules and two maps $\phi:F \to G$ and $\psi:G \to F$ such that $\phi\psi = \sigma\mathrm{id}_G $ and $\psi\phi=\sigma\mathrm{id}_F$.

  	\begin{ex}
  		For the nodal cubic $\sigma=y^2-x^2-x^3$ in the plane, one matrix factorisation is  $$\phi=\psi=\begin{pmatrix}y&x+x^2\\-x&-y\end{pmatrix}$$ 
  		\end{ex}

        If $M$ is an $A$-module, we put $\bar M\coloneqq M\otimes_A R \cong M/\sigma$.
        \begin{lem}\label[lem]{mfperlem}
            If $X$ is a matrix factorisation of $\sigma$, then the corresponding 2-periodic diagram
           $$\cdots \xrightarrow{\bar d_0}\bar X_0 \xrightarrow{\bar d_1}\bar X_1 \xrightarrow{\bar d_0}\bar X_0\xrightarrow{\bar d_1}\bar X_1 \xrightarrow{\bar d_0}\cdots$$
           is an acylic complex of free finite rank $R$-modules.
        \end{lem}
        \begin{proof}
            The diagram is a complex since $d^2 = \sigma$. It is obviously a complex of free finite rank $R$-modules so we are left to check the exactness. By symmetry it suffices to check the exactness at $\bar X_0$. To do this, let $a\in X_0$ be such that $\bar a \in \ker \bar d_0$. This is equivalent to asking that $d_0a = \sigma b$ for some $b\in X_1$. Hence we see that $$\sigma a = d_1d_0a = d_1\sigma b =d_1d_0d_1b = \sigma d_1b$$Since $\sigma$ is not a zerodivisor in $A$, it follows that $a=d_1b$, and hence $\bar a = \bar d_1 \bar b$, as required.
        \end{proof}

        \begin{proof}[Proof of \Cref{perthm}]
            It suffices to show that every MCM $R$-module $M$, without free summands, has a 2-periodic free resolution. By \Cref{MFABFlem}, $M$ admits a two-step $A$-free resolution $X_1\xrightarrow{d_0}X_0\to M$. Multiplication by $\sigma$ annihilates $M$, so is nullhomotopic on the resolution. This nullhomotopy is witnessed by a map $d_1:X_0 \to X_1$ which fits into a commutative diagram $$\begin{tikzcd}
 				X_1 \ar[r,"d_0"]\ar[d,"\sigma"]& X_0 \ar[d,"\sigma"]\ar[dl,"d_1"]
 				\\ X_1 \ar[r,"d_0"]& X_0
 				\end{tikzcd}$$and hence $(X_0,X_1, d_0, d_1)$ is a matrix factorisation of $\sigma$. By \Cref{mfperlem}, the complex
                 $$\cdots \xrightarrow{\bar d_0}\bar X_0 \xrightarrow{\bar d_1}\bar X_1 \xrightarrow{\bar d_0}\bar X_0$$ is a 2-periodic free resolution of its cohomology $\mathrm{coker}(d_0) = \bar M$. Since $A \to R$ is a quotient map, it is an epimorphism, and hence the natural map $R\otimes_A R \to R$ is an isomorphism. It follows that $\bar M \cong M$ as $R$-modules.
        \end{proof}

        \begin{rmk}
            If $A$ is a noetherian local CM ring, and $\sigma$ a non-zerodivisor in $\mathfrak{m}_A$, then the above proof shows that every MCM $A/\sigma$-module admits a 2-periodic resolution constructed using matrix factorisations. On the other hand, typically one wants $A$ to be a regular complete local $k$-algebra; the Cohen structure theorem then gives an isomorphism $A\simeq k\llbracket x_1,\ldots, x_n \rrbracket$.
        \end{rmk}

    \begin{rmk}
    The converse of \Cref{perthm} is also true: if $R$ is a complete local Gorenstein ring with $\Omega^p\cong \id$ for some $p\geq 1$ then $R$ is actually a hypersurface singularity. The proof uses a criterion of Gulliksen \cite{gulliksen}.
    \end{rmk}

    We finish with the useful observation that for a matrix factorisation $X$, the matrices $d_0$ and $d_1$ are necessarily square:
    \begin{lem}
        If $X$ is a matrix factorisation of $\sigma$, then $X_0 \cong X_1$ as $A$-modules.
    \end{lem}
    \begin{proof}
        Putting $X_0 \cong A^n$ and $X_1 \cong A^m$, we want to show that $n=m$. If $K$ denotes the fraction field of $A$, then base changing to $K$ gives us a pair of $K$-vector spaces $K^n, K^m$ and linear maps $K^n \to K^m$ and $K^m \to K^n$ such that the composite maps $K^n \to K^m \to K^n$ and $K^m \to K^n \to K^m$ are isomorphisms. This immediately implies that $n=m$, as desired.
    \end{proof}

     \section{The homotopy category}\label[section]{mfhcatsect}
  	The category of matrix factorisations admits a notion of mapping complex: if $X,Y$ are matrix factorisations we put $$\hom^i(X,Y)=\hom_A(X_0,Y_{i})\oplus \hom_A(X_1,Y_{1+i})$$where we take the indices mod $2$. The differential is given by the usual formula $\partial(f)=d_Yf - (-1)^{\vert f\vert}fd_X$.
    \begin{exer}
        Check that this is actually a differential.
    \end{exer}
    We let $\mathbf{MF}(A,\sigma)$ denote the \textbf{homotopy category of matrix factorisations}: the objects are the matrix factorisations of $\sigma$ and the morphisms are the homotopy classes of chain maps in the mapping complexes. Observe that $\mathbf{MF}(A,\sigma)$ has a \textbf{shift functor} $\Sigma$ which sends $(X_0,X_1)$ to $(X_1,X_0)$. Clearly $\Sigma^2=\mathrm{id}$.

  	\begin{thm}
  		There is an equivalence
  		$$\mathrm{C}:\mathbf{MF}(A,\sigma) \to \underline{\mathbf{MCM}}(R)$$which sends a matrix factorisation to $\mathrm{coker}(d_0)$. It sends $\Sigma$ to $\Omega = \Omega^{-1}$, and hence makes $\mathbf{MF}(A,\sigma)$ into a triangulated category.
  		\end{thm}
        \begin{proof}
                To construct the functor we can argue as follows. If $X$ is a matrix factorisation, we let $cX \in K_\mathrm{ac}(\mathrm{proj}R)$ be the acyclic complex $$\cdots \xrightarrow{\bar d_0}\bar X_0 \xrightarrow{\bar d_1}\bar X_1 \xrightarrow{\bar d_0}\bar X_0\xrightarrow{\bar d_1}\bar X_1 \xrightarrow{\bar d_0}\cdots$$The assignment $X\mapsto cX$ is functorial, since, as in the proof of \Cref{buchprop}, if two morphisms of matrix factorisations are homotopy equivalent then their difference factors through a projective module. The functor $\mathrm C$ is then the composition of $c$ with the equivalence $\Omega_0:K_\mathrm{ac}(\mathrm{proj}R) \to \underline{\mathbf{MCM}}(R)$. In particular, we see that $\mathrm{C}(X)\simeq \mathrm{coker}(d_0)$. It is easy to see that $c(\Sigma X) \simeq (cX)[1]\simeq (cX)[-1]$, and hence $\mathrm{C}(\Sigma X)\simeq \Omega\mathrm{C}(X)$. During the proof of \Cref{perthm}, we showed that $\mathrm{C}$ was essentially surjective, so to finish we need to show that it is fully faithful. 
                
                For fullness, let $f:M \to N$ be a map in the stable category. Find matrix factorisations $X,Y$ with $\mathrm{C}(X)\cong M$ and $\mathrm{C}(Y)\cong N$. Lifting $f$ to a map of resolutions $$\begin{tikzcd}
 				X_1 \ar[d,"\tilde f_1"]\ar[r,"d_0"] & X_0 \ar[d,"\tilde f_0"]
 				\\ Y_1 \ar[r,"d_0"] & Y_0
 				\end{tikzcd}$$defines a morphism $\tilde f:X \to Y$ of matrix factorisations which lifts $f$. 

                Faithfullness is the hardest part of the proof. For this we introduce some notation: if $X$ is a matrix factorisation, we let $\tilde{\mathrm{C}}(X)$ denote the complex $$\cdots \xrightarrow{\bar d_0}\bar X_0 \xrightarrow{\bar d_1}\bar X_1 \xrightarrow{\bar d_0}\bar X_0$$ which is a 2-periodic free resolution of the MCM module $\mathrm{C}(X)$. Take two maps $f,g:X \to Y$ with $\mathrm{C}f=\mathrm{C}g$. We want to prove that they are homotopic in $\mathbf{MF}(A,\sigma)$. Extending $f$ periodically gives a map $\tilde{\mathrm{C}}(f): \tilde{\mathrm{C}}(X) \to \tilde Y$. Since $\tilde{\mathrm{C}}(f)$ is a resolution of ${\mathrm{C}}(f)={\mathrm{C}}(g)$, we obtain a chain homotopy $h:\tilde{\mathrm{C}}(f)\simeq \tilde{\mathrm{C}}(g)$. If this chain homotopy were 2-periodic, we would be done, since it would lift to a homotopy of maps of matrix factorisations. In fact, we don't need full periodicity of $h$; writing $h=(h_1,h_2,h_3,\ldots)$ in components, it is enough that $h_3=h_1$. By modifying $h$ a little we can ensure this, and hence $(h_1,h_2)$ gives a homotopy $f\simeq g$.
        \end{proof}

\begin{rmk}
    A very general construction of the homotopy category of matrix factorisations is given in \cite{bjcat}.
\end{rmk}

   \section{Kn\"orrer periodicity}
   Here we let $k$ be algebraically closed and of characteristic not 2.
   \begin{thm}[{\cite{knoerrer}}]
       There is a triangle equivalence $$\mathbf{MF}(k\llbracket x_1,\ldots, x_n \rrbracket,\sigma) \simeq \mathbf{MF}(k\llbracket x_1,\ldots, x_{n},y,z \rrbracket,\sigma+y^2 + z^2)$$
   \end{thm}
   The proof will be quite hands-on. In what follows we will set $A=k\llbracket x_1,\ldots, x_n \rrbracket$ for brevity. We will write $A\langle y \rangle \coloneqq k\llbracket x_1,\ldots, x_n,y \rrbracket$ and similarly for $A\langle y,z\rangle$. We do not use the notation $A\llbracket y \rrbracket$ since, as is well known, the two rings $k\llbracket x,y \rrbracket$ and $k\llbracket x \rrbracket \llbracket y \rrbracket$ are \textit{not} isomorphic.
\begin{lem}\label[lem]{knlemsym}
    Suppose $X$ is a matrix factorisation such that $\Sigma X \simeq X$. Then $X$ is isomorphic to a matrix factorisation of the form $(\omega, \omega)$.
\end{lem}
\begin{proof}Let $(\alpha, \beta)$ be components of an isomorphism $X \to \Sigma X$. Then $(\beta\alpha,\alpha \beta)$ is an automorphism of $X$. Without loss of generality $X$ is indecomposable, so the maximal semisimple quotient of the endomorphism ring of $X$ is a copy of $k$. So put $(\beta\alpha,\alpha \beta) = (\mathrm{id}, \mathrm{id})+ (f_1, f_2)$ with $f_i$ in the radical of the endomorphism ring. It follows that $\alpha f_1 = f_2 \alpha$ and $\beta f_2 = f_1 \beta$. Let $p$ be a power series expansion for $(1+x)^{-\frac{1}{2}}$ and consider the new morphisms $\alpha'=\alpha p(f_1) = p(f_2)\alpha$ and $\beta' = \beta p(f_2) = p(f_2)\beta$. These new morphisms define a morphism $X \to \Sigma X$ with the property that $(\beta'\alpha',\alpha' \beta') = (\mathrm{id}, \mathrm{id})$. Since the endomorphisms of $X_0$ are a product of power series ring, we may take a square root $\gamma$ of $\alpha'$. Then we set $\omega = \gamma \psi \gamma = \gamma^{-1} \phi \gamma^{-1}$.
\end{proof}
If $X$ is a matrix factorisation of $\sigma$, we let $GX=(\omega, \omega)$ denote the matrix factorisation of $\sigma + y^2$ where $$\omega \coloneqq \tbtm{y}{\psi}{\phi}{-y}$$This notation means that $\omega$ is a block matrix, where $y$ denotes $y\mathrm{Id}$. Given a morphism $f:X \to Y$, we let $Gf:GX \to GY$ be the morphism with components $$\tbtm{f_0}{0}{0}{f_1} \ , \ \tbtm{f_0}{0}{0}{f_1}$$We omit the check that $G$ defines a functor. In fact it is a triangle functor; the natural isomorphism $G \Sigma\simeq \Sigma G = G$ is given by the matrix$$\tbtm{0}{i}{-i}{0}$$
We also let $\rho $ denote the restriction functor from factorisations of $\sigma + y^2$ to $\sigma$; concretely we have $\rho(X) = X/y$. Let $\tau$ be the involution of $A\langle y \rangle/(\sigma+y^2)$ that sends $y\mapsto -y$.

\begin{lem}\label[lem]{knlemlist}
\begin{enumerate}
    \item  $\rho \circ G \simeq \mathrm{id}\oplus \Sigma$.
    \item If $X$ is nontrivial and indecomposable, then $G\rho X \simeq X \oplus \tau X$.
    \item The image of $G$ consists of those matrix factorisations with $Y\simeq \tau Y$.
    \item If $Z$ is an indecomposable factorisation of $\sigma + y^2$ then $\Sigma Z \simeq \tau Z$.
\end{enumerate}
\end{lem}
We omit the proof (claim 3 relies on another power series trick).
\begin{cor}\label[cor]{kncor}
    Let $X$ be an indecomposable factorisation of $\sigma + y^2$. If we have $\Sigma X \simeq X$, then $X$ is in the image of $G$.
\end{cor}

\begin{prop}\label[prop]{knprop}\hfill 
\begin{enumerate}
    \item Let $X$ be an indecomposable factorisation of $\sigma$. Then $GX$ is decomposable if and only if $\Sigma X \simeq X$; in this case we have $GX \simeq Y \oplus \Sigma Y$ for some indecomposable $Y$ with $\Sigma Y \not\simeq Y$.
    \item Let $U$ be an indecomposable factorisation of $\sigma + y^2$. Then $\rho U$ is decomposable if and only if $\Sigma U \simeq U$; in this case we have $\rho U \simeq V \oplus \Sigma V$ for some indecomposable $V$ with $\Sigma V \not\simeq V$.
\end{enumerate}
\end{prop}
\begin{proof}
  For the first claim, suppose first that $\Sigma X \simeq X$. Then by \Cref{knlemsym} we may write $X\simeq (\omega, \omega)$. In this case we have $$GX \simeq (\omega + iy, \omega - iy)\oplus (\omega-iy, \omega + iy).$$
  We need to prove that $Y=(\omega + iy, \omega - iy)$ is indecomposable and $\Sigma Y \not \simeq Y$. The isomorphisms$$X\oplus \Sigma X \simeq \rho G X \simeq \rho Y \oplus \rho \Sigma Y\simeq \rho Y \oplus \Sigma \rho Y$$show that $\rho Y$ is indecomposable, and hence $Y$ must be indecomposable. Moreover, if $\Sigma Y \simeq Y$ then $Y$ must be in the image of $G$ by \Cref{kncor}, so writing $Y=GY'$ and applying $\rho$ to $GX\simeq GY'\oplus G\Sigma Y'$, we obtain an isomorphism $$X\oplus \Sigma X \simeq Y'\oplus \Sigma Y' \oplus \Sigma Y' \oplus Y' $$which shows that the left hand side has at least four nontrivial summands, which contradicts the indecomposability of $X$.
  
  Conversely, if $GX$ is decomposable, with a nontrivial indecomposable summand $Z$, then $\rho Z$ is a summand of $\rho G X\simeq X \oplus \Sigma X$. Hence it is either $X$ or $\Sigma X$. We also have $\Sigma \rho Z \simeq \rho\Sigma Z \simeq \rho\tau Z\simeq \rho Z$ and hence $X \simeq \Sigma X$. 

  The proof of the second claim is similar and uses the observation that if $\Sigma U \simeq U$ then $U$ is in the image of $G$, so $\rho U \simeq V\oplus \Sigma V$ by the above Lemma.
\end{proof}
Now we move on to the case where we add \textit{two} new quadratic terms. To begin the proof, we put $u=y+iz$ and $v=y-iz$ so that $y^2+z^2=uv$. If $X=(\phi,\psi)$ is a matrix factorisation of $\sigma$, denote by $HX$ the matrix factorisation $$\tbtm{u}{\psi}{\phi}{-v} \ , \ \tbtm{v}{\psi}{\phi}{-u}$$ of $\sigma + uv$. Given a morphism $f:X \to Y$, we let $Hf:HX \to HY$ be the morphism with components $$\tbtm{f_0}{0}{0}{f_1} \ , \ \tbtm{f_0}{0}{0}{f_1}$$We omit the check that $H$ defines a functor. We denote

$$G_1: \mathbf{MF}(A,\sigma) \to \mathbf{MF}(A\langle y \rangle,\sigma+ y^2)$$
$$G_2: \mathbf{MF}(A\langle y \rangle,\sigma+y^2) \to \mathbf{MF}(A\langle y, z \rangle,\sigma+ y^2 + z^2)\simeq \mathbf{MF}(A\langle u,v \rangle,\sigma+ uv)$$and $\rho_1, \rho_2$ similarly. The following are easy checks:

\begin{lem}\label[lem]{knlemeasy} \hfill
    \begin{enumerate}
        \item $G_2 G_1 \simeq H \oplus \Sigma H$.
        \item $\rho_1\rho_2 H \simeq \mathrm{id}\oplus \Sigma$.
        \item $\Sigma H \simeq H\Sigma$.
    \end{enumerate}
\end{lem}

         \begin{proof}[Proof of the main theorem]
         \Cref{knlemeasy}(3) shows that $H$ is a triangle functor, and \Cref{knlemeasy}(2) shows that it is faithful. We omit showing that $H$ is full - it is a standard argument about factoring maps through projective modules, combined with some power series manipulations. For essential surjectivity, it suffices to show that $H$ induces an bijection on isoclasses of indecomposable objects. To see this, let $X$ be an indecomposable factorisation of $\sigma$. Then $G_2G_1X \simeq HX \oplus \Sigma HX$ is decomposable, so by \Cref{knprop} it has exactly two indecomposable summands. Hence $HX$ is indecomposable. On the other hand, if $Y$ is an indecomposable factorisation of $\sigma + uv$, then $Y$ is a summand of $G_2\rho_2Y$ by \Cref{knlemlist}. Similarly, $\rho_2Y$ is a summand of $G_1\rho_1Y$, and it follows that $Y$ is a summand of $G_2G_1\rho_1\rho_2Y \simeq H\rho_1\rho_2Y \oplus H\Sigma \rho_1\rho_2Y$. In particular, $Y$ appears as an indecomposable summand of some $HX$. To finish, suppose that $X,X'$ are indecomposable factorisations of $\sigma$, with $HX\simeq HX'$. We want to show that $X\simeq X'$. Applying $\rho_1\rho_2$ to the isomorphism $HX\simeq HX'$ shows that $X\simeq X'$ or $X\simeq \Sigma X'$. We are done unless $X\simeq \Sigma X'$ and $\Sigma X \not \simeq X$. In this situation it follows that $G_1X$ is indecomposable by \Cref{knprop}. Now $G_2G_1X \simeq HX \oplus \Sigma HX$ has exactly two indecomposable non-isomorphic summands by \Cref{knprop} again, and hence $HX \not\simeq \Sigma HX$. But we have $HX \simeq HX' \simeq H\Sigma X \simeq \Sigma HX$, which is a contradiction.
         \end{proof}

         \begin{ex}
             A \textbf{simple singularity} is a hypersurface singularity of the form $g(x,y)+ z_1^2 + \cdots + z_n^2$, where $g$ defines a simple plane curve singularity (i.e. $g$ has a term of order two or three). The two-dimensional simple singularities are precisely the Kleinian singularities (see \Cref{kleinchapt} for more about these). Kn\"orrer periodicity shows that all simple singularities are singular equivalent to either their defining simple curve singularity or a Kleinian singularity.
         \end{ex}
         
         \begin{rmk}Over $\bbC$, Kalck has shown that \textit{all} singular equivalences between complete local isolated Gorenstein singularities of different Krull dimension come from Kn\"orrer periodicity \cite{kalckclassifying}.
         \end{rmk}

		\chapter{Differential graded categories} 	
        All of the triangulated categories we have considered thus far are in fact examples of dg categories, in that, roughly, they have homotopy coherent notions of mapping complexes. As we will later see, this allows us to extract quite fine information from them. In this chapter we set up the relevant dg category theory needed as background. Our references here are \cite{keller, toendglectures, drinfeldquotient}. In this chapter, $k$ will be a fixed commutative base ring. Unadorned tensor products, homs, etc.\ are to be taken over $k$.

        \section{First definitions}
        Recall that the category of complexes $\mathrm{Ch}(k)$ is a monoidal category under the \textbf{tensor product of complexes}: we have $(X\otimes Y)_i = \bigoplus_{p+q=i}X_p\otimes Y_q$ with differential defined analogously. The unit is the complex $k$ concentrated in degree zero.
  \begin{defn}
            A \textbf{dg algebra} is a monoid in $(\mathrm{Ch}(k),\otimes)$: in other words it is a chain complex $A$ together with a multiplication map $\mu: A\otimes A \to A$ and a unit map $e:k \to A$ satisfying associativity and unitality. We also ask that these be chain maps; for $\mu$ this gives the Leibniz rule and for $e$ this says that $d(1)=0$.
        \end{defn}
        \begin{exer}
            Let $A$ be a complex and $\mu:A \otimes A \to A$ a morphism. Write $\mu(x,y)=x.y$. Show that $\mu$ is a chain map precisely when it satisfies the \textbf{Leibniz rule} $d(x.y)=dx.y + (-1)^{\deg (x)}x.dy$. 
        \end{exer}
		\begin{defn}
			A $k$-linear \textbf{dg category} is a category $\mathcal{C}$ enriched over the monoidal category $(\mathrm{Ch}(k),\otimes)$ of complexes with the usual tensor product. In other words, to every pair of elements $(x,y)\in\mathcal{C}^2$ we assign a chain complex $\mathcal{C}(x,y)$, to every triple $(x,y,z)$ we assign a chain map $$\mu_{xyz}:\mathcal{C}(x,y)\otimes \mathcal{C}(y,z) \to \mathcal{C}(x,z)$$ satisfying associativity, and for every $x \in \mathcal{C}$ we assign a map $\eta_x: k \to \mathcal{C}(x,x)$ which is a unit with respect to composition.
		\end{defn}
        \begin{exer}
            Draw the diagram for associativity of the $\mu$. Use this to write down an equation expressing associativity. 
        \end{exer}
		Note in particular that for any object $x \in\mathcal{C}$, the complex $\mathcal{C}(x,x)$ naturally has the structure of a dg algebra under the multiplication $\mu_{xxx}$.
      
		\begin{defn}
			A \textbf{dg functor} $F:\mathcal{C}\to\mathcal{D}$ between two dg categories is a $\mathrm{Ch}(k)$-enriched functor; i.e.\ a map of objects $F:\mathrm{ob}(\mathcal{C})\to \mathrm{ob}(\mathcal{D})$ together with, for every pair $(x,y)\in\mathcal{C}^2$, a map of complexes $F_{xy}:\mathcal{C}(x,y) \to \mathcal{D}(Fx,Fy)$. We require that $F$ satisfies the associativity condition $$\mu_{Fx\ Fy \ Fz}\circ (F_{xy}\otimes F_{yz}) = F_{xz}\circ \mu_{xyz}$$ and the unitality condition $F_{xx}\circ \eta_x = \eta_{Fx}$.
		\end{defn}
		In particular, a dg functor $F:\mathcal{C}\to\mathcal{D}$ induces dg algebra morphisms $$F_{xx}:{\mathcal{C}}(x,x)\to{\mathcal{D}}(Fx,Fx)$$ for every $x \in \mathcal{C}$.
		\begin{ex}
			Examples of dg categories include:
			\begin{itemize}
				\item $\mathrm{Ch}(k)$ itself, with the usual notion of mapping complexes. More generally, if $A$ is a $k$-algebra then $\mathrm{Ch}(A)$ is a dg category.
				\item If $\C$ is a dg category then so is $\C^\text{op}$.
				\item If $X$ is a topological space with a sheaf of $k$-algebras $\mathcal{O}$, then the category $\mathrm{Ch}(\mathcal{O})$ of complexes of sheaves of $\mathcal{O}$-modules is a dg category.
				\item A dg category with one object $*$ is the same thing as a dg algebra $\mathrm{End}(*)$. The inclusion of dg algebras into dg categories is fully faithful.
				\end{itemize}
			\end{ex}	

              \begin{ex}
        DG categories can be thought of as \textit{many-object dg algebras}, in the following sense. A dg category $\C$ is equivalently the data of:
        \begin{itemize}
            \item A collection of objects.
            \item For each object $x$, a dg algebra $A_x = \C(x,x)$.
            \item For each pair of distinct objects $(x,y)$, an $A_x$-$A_y$-bimodule $M_{xy} = \C(x,y)$. 
            \item Bilinear composition maps between bimodules.
        \end{itemize}
          One can visualise this as an extremely large dg matrix algebra. For example, a dg category with exactly $n$ objects is the same thing as a dg algebra with exactly $n$ orthogonal idempotents. We caution that this description is not functorial, since dg functors need not be bijective on objects.
        \end{ex}

		\begin{defn}
			Let $\mathcal{C}$ be a dg category. The \textbf{homotopy category} of $\mathcal{C}$ is the $k$-linear category $H^0\mathcal{C}$ whose objects are the same as $\mathcal{C}$, and whose hom-spaces are given by ${(H^0\mathcal{C})}(x,y)\coloneqq H^0(\mathcal{C}(x,y))$. Composition is inherited from $\mathcal{C}$.
		\end{defn}
	\begin{ex}Let $k$ be $\Z$ or a field and let $A$ be a $k$-algebra. Then $H^0(\mathrm{Ch}_\mathrm{dg}(A))$ is the chain homotopy category.
		\end{ex}
		\begin{defn}Let $F:\mathcal{C}\to\mathcal{D}$ be a dg functor.\begin{itemize}
				\item $F$ is \textbf{quasi-fully faithful} if all of its components $F_{xy}$ are quasi-isomorphisms.
				\item $F$ is \textbf{quasi-essentially surjective} if the induced functor $H^0F$ is essentially surjective.
				\item $F$ is a \textbf{quasi-equivalence} if it is quasi-fully faithful and quasi-essentially surjective.
			\end{itemize}
            \end{defn}
             \begin{exer}
                Show that two dg algebras are quasi-equivalent as dg categories if and only if they are quasi-isomorphic as dg algebras.
            \end{exer}
            \begin{exer}Let $F$ be a dg functor. 
            \begin{enumerate}
                \item If $F$ is a quasi-equivalence, show that $H^0F$ is an equivalence.
                \item Suppose that $F$ is quasi-fully faithful. Show that $F$ is a quasi-equivalence if and only if $H^0F$ is an equivalence.
                \item Give an example of a functor $G$ such that $H^0G$ is an equivalence, but $G$ is not a quasi-equivalence.
            \end{enumerate}
            \end{exer}

    \section{Pretriangulated dg categories}
	\begin{defn}
	    If $\C$ is a dg category, its category of \textbf{right $\C$-modules} is the dg category $\Mod\C\coloneqq \mathbf{dgFun}(\C^\text{op},\mathrm{Ch}_\mathrm{dg}(k))$.
	\end{defn}
 A $\C$-module $M$ can be described in terms of the following data:
\begin{itemize}
    \item For every object $x$ of $\C$ we assign a complex $M(x)$.
    \item For every pair of objects $x,y$ of $\C$ we assign an action map $$ M(x)\otimes \C(x,y)\to M(y)$$These action maps should be compatible with the composition and unit maps of $\C$.
\end{itemize}
The category of left $\C$-modules is defined dually as dg functors on $\C$.
\begin{exer}
    If $A$ is a $k$-algebra (or more generally a dg-$k$-algebra), regarded as a one-object dg category, show that this recovers the usual definition of dg-$A$-module.
\end{exer}
 If $x$ is an object of $\C$, then we can define a left module $h^x\coloneqq \C(x,-)$ and a right module $h_x\coloneqq \C(-,x)$. The assignment $x\mapsto h_x$ is a (quasi-)fully faithful dg functor $\C \to\Mod\C$ which we call the \textbf{Yoneda embedding}.
 \begin{exer}
     If $A$ is a $k$-algebra, regarded as a one-object dg category, show that the Yoneda embedding sends the unique object of $A$ to the regular $A$-module $A_A$.
 \end{exer}

\p The category $\Mod\C$ has shifts and cones, defined pointwise: if $M$ is a module then $M[1]$ denotes the module $x\mapsto M(x)[1]$. Similarly, if $f:M \to N$ is a $\C$-linear map then we may define its cone by $\mathrm{cone}(f):x \mapsto \mathrm{cone}(f(x))$. These make the homotopy category $H^0(\Mod\C)$ into a triangulated category.
\begin{exer}
    *If $\C$ is a dg category, show that its arrow category $\mathbf{Ar}(\C)$ is also a dg category. Show that the mapping cone induces a dg functor $$\mathbf{Ar}(\Mod\C) \to \Mod\C.$$
\end{exer}
Recall that a module over $\C$ is \textbf{representable} if it is isomorphic to a module of the form $h_x$.
		\begin{defn}
		Say that a dg category is \textbf{strongly pretriangulated} if, for all morphisms $f:x \to y$, the following modules are representable:
		\begin{itemize}
			\item The zero module $0$.
			\item The shifts $h_x[n]$ for all $n$.
			\item The cone $\mathrm{cone}(h_f)$. 
			\end{itemize} 
		When these are representable, we define $x[n]$ to be the object representing $h_x[n]$ and $\mathrm{cone}(f)$ to be the object representing $\mathrm{cone}(h_f)$.
	\end{defn}
    If $\C$ is strongly pretriangulated, then its homotopy category $H^0(\C)$ is canonically triangulated. It is moreover a triangulated subcategory of $H^0(\Mod\C)$.
    \begin{ex}
    Clearly $\Mod\C$ itself is strongly pretriangulated.
    \end{ex}

	Every dg category $\C$ admits a \textbf{strongly pretriangulated envelope} $\mathrm{tri}(\C)$, defined to be the closure of the image of the Yoneda embedding under the zero module, cones and shifts. The construction is very similar to the construction of thick subcategories given in \Cref{thickexistence}. This comes with a natural Yoneda embedding $\C \into \mathrm{tri}(\C)$.
    \begin{ex}
        If $A$ is a ring, regarded as a one-object dg category, then $\mathrm{tri}(A)$ is the dg subcategory of $\Mod A$ on the strictly bounded complexes of finitely generated projective modules. The homotopy category $H^0(\mathrm{tri}(A))$ is precisely $\mathbf{per} A$.
    \end{ex}
    
    \begin{rmk}
         The assignment $\C \mapsto \mathrm{tri}(\C)$ is functorial, and the left adjoint to the inclusion of strongly pretriangulated dg categories in all dg categories. The unit $\C \to \mathrm{tri}(\C)$ is the Yoneda embedding.
    \end{rmk}
\begin{defn}
		A dg category $\C$ is \textbf{pretriangulated} if the natural map $C \into \mathrm{tri}(\C)$ is a quasi-equivalence.
		\end{defn}
    In particular, if $\C$ is pretriangulated then $H^0\C$ is canonically triangulated.

\begin{exer}
    	Let $F:\C \to \D$ be a functor between pretriangulated dg categories. \begin{enumerate}
    	   \item Show that $H^0(F)$ is a triangle functor.
            \item *Show that $F$ is a quasi-equivalence if and only if $H^0(F)$ is a triangle equivalence.
    	\end{enumerate}
\end{exer}

\begin{rmk}
  Bondal and Kapranov described a \textbf{totalisation functor} \\$\mathrm{Tot}:\mathrm{tri}(\mathrm{tri}(\C)) \to \mathrm{tri}(\C)$ which is an equivalence (and in particular a quasi-equivalence) of dg categories \cite{bkframed}.
\end{rmk}

\begin{rmk}
       Say that two dg categories $\C,\D$ are \textbf{derived Morita equivalent} if $\mathrm{tri}(\C)$ and $\mathrm{tri}(\D)$ are quasi-equivalent. Clearly if two dg categories are quasi-equivalent then they are derived Morita equivalent. If two abelian categories $\A,\B$ are Morita equivalent in the classical sense, then they are derived Morita equivalent, but the converse is not true - for example, $\mathrm{Coh}(\P^1_k)$ is derived Morita equivalent to the finite dimensional representations of the Kronecker quiver, but the two abelian categories are not Morita equivalent. Many important invariants of dg categories - for example, Hochschild co/homology, $K$-theory, or periodic cyclic homology - are derived Morita invariants.
\end{rmk}

    \begin{defn}
        If $\mathcal{T}$ is a triangulated category, then a pretriangulated dg category $\C$ is an \textbf{enhancement} of $\mathcal{T}$ if there is a triangle equivalence $H^0(\C)\simeq \mathcal{T}$.  
    \end{defn}  
    All of the triangulated categories we have seen thus far admit enhancements. For some of them, this is easy to see; for those constructed as Verdier quotients this is much less trivial. In the next section, we give a construction of a \textbf{dg quotient} of dg categories, and show that - in pretriangulated settings - it enhances the Verdier quotient.

\section{DG quotients}
Here we assume for simplicity that $k$ is a field. We say that an object $X$ of a dg category $\C$ is \textbf{contractible} if $\mathrm{id_X \in \C(X,X)}$ is a boundary, or, equivalently, if $\mathrm{id}_X=0 \in H^0(\C)(X,X)$. Note that this implies that $\C(X,X)\simeq 0$.
\begin{ex}
    Let $A$ be a dg algebra, regarded as a dg category with a single object $*$. Then $*$ is contractible if and only if $A\simeq 0$.
\end{ex}
\begin{ex}
    If $\C$ is a pretriangulated dg category, then an object $X$ of $\C$ is contractible if and only if $X\simeq 0$ in $H^0(\C)$.
\end{ex}

Let $\A$ be a dg category and $\B \into \A$ be a full dg subcategory. Say that a functor $F: \A \to \C$ \textbf{annihilates} $\B$ if it sends every object of $\B$ to a contractible object of $\C$.
\begin{defn}
    The \textbf{dg quotient} $\A/\B$ is the universal dg subcategory under $\A$ which annihilates $\B$.
\end{defn}
To spell out the universal property in more detail, we write $\mathbf{Hqe}$ for the homotopy category of dg categories (i.e.\ the category of dg categories localised at the quasi-equivalences). The universal property states that there is a morphism $\pi:\A \to \A/\B$ in $\mathbf{Hqe}$ such that, if $\A\to\C$ is a morphism in $\mathbf{Hqe}$ which annihilates $\B$, then there exists a map $\A/\B \to \C$ in $\mathbf{Hqe}$ which factors $\A\to\C$ through $\pi$.

\begin{thm}
DG quotients exist.    
\end{thm}
Before we give an idea of the proof, we make some historical remarks. The first construction of a dg quotient was given by Keller \cite{kellerdgquot}. Drinfeld \cite{drinfeldquotient} then came up with a simple construction, which is the one we give below. The universal property was later shown to hold by Tabuada \cite{tabquot}.
\begin{proof}[Proof idea]
We define a dg category $\A/\B$ as follows. The objects of $\A/\B$ are the same as the objects of $\A$. The hom-complexes of $\A/\B$ are obtained from those of $\A$ by freely adjoining, for every $b\in \B$, a contracting homotopy $h:b \to b$ with $dh=\mathrm{id}_b$. In other words, a morphism $x \to y$ in the quotient is given by a length $n+1$ path $f_n \varepsilon_nf_{n-1}\varepsilon_{n-1}\cdots \varepsilon_2f_1\varepsilon_1f_0$ where $$f_0:x \to u_1, \quad f_1:u_1 \to u_2, \quad \ldots \quad f_n: u_n \to u_{n+1}=y$$ with all $u_i \in \B$ for $1\leq i \leq n$. The differential is given by applying the Leibniz rule and using $d\varepsilon_i = \mathrm{id}_{b_i}$. For example, a path $f\varepsilon g$ of the form $x\to b \to y$ has differential $(df)\varepsilon g \ \pm\ fg \ \pm\ f\varepsilon(dg)$. This makes $\A/\B$ into a dg category.

The natural functor $\A \to \A/\B$ sends a morphism $f$ to the length one path $f:x \to y$. It annihilates $\B$, by consideration of the length two path $\mathrm{id}_b \varepsilon_b \mathrm{id}_b$ for $b\in \B$. We do not check the universal property.
\end{proof}

\begin{ex}
	Let $A$ be a dg algebra, regarded as a one-object dg category. This does not have many full subcategories: only the empty subcategory $\emptyset$ along with $A$ itself. The quotient $A/\emptyset$ is simply $A$ again. The quotient $A/A$ is the dg algebra $A\langle h \rangle$ with $dh=1$; in particular $A\langle h \rangle$ is acyclic and hence $A/A$ is quasi-equivalent to the zero dg category.
	\end{ex}

\begin{thm}[Drinfeld]
	If both $\B$ and $\A$ are pretriangulated then so is $\A/\B$. Moreover, there is a triangle equivalence $$H^0(\A/\B)\simeq H^0(\A)/H^0(\B)$$where the right hand side is the Verdier quotient.
	\end{thm}
    In other words, if $\mathcal{T}$ is a triangulated category, $\mathcal{T}'$ a full triangulated subcategory, and the inclusion $\mathcal{T}'\into \mathcal{T}$ is enhanceable, then the Verdier quotient $\mathcal{T}/\mathcal{T}'$ is also enhanceable.
\begin{proof}[Proof idea]
Shifts and cones in $\A/\B$ are inherited from $\A$, making $\A/\B$ pretriangulated. The dg functor $\A \to \A/\B$ yields a triangle functor $H^0(\A) \to H^0(\A/\B)$ which annihilates $\B$. Hence we obtain a triangle functor $H^0(\A)/H^0(\B) \to H^0(\A/\B)$ from the universal property of the Verdier quotient. It is essentially surjective, so we just need to check that it is fully faithful. For this, it suffices to check that the natural map $\mathrm{Ext}^i_{H^0(\A)/H^0(\B)}(X,Y) \to \mathrm{Ext}^i_{H^0(\A/\B)}(X,Y)$ is an isomorphism for all $i$ and for all $X,Y\in \A$. To do this, Drinfeld uses an alternate construction of the dg quotient defined using ind-categories.
\end{proof}

\begin{ex}
    Let $A$ be a ring. The pretriangulated dg category $\Mod A$ has a pretriangulated full dg subcategory $\mathrm{Acyc}(A)$ of acyclic complexes of modules. The dg quotient $\Mod A / \mathrm{Acyc}(A)$ is the \textbf{dg derived category} of $A$, denoted by $D_\mathrm{dg}(A)$. The homotopy category of $D_\mathrm{dg}(A)$ is the usual triangulated derived category $D(A)$ of $A$. One can then define the dg bounded derived category $D^b_\mathrm{dg}(A)$ to be the full subcategory on those objects with bounded cohomology, and the dg perfect derived category $\cat{per}_\mathrm{dg}(A)$ to be the full subcategory on those objects which represent perfect complexes. Similar constructions can be made when $A$ is a dg algebra or even a dg category.
\end{ex}

\begin{rmk}
There are other ways of constructing the dg derived category, but all of them yield quasi-equivalent results. Typically the dg derived category is constructed by taking the dg \textit{localisation} of $\Mod A$ at the quasi-isomorphisms. A simple alternate construction is given by taking the full dg subcategory of $\Mod A$ on the bifibrant objects in either the projective or injective model structures. In particular, a natural enhancement of $D^-(A)$ is the full dg subcategory of $A$ on those strictly right bounded complexes of projectives, just like in \Cref{KProjthm}.
	\end{rmk}

\section{DG singularity categories}
In this section, $k$ is a field and $R$ is a two-sided noetherian $k$-algebra.

\begin{defn}
	The \textbf{dg singularity category} of $R$ is the dg category $$D^\mathrm{dg}_\text{sg}(R)\coloneqq \frac{D^b_\mathrm{dg}(R) }{\cat{per}_\mathrm{dg}(R)}.$$
	\end{defn}
It is a pretriangulated dg category with $H^0\left(D^\mathrm{dg}_\text{sg}(R)\right)\simeq D_\mathrm{sg}(R)$.

\begin{thm}
    Suppose that $R$ is Gorenstein. Then both $\underline{\mathbf{MCM}}(R)$ and  $K_\mathrm{ac}(\mathrm{proj} (R))$ admit dg enhancements, and there is a commutative  diagram in $\mathbf{Hqe}$ $$\begin{tikzcd}
        &\underline{\mathbf{MCM}}^\mathrm{dg}(R)\ar[d,"\iota"]\ar[dl, "\mathbf{CR}", bend left=5] \\
        \mathrm{proj}_\mathrm{ac} (R)\ar[r, "\sigma_0"]\ar[ur, "\Omega_0", bend left=5]&D^\mathrm{dg}_\text{sg}(R)  
    \end{tikzcd}$$of quasi-equivalences of pretriangulated dg categories. In addition, if $R=A/\sigma$ is a complete local hypersurface singularity, then $\mathbf{MF}(A,\sigma)$ also admits a dg enhancement, and there is a commutative diagram in $\mathbf{Hqe}$
    $$\begin{tikzcd}
       \mathbf{MF}^\mathrm{dg}(A,\sigma)  \ar[d,"c"]\ar[r,"\mathrm{C}"] &\underline{\mathbf{MCM}}^\mathrm{dg}(R)\ar[d,"\iota"]\ar[dl, "\mathbf{CR}", bend left=5]  \\
        \mathrm{proj}_\mathrm{ac} (R)\ar[r, "\sigma_0"]\ar[ur, "\Omega_0", bend left=5]&D^\mathrm{dg}_\text{sg}(R)  
    \end{tikzcd}$$of quasi-equivalences of pretriangulated dg categories.
\end{thm}
\begin{proof}We have just defined the dg singularity category $D^\mathrm{dg}_\text{sg}(R)$. The category $K_\mathrm{ac}(\mathrm{proj} (R))$ is, by construction, the homotopy category of the dg category $\mathrm{proj}_\mathrm{ac} (R)$ of acyclic complexes of finitely generated projective $R$-modules. The mapping complexes in the dg category $\mathbf{MF}^\mathrm{dg}(A,\sigma)$ were described in \Cref{mfhcatsect}. Finally, we let $\underline{\mathbf{MCM}}^\mathrm{dg}(R)$ be defined as follows. The objects are the MCM $R$-modules. The hom-complexes are given by $$\underline{\mathbf{MCM}}^\mathrm{dg}(R)(M,N)\coloneqq \hom_R( \mathbf{CR}(M),\mathbf{CR}(N))$$where in defining $\mathbf{CR}$ we need to make a choice of functorial projective resolutions for MCM modules (which can be done using e.g.\ the small object argument). All of these dg categories are clearly pretriangulated and enhance the given triangulated categories.

To conclude, all that we need to do is show that the named triangle equivalences lift to morphisms in $\mathbf{Hqe}$. First recall that $c$ was the functor sending a matrix factorisation $X$ to the corresponding acyclic complex $$\cdots \xrightarrow{\bar d_0}\bar X_0 \xrightarrow{\bar d_1}\bar X_1 \xrightarrow{\bar d_0}\bar X_0\xrightarrow{\bar d_1}\bar X_1 \xrightarrow{\bar d_0}\cdots$$It is not hard to see that $c$ lifts to a dg functor: a pair of maps $X_0 \to Y_i$ and $X_1 \to Y_{i+1}$ clearly define a map $cX \to cY[i]$, and one can check that this induces a chain map on mapping complexes. Secondly, it is clear by construction that we have a complete resolution dg functor $$\mathbf{CR}: \underline{\mathbf{MCM}}^\mathrm{dg}(R) \to \mathrm{proj}_\mathrm{ac} (R)$$Since this is a quasi-equivalence, it has an inverse in $\mathbf{Hqe}$, which must lift $\Omega_0$. We now have the upper commutative triangle. For the lower commutative triangle, it suffices to observe that $\sigma_0$ lifts to a dg functor $\mathrm{proj}_\mathrm{ac} (R) \to D^b_\mathrm{dg}(R)$.

\end{proof}

\chapter{Compact generators}
Here we focus on some of the extra structure of matrix factorisation categories provided to us by the dg enhancement. Our main reference here is \cite{dyck}; see also \cite{eisenbudper}. For the remainder of this section we let $k$ be a field of characteristic zero, $A\coloneqq k\llbracket x_1,\ldots, x_n\rrbracket$ and $R\coloneqq A/\sigma$ a complete local hypersurface singularity. For technical reasons we will also need to require that $R$ is an \textbf{isolated singularity}, i.e. for all non-maximal $p\in\spec(R)$ the ring $R_p$ is regular.
\section{Stabilisation}
The quasi-equivalence $\mathbf{MF}^\mathrm{dg}(A,\sigma)  \simeq\underline{\mathbf{MCM}}^\mathrm{dg}(R)$ means that any MCM module $L$ functorially corresponds to a matrix factorisation $L^\mathrm{stab}$, which we call the \textbf{stabilisation} of $L$. 

We give a recipe due to Eisenbud for constructing $L^\mathrm{stab}$, in the special case that $L=A/I$ is cut out by a regular sequence $I=(f_1,\ldots, f_r)$ and $\sigma \in I$. For example, the residue field $k$ is a module of this form with $I=\mathfrak{m}_A=(x_1,\ldots, x_n)$.

As described in \Cref{KosExer}, consider the Koszul complex $K$ of $I$, which in degree $-n$ is given by $\wedge^n(A^{\oplus r})$ and whose differential is given by contraction:  
$$d(e_{i_1}\wedge\cdots \wedge e_{i_n})=\sum_j(-1)^j f_{i_j}e_{i_1}\wedge\cdots \wedge \tilde e_{i_j}\wedge\cdots \wedge e_{i_n}$$
		View $K$ as an $A$-free resolution of $L$, so that multiplication by $\sigma$ is nullhomotopic on $K$. In fact we can write down a nullhomotopy: put $\sigma=\sum_i \sigma_if_i$ and consider the element $\underline{\sigma}\coloneqq \sum_i\sigma_ie_i\in A^r$. Let $h:K \to K$ be the map of degree $-1$ defined by exterior multiplication on the left with $\underline\sigma$.
        \begin{lem}\label[lem]{mfdlem}
            $\partial h = \sigma$.
        \end{lem}
        \begin{proof}
        Let $x= e_{i_1}\wedge\cdots \wedge e_{i_n}$ be a basis element. We have $$h(x)=\sum_l \sigma_le_l\wedge x = -\sum_l \sigma_lx\wedge e_l$$so we compute 
        \begin{align*}
            d(hx) = &-\sum_j\sum_l(-1)^j \sigma_lf_{i_j}e_{i_1}\wedge\cdots \wedge \tilde e_{i_j}\wedge\cdots \wedge e_{i_n}\wedge e_l \quad + \quad \sum_l \sigma_l f_l x\\
            =& -h(dx) + x\sigma
        \end{align*}
        and hence $(\partial h)(x) = x\sigma$ as required.
        \end{proof}

        Placing $A^r$ in odd parity, we can view the Koszul complex $K$ (without its differential) as a $\Z/2$-graded free $A$-module $\wedge^*(A^r)$. Since $d$ is a differential, and $h^2(x) = {\underline\sigma}\wedge{\underline\sigma}\wedge x =0$, we have $(h+d)^2 = \partial(h)$. Hence by \Cref{mfdlem} we have $(h+d)^2 = \sigma$, so that $(\wedge^*(A^r),h+d)$ is a matrix factorisation of $\sigma$. We put $$L^\mathrm{stab}\coloneqq (\wedge^*(A^r),h+d)$$
        \begin{prop}
            The equivalence $\mathbf{MF}(A,\sigma)  \to \underline{\mathbf{MCM}}(R)$ sends $L^\mathrm{stab}$ to $L$.
        \end{prop}
        \begin{proof}
            Recall that the equivalence sends a matrix factorisation $X$ to the cokernel of $\bar{d_0}$. But since $K$ resolves $L$, this cokernel is $L/\sigma$, which by construction is the $R$-module $L$.
        \end{proof}

\section{Endomorphism rings}
Recall that an object $X$ of a triangulated category $\mathcal{T}$ is \textbf{compact} if $\mathcal{T}(X,-)$ commutes with all direct sums, and a \textbf{thick generator} if $\mathrm{thick}_\mathcal{T}(X) = \mathcal{T}$.
\begin{ex}
    $A$ is compact in $D(A)$. It is a thick generator for $\mathbf{per}(A)$.
\end{ex}
\begin{prop}[Keller]\label[prop]{tiltingprop}
    Let $\mathcal{C}$ be a pretriangulated dg category such that $H^0\C$ is idempotent complete. Let $X\in \C$ be a compact thick generator of $H^0\C$. Then there is a quasi-equivalence of pretriangulated dg categories
    $$\C \simeq \mathbf{per} \C(X,X)$$In particular, $\C$ and $\C(X,X)$ are derived Morita equivalent.
\end{prop}
		\begin{thm}[Dyckerhoff]
			The object $k^\mathrm{stab}$ is a compact thick generator for $\mathbf{MF}(A,\sigma)$.
			\end{thm}
The proof is a difficult piece of homological algebra and makes use of a homological version of the Nakayama lemma.	
	\begin{cor}
		There is a quasi-equivalence of $\Z/2$-graded dg categories $$\mathbf{MF}^\mathrm{dg}(A,\sigma)\simeq \per\enn(k^\mathrm{stab})$$In particular,  $\mathbf{MF}^\mathrm{dg}(A,\sigma)$ and $\enn(k^\mathrm{stab})$ are derived Morita equivalent.
		\end{cor}

        In particular, we are interested in \textit{computing} the $\Z/2$-graded dg algebra $\enn(k^\mathrm{stab})$. In order to do this we first give an alternate, more concise description of the stabilisation. Fix an MCM $R$-module $L$ as before. Observe that the $A$-module $\wedge^*(A^r)$ can be identified with the free supercommutative algebra $A[\theta_1,\ldots, \theta_r]$ with all $\theta_i$ in odd parity. Note that this is a finitely generated $A$-module, since $\theta_i\theta_j=-\theta_j\theta_i$. The differential $d$ becomes identified with the differential operator $\sum_{i=1}^rf_i\frac{\partial}{\partial \theta_i}$, and the nullhomotopy $h$ becomes identified with multiplication by $\sum_{i=1}^r\sigma_i\theta_i$. Hence, the total differential $h+d$ on $L^\mathrm{stab}$ is identified with the left action of the polynomial differential operator $$D=\sum_{i=1}^rf_i\frac{\partial}{\partial \theta_i}+ \sigma_i\theta_i.$$

        \begin{defn}
            We let $\mathrm{Poly}(r)$ denote the $A$-algebra of polynomial differential operators on $A[\theta_1,\ldots, \theta_r]$.
        \end{defn}
More concretely, $\mathrm{Poly}(r)$ is a $\Z/2$-graded dg-$A$-algebra with generators $\theta_1,\ldots,\theta_n$ and $T_1,\ldots,T_n$ of odd parity, where we think of $T_i=\frac{\partial}{\partial \theta_i}$. The relations are 
		
		\begin{itemize}
			\item $\theta_i\theta_j = -\theta_j \theta_i$.
			\item the graded Weyl relations $T_i\theta_j + \theta_jT_i=\delta_{ij}$. This is the graded Leibniz rule: if $u\in A[\theta_1,\ldots, \theta_r]$ then $(T_i\theta_j)(u)=T_i(\theta_ju)=\delta_{ij}u - \theta_jT_i(u)$.

			\item $T_iT_j=-T_jT_i$. This follows from the graded Leibniz rule again: for a polynomial $p\in A[\theta_1,\ldots, \theta_r]$, write $p=\theta_i\theta_jq + p'$ where $\theta_i\theta_j$ does not divide $p'$. Since both sides vanish on $p'$ we may take $p'=0$. Clearly the claim also holds for $p=0$ so we can assume that neither $\theta_i$ nor $\theta_j$ divide $q$, so both $T_i(q)$ and $T_j(q)$ are zero. The graded Leibniz rule then tells us that $T_iT_j(p)=-q$ and hence by symmetry $T_jT_i(-p)=-q$ too.
			\end{itemize}
            Define a differential $\delta$ on $\mathrm{Poly}(r)$ by putting $\delta=[D,-]$. We see that $\delta \theta_i = f_i$ and $\delta T_i = \sigma_i$.
            \begin{prop}
                There is an isomorphism $$(\mathrm{Poly}(r), \delta) \cong \mathrm{End}(L^\mathrm{stab})$$
            \end{prop}
            \begin{proof}
                We have an obvious inclusion $(\mathrm{Poly}(r), \delta) \into \mathrm{End}(L^\mathrm{stab})$ of dg algebras so we need to check that it is a surjection. This we can check on the underlying $A$-modules; it follows from a dimension count.
            \end{proof}
In particular, taking $L=k$ we obtain:
            \begin{cor}\label[cor]{dyccor}
            Let $P_\sigma$ be the differential graded algebra given by $\mathrm{Poly}(r)$ equipped with the differential $\delta \theta_i = x_i$ and $\delta T_i = \sigma_i$. Then there is a quasi-equivalence of $\Z/2$-graded dg categories $$\mathbf{MF}^\mathrm{dg}(A,\sigma)\simeq \per(P_\sigma)$$In particular,  $\mathbf{MF}^\mathrm{dg}(A,\sigma)$ and $P_\sigma$ are derived Morita equivalent.
            \end{cor}

		\begin{ex}
			Take $A=k\llbracket x\rrbracket$, $\sigma=x^2$, and $I=(x)$. In this case, $P_\sigma$ is of the form $$\frac{k\llbracket x\rrbracket\left\langle\theta,T\right\rangle}{(\theta^2, T^2, T\theta+\theta T = 1)}$$
			with $k\llbracket x\rrbracket$-linear differential defined on generators by $\delta\theta=x = \delta T$. The element $t\coloneqq T-\theta$ is a cocycle and the Weyl relations give $t^2=-1$. We have $\delta(\theta T)=-\delta(T\theta) = xt$, which implies that the cohomology algebra of $P_\sigma$ is $k[t]/(t^2+1)$. In fact, the $k$-subalgebra of $P_\sigma$ generated by $t$ is also $k[t]/(t^2+1)$, which implies that the inclusion $k[t]/(t^2+1) \into P_\sigma$ is a quasi-isomorphism. In particular, the singularity category of $\mathbb{R}[x]/x^2$ is $D^b(\mathbb{C})$.
			\end{ex}
		
		\begin{ex}
			The previous example was a special case of a theorem which says that if $\sigma$ is a quadratic form then $P_\sigma$ is quasi-isomorphic to the associated Clifford algebra. Suppose for an example that we are interested in the singularity $R=k\llbracket x,y\rrbracket/xy$ defined by the coordinate axes in the plane. After a change of coordinates $x=u-v, y=u+v$, we may take $R\cong k\llbracket u,v\rrbracket/(u^2-v^2)$ to be given by a quadratic form. Then $P_\sigma$ is quasi-isomorphic to the algebra $Cl_{1,1}(k)\coloneqq\frac{k[U,V]}{(U^2=-1, V^2=1, UV=-VU)}$ with $U,V$ in odd parity, and zero differential. The assignment $$U\mapsto \tbtm{0}{1}{-1}{0},\quad V\mapsto \tbtm{0}{1}{1}{0}$$defines an isomorphism $Cl_{1,1}(k)\cong M_2(k)$. Since $M_2(k)$ is Morita equivalent to $k$, we obtain a quasi-equivalence $\mathbf{MF}(k\llbracket x,y\rrbracket,xy)\simeq D^b(k)$.
			\end{ex}

\section{Hochschild theory of matrix factorisations}\label{hochsect}
We continue to follow \cite{dyck}, although we note that closely related results in more general contexts appear in \cite{preygel}. For more on Hochschild theory for dg categories, see \cite{toendglectures}. If $\C,\D$ are two dg categories, their \textbf{tensor product} has objects the pairs $(c,d)$ with $c\in\C$ and $d\in\D$ and hom-complexes given by $$(\C\otimes\D)((c,d),(c',d'))=\C(c,c')\otimes \D(d,d').$$

Let $\C$ be a dg category. Its \textbf{enveloping dg category} is the dg category $\C^e \coloneqq \C\otimes \C^\mathrm{op}$. Right modules over $\C^e$ are called $\C$\textbf{-bimodules}. 
\begin{ex}
    If $F:\C \to \D$ is a functor, the assignment $(x,y)\mapsto \D(Fx, Fy)$ is a $\C$-bimodule. The bimodule corresponding to the identity functor $\C \to \C$ is the \textbf{diagonal bimodule}, which we denote simply by $\C$.
\end{ex}
If $\D$ is any dg category, then $\Mod\D$ is an abelian category, and we can hence define $\mathrm{Ext}_\D$ and $\mathrm{Tor}_\D$ between two modules via the usual machinery of projective (or injective) resolutions.
\begin{defn}
    Let $\C$ be a dg category. Its \textbf{Hochschild homology} is $$\mathrm{HH}_*(\C)\coloneqq \mathrm{Tor}_*^{\C^e}(\C,\C)$$and its \textbf{Hochschild cohomology} is $$\mathrm{HH}^*(\C)\coloneqq \mathrm{Ext}^*_{\C^e}(\C,\C)$$
\end{defn}
		The assignment $\C\mapsto \mathrm{HH}_*(\C)$ is functorial. The assignment $\C\mapsto \mathrm{HH}^*(\C)$ is \textit{not} functorial, but it does produce a graded algebra.
        \begin{rmk}
            $\mathrm{HH}^*(\C)$ has lots of extra structure, such as the Gerstenhaber bracket and the related BV-algebra structure. On the chain level this enhances to a $B_\infty$-algebra structure \cite{gjoperads} and in fact an $E_2$-algebra structure \cite{msdeligne}, a fact known as Deligne's conjecture.
        \end{rmk}
        \begin{rmk}
             $\mathrm{HH}^*(\C)$ also has an interpretation in terms of functors: it is the algebra of derived endomorphisms of the identity functor, computed in the internal derived mapping category $\R\hom_{\mathbf{Hqe}}(\C,\C)$. The homology $\mathrm{HH}_*(\C)$ has a similar interpretation in terms of traces. See \cite{toendglectures} for more information.
        \end{rmk}
        \begin{rmk}
            Similarly to the classical case, one can compute the Hochschild co/homology of dg categories using a many-object version of the bar complex. See e.g.\ \cite{kellerdih} for an example of its use.
        \end{rmk}
		\begin{thm}[Keller, Lowen--Van den Bergh {\cite{kellerhoch, kellerdih, lvdb}}]
        Hochschild homology and cohomology are derived Morita invariant.
		\end{thm}
        \begin{proof}[Proof idea] 
        We describe only the proof strategy for cohomology; homology is easier since $\mathrm{HH}_\bullet$ is a functor (see \cite{kellernotes} for a more detailed exposition of both cases). Using the bar complex to compute $\mathrm{HH}^*(\C)$ gives us \textbf{limited functoriality}: if $\C\to \D$ is quasi-fully faithful, it induces a morphism $\mathrm{HH^\bullet}(\D) \to \mathrm{HH}^\bullet(\C)$ between Hochschild cohomology complexes. One then uses a gluing argument: if $\C\to\D$ is a functor then one can glue $\C$ and $\D$ together to form an upper-triangular matrix category $\mathcal{G}$, and one then shows that gluing produces long exact sequences in $\mathrm{HH}^*$. One uses this to obtain a zigzag of quasi-isomorphisms between $\mathrm{HH}^\bullet(\C)$ and $\mathrm{HH}^\bullet(\D)$.
        \end{proof}
	\begin{cor}
	    There are isomorphisms
\begin{align*}
    \mathrm{HH}_*(\mathbf{MF}^\mathrm{dg}(A,\sigma)) &\cong \mathrm{HH}_*(P_\sigma)\\
       \mathrm{HH}^*(\mathbf{MF}^\mathrm{dg}(A,\sigma)) &\cong \mathrm{HH}^*(P_\sigma).
\end{align*}
\end{cor}
Let $R=A/\sigma$ be a complete local hypersurface singularity. The \textbf{Milnor algebra} of $R$ is the algebra $$M_\sigma\coloneqq \frac{R}{{\left(\frac{\partial \sigma}{\partial x_1},\ldots,\frac{\partial \sigma}{\partial x_n}\right)}}$$It is finite-dimensional if and only if $R$ is an isolated singularity. The \textbf{Tjurina algebra} is $T_\sigma\coloneqq M_\sigma/\sigma$.
\begin{rmk}
    When $R$ is a quasi-homogeneous singularity, one has $T_\sigma\cong M_\sigma$, since $\sigma$ is contained in the ideal $\left(\frac{\partial \sigma}{\partial x_1},\ldots,\frac{\partial \sigma}{\partial x_n}\right)$ by Euler's theorem on homogenous functions.
\end{rmk}

\begin{thm}[Dyckerhoff]\label[thm]{dythm}
Let $R=A/\sigma$ be a complete local isolated hypersurface singularity. \begin{enumerate}
    \item The Hochschild cohomology of the 2-periodic dg category $\mathbf{MF}^\mathrm{dg}(A,\sigma)$ is the Milnor algebra, concentrated in even parity.
    \item The Hochschild homology of the 2-periodic dg category $\mathbf{MF}^\mathrm{dg}(A,\sigma)$ is the Milnor algebra, concentrated in parity $n=\mathrm{dim}(A)=1+\mathrm{dim}(R)$.
\end{enumerate}
\end{thm}
\begin{proof}[Proof sketch]
We give the argument for cohomology only. The argument for homology is similar and requires one to interpret $\mathrm{HH}_*(\C)$ as the trace of the identity functor, via a theorem of To\"en (cf.\ \cite{toendglectures}).

We introduce a \textbf{tensor product of matrix factorisations}: if $R'=A'/\sigma'$ is another complete local isolated hypersurface singularity, the tensor product of the two corresponding dg categories is defined to be the dg category $\mathbf{MF}^\mathrm{dg}(A\hat\otimes A', \sigma\otimes 1+ 1\otimes \sigma')$. Dyckerhoff proves that this tensor product is compatible with the Morita theory, in the sense that there is a derived Morita equivalence $$\mathbf{MF}^\mathrm{dg}(A\hat\otimes A', \sigma\otimes 1+1\otimes \sigma') \simeq P_{\sigma}\otimes P_{\sigma'}.$$Moreover, he proves that there is a derived Morita equivalence  $$\mathbf{MF}^\mathrm{dg}(A, -\sigma) \simeq P_{\sigma}^\mathrm{op}.$$
In particular, if we put $\tilde \sigma \coloneqq \sigma\otimes 1-1\otimes \sigma$, the dg category $\mathbf{MF}^\mathrm{dg}(A\hat\otimes A, \tilde\sigma)$ is derived Morita equivalent to $P_\sigma^e$, which yields a quasi-equivalence $$D_\mathrm{dg}(\mathbf{MF}^\mathrm{dg}(A\hat\otimes A, \tilde\sigma))\simeq D_\mathrm{dg}(P_\sigma^e)$$

Across this equivalence, the diagonal bimodule $P_\sigma$ corresponds to the stabilisation of the diagonal bimodule $R^\mathrm{stab}$. So we have an isomorphism\begin{align*}
    \mathrm{HH}^*(\mathbf{MF}^\mathrm{dg}(A, \sigma))&\cong H^*\mathrm{End}_{\mathbf{MF}^\mathrm{dg}(A\hat\otimes A, \tilde\sigma)}(R^\mathrm{stab})
\end{align*}
The claim follows from an explicit calculation similar to that of $\mathrm{End}(k^\mathrm{stab})$.
\end{proof}
\begin{rmk}\label{kunrmk}
The dg derived category $D_\mathrm{dg}(\mathbf{MF}^\mathrm{dg}(A\hat\otimes A', 1\otimes \sigma'-\sigma\otimes 1))$ can be identified with the internal derived mapping dg category $$\R\hom_{\mathbf{Hqe}}(\mathbf{MF}^\mathrm{dg}(A, \sigma), \mathbf{MF}^\mathrm{dg}(A',\sigma'))$$and in particular $D_\mathrm{dg}(\mathbf{MF}^\mathrm{dg}(A\hat\otimes A, \tilde\sigma))$ is identified with the internal derived endomorphism dg category of $\mathbf{MF}^\mathrm{dg}(A, \sigma)$.
			\end{rmk}

\begin{rmk}
    An explicit model for $D_\mathrm{dg}(\mathbf{MF}^\mathrm{dg}(A, \sigma))$ is given by a category $\mathbf{MF}_\infty^\mathrm{dg}(A, \sigma)$ of \textbf{infinite rank matrix factorisations}. The MCM analogue of this construction uses certain infinite rank MCM modules \cite{chenMCM}.
\end{rmk}

\begin{rmk}
	The isomorphism
    $$\mathrm{HH}_*(\mathbf{MF}^\mathrm{dg}(A,\sigma))\simeq \mathrm{HH}^*(\mathbf{MF}^\mathrm{dg}(A,\sigma))[n]$$
    is a manifestation of the fact that $\mathbf{MF}(A,\sigma)$ is a $n$-Calabi-Yau dg category.
	\end{rmk}

\begin{rmk}\label[rmk]{mfremark}
Let $k[u,u^{-1}]$ be the graded Laurent polynomial ring, where $u$ has degree $2$. Then a $\Z/2$-graded $k$-linear dg category is the same thing as a $\Z$-graded $k[u,u^{-1}]$-linear dg category: note that a $\Z/2$-graded complex $X^0\substack{\leftarrow\\[-1em] \rightarrow} X^1$ is the same thing as a $2$-periodic $\Z$-graded complex $$\cdots\to X^0 \to X^1 \to X^0 \to X^1 \to X^0 \to \cdots$$ A $2$-periodic complex is then the same thing as a complex over $k[u,u^{-1}]$, with the action of $u$ giving the periodicity isomorphisms. There is an obvious algebra map $k \to k[u,u^{-1}]$, and by restriction along this map we can view a $\Z/2$-graded dg category as a $\Z$-graded dg category. In particular, if $\C$ is a $\Z/2$-graded dg category, it has two different kinds of Hochschild cohomology: firstly $\mathrm{HH}^*_{k[u,u^{-1}]}(\C)$, where we work over the base ring of Laurent polynomials, and $\mathrm{HH}^*_{k}(\C)$, where we work over the base ring $k$. \textit{There is no need for these to agree!} In fact, $\mathrm{HH}^0_{k}(\mathbf{MF}^\mathrm{dg}(A,\sigma))$ is the Tjurina algebra \cite{kellersing}, whilst we know from \Cref{dythm} that $\mathrm{HH}^0_{k[u,u^{-1}]}(\mathbf{MF}^\mathrm{dg}(A,\sigma))$ is the Milnor algebra.
	\end{rmk}

\chapter{Relative singularity categories}\label[chapter]{rschapt}
Here we study a relative notion of singularity category and show that relative singularity categories are often controlled by derived quotients. Along the way we will see some applications to algebraic geometry. The original references here are \cite{chenMCM, burbankalck}. A comprehensive treatment of relative singularity categories can be found in \cite{kalckyang, kalckyang2, kalckyang3} and a general overview is \cite{martinthesis}.
\section{Some noncommutative geometry}
As always, we let $A$ be a two-sided noetherian ring. We fix an idempotent $e\in A$, and we put $R\coloneqq eAe$, the \textbf{cornering}.

\begin{exer}
    Show that the functor $j_!:D(R) \to D(A)$ defined by \\ $j_!(X) \coloneqq X\lot_R eA$ is fully faithful.
\end{exer}
In particular, we obtain a fully faithful embedding $j_!:\per(R) \into D(A)$ which sends $R$ to $eA$. 

\begin{defn}
    The \textbf{relative singularity category} is the Verdier quotient
    $$\Delta_R(A)\coloneqq \frac{D^b(A)}{j_!\per(R)} = \frac{D^b(A)}{\mathbf{thick}(eA)}$$
\end{defn}

\begin{ex}If $e=0$, then $R=0$ and we have $\Delta_0(A) = D^b(A)$. If $e=1$, then $R=A$ and we have $\Delta_A(A) = D_\mathrm{sg}(A)$.  
\end{ex}
\begin{ex}
    Let $R$ be a commutative ring and $M$ an $R$-module. Let $A$ be the ring $A\coloneqq \enn_R(R\oplus M)$ and let $e=\mathrm{id}_R$, so that we have isomorphisms $eA\cong R\oplus M$ and $eAe\cong R$. We call such $A$ a \textbf{noncommutative partial resolution} of $R$.
\end{ex}
\begin{exer}
    Show that the functor $j_!$ has a right adjoint $j^*$ defined by $j^*(X) \coloneqq Xe$. Show that $j^*$ sends $\mathbf{thick}(eA)$ to $\per(R)$, and hence induces a triangle functor $\Delta_R (A) \to D_\mathrm{sg}(R)$.
\end{exer}

 		Where does the definition of a relative singularity category come from? One motivation is by analogy with algebraic geometry. Let $k$ be a field of characteristic zero, let $X$ be a $k$-variety, and let $\pi:\tilde X \to X$ be a resolution of singularities. There is a derived pushforward functor $\R\pi_*:D^b(\tilde X) \to D^b(X)$ and a derived pullback functor $\mathbb{L}\pi^*:D^b(X)\to D(\tilde X)$. They are adjoints, in the sense that for any $\mathcal{F}\in D^b(X)$ and $\mathcal{G}\in D^b(\tilde X)$ then there is a natural isomorphism $$\hom(\mathbb{L}\pi^*\mathcal{F},\mathcal{G})\cong \hom(\mathcal{F},\R\pi_*\mathcal{G})$$ The projection formula tells us that the unit of this adjunction is the natural map $\mathcal{F} \to \mathcal{F}\lot \R\pi_* \mathcal{O}_{\tilde X}$. Say that $X$ \textbf{has rational singularities} if this unit is an equivalence at $\mathcal{O}_X$, or equivalently if the natural map $\mathcal{O}_{X} \to \R\pi_*\mathcal{O}_{\tilde X}$ is a quasi-isomorphism. If this is the case, then by the projection formula $\mathbb{L}\pi^*$ is fully faithful, and in particular we get an embedding $\per(X)\into D^b(\tilde X)\simeq \per(\tilde X)$.
 		
 		\p The ring $R$ is to be thought of as $X$, and the ring $A$ is supposed to be like its resolution $\tilde X$. The pullback $\mathbb{L}\pi^*$ gets replaced by $j_!$, and hence the relative singularity category is supposed to behave like the Verdier quotient $\per(\tilde X)/\per(X)$, which is a `geometric' relative singularity category which measures the failure of $\mathbb{L}\pi^*\per(X)$ to be all of $\per(\tilde X)$.

        \p With this in mind, we can try to globalise our definition of noncommutative partial resolution. Let $X$ be a variety with isolated Gorenstein singularities (i.e.\ $X$ has isolated singularities, and at each singular point $p\in X$ the formal completion $\hat X_p$ is a complete local Gorenstein ring). If $\mathcal{F}$ is a coherent sheaf on $X$, put $\mathcal{F}'\coloneqq \mathcal{F}\oplus \mathcal{O}_X$. Let $\mathcal{A}\coloneqq \mathcal{E}\kern -1pt\mathrm{nd}(\mathcal{F}')$ be the endomorphism sheaf, which is a coherent sheaf of rings on $X$. This defines a ringed space $\mathbb{X}\coloneqq (X,\mathcal{A})$ which is supposed to behave like a partial resolution of $X$. Note that $\mathbb{X}$ is a sheaf of noetherian rings so it makes sense to consider the abelian category $\mathrm{Coh}(\mathbb{X})$ and its bounded derived category $D^b(\mathbb{X})$. Observe that there are natural functors $\phi\coloneqq \mathcal{F}'\lot-: \per(X) \to D^b(\mathbb{X})$ and $\psi\coloneqq \R\hom(\mathcal{F}',-): D^b(\mathbb{X}) \to D^b(X)$.

 		\begin{thm}[e.g.\ \cite{bdtilt}]
 			$\phi\dashv\psi$ and the unit $\id \to \psi\phi$ is an isomorphism; in particular $\phi$ is fully faithful.
\end{thm} 			
 		Hence we may define the \textbf{relative singularity category} $$\Delta_X(\mathbb{X})\coloneqq \frac{D^b(\mathbb{X})}{\per(X)}.$$
        We will give a description of the relative singularity category in a special case, due to Burban and Kalck \cite{burbankalck}. Before we do so, we first introduce some useful concepts.

        \begin{defn}
              If $\mathcal{T}$ is a triangulated category, an \textbf{idempotent} in $\mathcal{T}$ is a map $e:X \to X$ with $e^2=e$. An idempotent \textbf{splits} if $X$ can be written as $\ker(e)\oplus \mathrm{im}(e)$.
        \end{defn}
       There is a universal triangle functor $\mathcal{T}\to\mathcal{T}^\omega$ which splits idempotents; we call $\mathcal{T}^\omega$ the \textbf{idempotent completion} of $\mathcal{T}$. The assignment $\mathcal{T}\mapsto\mathcal{T}^\omega$ is functorial. For more on idempotent completion see \cite{bsidempotent}.

\begin{defn}
    A \textbf{quiver} is a directed (multi-)graph. A \textbf{path} in a quiver is a sequence of arrows $a_1 a_2a_3\cdots a_n$ such that the head of $a_i$ is the tail of $a_{i+1}$. We allow paths of length zero; if $v$ is a vertex of $Q$ then we denote the unique path $v\to v$ of length zero by $e_v$. Paths can be composed when their endpoints agree. Observe that $e_va=a$ for any $a$ starting at $v$, and similarly $be_v=b$ for any $b$ ending at $v$. If $Q$ is a quiver and $k$ is a commutative ring, the \textbf{path algebra} $kQ$ has as a $k$-basis the set of paths in $Q$. Multiplication in $kQ$ is composition of paths when defined and is zero otherwise.
\end{defn}
\begin{exer}[idempotents]Compute $e_ue_v$ for any two (not necessarily distinct) vertices $u,v$ of $Q$. If $Q$ has finitely many vertices $1,\ldots, n$, show that we have an equality $1=e_1+\cdots +e_n$ in $kQ$.
\end{exer}
\begin{exer}\label[exer]{QuiverTensorEx}
   We write $Q_0$ for the set of vertices of $Q$ and $Q_1$ for the set of arrows. Show that there is an isomorphism $kQ\cong T_{kQ_0}(kQ_1)$ between the path algebra and the tensor algebra (over $kQ_0$) on the arrows. (Hint: consider the map which sends a tensor $a_1\otimes\cdots \otimes a_n$ to the path $a_1\cdots a_n$.)
\end{exer}

 		\begin{thm}[\cite{burbankalck}]\label{bkthrm}Let $X$ be a nodal curve with $n$ singular points and let $\mathcal{F}$ be the ideal sheaf of the singular locus. Then the corresponding ringed space $\mathbb{X}$ is a noncommutative resolution of $X$ (in fact it has global dimension $2$). There is an equivalence of triangulated categories
 			$$\Delta_X(\mathbb{X})^\omega\simeq \bigoplus_{i=1}^n \left(\frac{D^b(\Lambda)}{\mathrm{Band}(\Lambda)}\right)^\omega$$where $\Lambda$ is the path algebra of the quiver
 			
 			$$
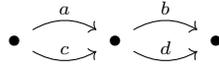
\begin{tikzcd}
 				\bullet \ar[r, bend left=30, "a"] \ar[r, bend right=30, "c"] & \bullet \ar[r, bend left=30, "b"]  \ar[r, bend right=30, "d"] & \bullet
 				\end{tikzcd}
 			$$with relations $ab=0=cd$, and the category of band modules is defined to be $$\mathrm{Band}(\Lambda)\coloneqq\{C\in D^b(\Lambda): \quad\tau(C)\simeq C\}$$where $\tau$, the Auslander--Reiten translate, is given by $\tau(C)\coloneqq \Lambda^* \lot_\Lambda C[-1]$.
 			\end{thm}
 		\begin{rmk}
 			In particular, the two triangulated categories $\Delta_X(\mathbb{X})$ and $\bigoplus_{i=1}^n \frac{D^b(\Lambda)}{\mathrm{Band}(\Lambda)}$ are equivalent up to direct summands, and hence have the same K-theory.
 			\end{rmk}
            \begin{rmk}
            Under the conditions of the above theorem, $\mathrm{Proj}$ of the Rees algebra of $\mathcal{F}$ is the blowup of $X$ at the singular locus, which is a resolution since $X$ has nodal singularities.
            \end{rmk}
 		
\section{Derived quotients}
In this section we follow \cite{bcl, kalckyang2, bsingcats}. Let $k$ be a fixed commutative semisimple ring, which we will use as our base.
\begin{defn}
    Let $A$ be a $k$-algebra with an idempotent $e$. The \textbf{derived quotient} $\dq$ is defined to be the initial object of the homotopy category of dg algebras $f:A\to B$ under $A$ such that $f(e)$ vanishes in $H^0(B)$.
\end{defn}
 More generally, one may make the above definition when $A$ is a dg algebra. We omit the proof of existence, but we give two constructions of the derived quotient. There are (at least) two methods:

\paragraph{1. Genuine quotients of resolutions} A dg-$k$-algebra $B$ is \textbf{cofibrant} if it is freely generated over $k$, by a (possibly transfinite) number of generators $x_1,x_2, x_3,\ldots$ and moreover the differential $d$ satisfies the \textbf{upper-triangular condition:} $d(x_i)$ is contained in the ideal of $B$ generated by the $x_j$ for $j<i$. A \textbf{cofibrant resolution} of $A$ is a quasi-isomorphism $\tilde A \to A$ with $A$ cofibrant.
\begin{ex}
    The dg algebra $k\langle x,y\rangle$ with $\deg(x)=0, \deg(y)=-1$ and $dy=x^2$ is a cofibrant resolution of $k[x]/x^2$. Here we may put $x=x_1$ and $y=x_2$.
\end{ex}
\begin{rmk}
    In characteristic zero, if $A$ is commutative then we may also take $\tilde A$ to be a commutative dg algebra, so that for example the dg algebra $$\Q[x,y]= \Q\langle x,y\rangle/(xy-yx, y^2)$$ with $dy=x^2$ is a commutative cofibrant resolution of $\Q[x]/x^2$.
\end{rmk}
\begin{ex}
    Let $Q$ be a finite quiver, $\mathbb{F}$ a field and $k=\mathbb{F}Q_0$ the semisimple algebra on the vertex idempotents. Then the path algebra $\mathbb{F}Q$ is a cofibrant $k$-algebra, since it is freely generated as a $k$-algebra by the arrows (cf.\ \Cref{QuiverTensorEx}). If $Q$ is a graded quiver (i.e.\ the arrows have a grading) then $\mathbb{F}Q$ inherits a grading and becomes a cofibrant dg algebra (with zero differential).
\end{ex}

If $\tilde A \to A$ is a cofibrant resolution of $A$, and $\tilde e \in \tilde A$ is an idempotent that lifts $e$, then we may compute $\dq$ as the actual quotient $\tilde A / \tilde A \tilde e \tilde A$.

\paragraph{2. Drinfeld quotients}
If $A$ is a ring with an idempotent $e$, we may view it as a two-object dg category: the first object $A_1$ corresponds to the idempotent $e$ and has endomorphisms $R=eAe$, and the second object $A_2$ corresponds to the idempotent $1-e$ and has endomorphisms $(1-e)A(1-e)$. The derived quotient $\dq$ is then quasi-isomorphic to the one-object dg category $$A/A_1 \quad\simeq\quad A\langle h\rangle / (dh=e).$$More concretely, the dg algebra $Q\coloneqq A\langle h\rangle / (dh=e)$ is connective, with $Q^{>0}=0$, $Q^0=A$, and $Q^{-1-i} = Ae\otimes R^{i} \otimes eA$ for $i\geq 0$. The differential is the alternating signed sum of the multiplication maps $Ae \otimes R \to Ae$, $R\otimes R \to R$, and $R\otimes eA \to eA$. In particular we see that $H^0(\dq)\cong A/AeA$, the usual quotient.
\begin{rmk}\label[rmk]{stabextrem}
    If $R$ is a commutative Gorenstein ring, $M$ is a MCM $R$-module, and $A=\enn_R(R\oplus M)$ with $e=\id_R$, then there is a $k$-algebra isomorphism $A/AeA \cong \underline{\enn}_R(M) = \enn_{\underline{\mathbf{MCM}}(R)}(M)$. In fact, one can check using \Cref{stabextprop} and the `Drinfeld quotient' description that $H^j(Q)\cong \underline{\mathrm{Ext}}^j_R(M,M)$ for $j<0$.
\end{rmk}

Now we return to the setting of relative singularity categories. As before, we choose a noetherian ring $A$ with an idempotent $e$ and put $R=eAe$.

\begin{prop}\label{recoll}
    There is a \textbf{recollement}
    $$\begin{tikzcd}[column sep=huge]
D(\dq) \ar[r,"i_*=i_!"]& D(A)\ar[l,bend left=25,"i^!"']\ar[l,bend right=25,"i^*"']\ar[r,"j^!=j^*"] & D(R)\ar[l,bend left=25,"j_*"']\ar[l,bend right=25,"j_!"']
\end{tikzcd}$$which is to say that the named functors exist and satisfy:\ \begin{itemize}
    \item If $F$ appears directly above $G$ then $F\dashv G$.
    \item $j^*i_*=0$.
    \item $i_*, j_*$ and $j_!$ are fully faithful.
    \item For every $X$ in $D(A)$ there are distinguished triangles of the form $$i_!i^!X \to X \to j_*j^* X \to\phantom{.}$$ and $$j_!j^!X \to X \to i_*i^* X \to.$$
\end{itemize}
\end{prop}
\begin{proof}[Proof idea]
    Write $Q\coloneqq \dq$ for brevity. We view $Q$ as both an $A$-bimodule and a $Q$-bimodule. Put \begin{align*}
		i^*\coloneqq -\lot_A Q, &\quad j_!\coloneqq  -\lot_{R} eA
		\\ i_*=\R\hom_{Q}(Q,-), &\quad j^!\coloneqq \R\hom_A(eA,-)
		\\ i_!\coloneqq \lot_{Q}Q, & \quad j^*\coloneqq -\lot_A Ae
		\\ i^! \coloneqq \R\hom_{A}(Q,-), & \quad j_*\coloneqq \R\hom_{R}(Ae,-)	\end{align*}
        One can then check the required statements quite explicitly. See e.g.\ \cite{bsingcats} for the full proof.
\end{proof}
Neeman--Thomason--Trobaugh--Yao localisation then yields:
\begin{cor}
    There is a triangle equivalence
$$\per(\dq) \longrightarrow\left(\frac{\per(A)}{j_!\per(R)}\right)^\omega$$which sends $\dq \mapsto A$.
\end{cor}

\begin{cor}\label[cor]{kyp1}
    There is a triangle functor
    $$\per(\dq) \to \Delta_R(A)^\omega$$which is an equivalence when $A$ has finite global dimension.
\end{cor}
\begin{proof}
    The natural map $\per(A) \into D^b(A)$ yields, after taking quotients, a map $\per A / j_!\per R \to \Delta_R(A)$ which is an equivalence when $A$ has finite global dimension.
\end{proof}
We will be interested in the composition $$\per(\dq) \to \Delta_R(A)^\omega \to D_\mathrm{sg}(R)^\omega$$ which sends $\dq$ to $eA$. The following is a result of Kalck and Yang.

\begin{thm}[{\cite[Theorem 6.6]{kalckyang2}}]\label[thm]{kycor}
Let $A$ be a noetherian $k$-algebra of finite global dimension with $e\in A$ an idempotent. Put $R=eAe$. Then there is a triangle equivalence
$$\frac{\per(\dq)}{D_\mathrm{fg}(\dq)} \longrightarrow D_\mathrm{sg}(R)^\omega$$where $D_\mathrm{fg}(\dq)$ denotes the triangulated subcategory of dg $\dq$-modules $N$ whose cohomology is finitely generated over $A/AeA$, i.e.\ only finitely many $H^i(N)$ are nonzero, and they are all finitely generated $A/AeA$-modules.
\end{thm}
\begin{proof}[Proof idea]
   By \Cref{kyp1} we have a triangle equivalence $F:\per(\dq) \to \Delta_R(A)^\omega$. If $\mathcal{K}$ denotes the kernel of the projection $\Delta_R(A)^\omega \to D_\mathrm{sg}(R)^\omega$ then we obtain a triangle equivalence $\per(\dq) / F^*(\mathcal{K}) \simeq D_\mathrm{sg}(R)^\omega$. It suffices to identify $F^*(\mathcal{K})$ with $D_\mathrm{fg}(\dq)$; one shows that they are both the thick subcategory generated by $\mathbf{mod}\text{-}A/AeA$.
\end{proof}

\begin{rmk}
    When $A/AeA$ is finite dimensional, then $D_\mathrm{fg}(\dq)$ is identified with $D_\mathrm{fd}(\dq)$, the subcategory of $D(\dq)$ on those modules $M$ such that $\oplus_iH^i(M)$ is a finite dimensional $k$-vector space. If $B$ is a homologically smooth dg algebra (i.e. $B$ is a perfect $B$-bimodule), then a standard argument shows that $D_\mathrm{fd}(B) \subseteq \per(B)$; indeed, if $X\in D_\mathrm{fd}(B) $ then $X\lot_B P \in \per(B)$ for any perfect $B$-bimodule $P$ (it suffices to check this for $P=B^e$ itself), and in particular for $P=B$. The quotient $\frac{D_\mathrm{fd}(B) }{\per(B)}$ is known as the \textbf{generalised cluster category} of $B$, after Amiot \cite{amiotcluster}.
\end{rmk}

\begin{rmk}\label[rmk]{clihsrm}
If $R$ is a complete local isolated hypersurface singularity, then $D_\mathrm{sg}(R)$ is idempotent complete.
\end{rmk}

\begin{rmk}
    All of the above triangle functors enhance to functors of pretriangulated dg categories \cite{bsingcats}.
\end{rmk}

\section{Classifying isolated hypersurface singularities}
In this section we follow \cite{bsingcats} and \cite{huakeller}. Here we work over an algebraically closed field of characteristic zero.

\begin{thm}
    Let $R_1$, $R_2$ be complete local isolated hypersurface singularities of the same dimension. Let $M_i$ be non-projective MCM $R_i$-modules and put $A_i=\enn_{R_i}(R_i\oplus M_i)$ with $e_i = \mathrm{id}_{R_i}$. Put $Q_i\coloneqq {A_i/^{\mathbb{L}}\kern -2pt A_ie_iA_i}$. Then $R_1\cong R_2$ if and only if there is a quasi-isomorphism $Q_1\simeq Q_2$.
\end{thm}
\begin{proof}[Proof idea]
One direction is clear, so we need to show that if $Q_1\simeq Q_2$ then we have an isomorphism $R_1\cong R_2$. We do this by showing that the quasi-isomorphism class of $Q_i$ determines $R_i$. For readability we now drop the $i$ subscripts. The proof consists of three steps:
\paragraph{Step 1: $Q$ determines $D_\mathrm{sg}(R)$.}If $A$ was smooth, we would be done by an application of \Cref{kycor} (and \Cref{clihsrm}). In general, we must be more careful. First one does an explicit computation to show that $Q$ is quasi-isomorphic to the truncation $\tau_{\leq 0}\R\underline\enn_R(M)$ of the dg-categorical endomorphism complex; this lifts the computations of \Cref{stabextrem} to the level of dg algebras. Since $R$ is a hypersurface, its singularity category has 2-periodicity (\Cref{perthm}). Since one can recover a 2-periodic complex from its truncation, a localisation argument now shows that $Q$ determines the full endomorphism dg algebra $\R\underline\enn_R(M)$. A theorem of Takahashi states that in this setting, $D_\mathrm{sg}(R)$ has no nontrivial thick subcategories - in general the thick subcategories are parameterised by the specialisation-closed subsets of $\spec (R)$ which are contained in the singular locus \cite{takahashi}. In particular, since $M\not\simeq 0$ in the singularity category, $M$ must be a thick generator. So we see that $D_\mathrm{sg}(R)\simeq \per(\R\underline\enn_R(M))$ by \Cref{tiltingprop}.
\paragraph{Step 2: $D_\mathrm{sg}(R)$ determines the Tjurina algebra $T_R$.} This follows from computations of \cite{kellersing} and \cite{buenosaires}; more accurately we have isomorphisms $\mathrm{HH}^0(D_\mathrm{sg}(R)) \cong \mathrm{HH}^0(R)\cong T_R$ (cf.\ \Cref{mfremark})

\paragraph{Step 3: $T_R$ determines $R$.} This is the \textbf{Mather--Yau theorem} (cf.\ \cite{gpmather}): let $\sigma_1, \sigma_2 \in \mathfrak{m}_{k\llbracket \underline x \rrbracket}$ define isolated singularities. Then ${k\llbracket \underline x \rrbracket}/\sigma_1\cong {k\llbracket \underline x \rrbracket}/\sigma_2$ if and only if $T_{\sigma_1}\cong T_{\sigma_2}$.
\end{proof}

\begin{rmk}
    Note that due to the existence of Kn\"orrer periodicity, singularities of different dimensions may have isomorphic Tjurina algebras. However, when the dimension is fixed, the Tjurina algebra is a complete invariant of the singularity, per the Mather--Yau theorem. As stated the theorem is false in characteristic $p$ - just consider the two hypersurfaces defined by the vanishing of $f\coloneqq x^{p+1}+y^{p+1}$ and $f+x^p$. However, a modified version is true. The Milnor algebra also classifies complete local isolated singularities, up to a different notion of equivalence.
\end{rmk}

\begin{rmk}Over $\mathbb{C}$, one can use the above machinery to classify certain threefold birational contractions $X \to \spec(R)$, known as \textbf{flopping contractions}, that arise when running the minimal model program. Van den Bergh proves that in this setting there exists an MCM $R$-module $M$ such that $X$ is derived equivalent to the corresponding partial resolution $A=\enn_R(R\oplus M)$ \cite{vdb}, and in this setting the finite dimensional algebra $A/AeA$ is known as the \textbf{contraction algebra} \cite{DWncdf}. The corresponding derived quotient $\dq$ is known as the \textbf{derived contraction algebra} \cite{bsingcats} and admits a derived deformation-theoretic interpretation in terms of the contracted curves \cite{ddefpt}. The above theorem then states that the quasi-isomorphism class of the derived contraction algebra recovers the analytic type of $R$ near its singularity. In fact, when $X$ is smooth, just the contraction algebra $A/AeA = H^0(\dq)$ will do to recover $\hat R$ \cite{JMK, WKL}.
\end{rmk}

\chapter{Deformations of Kleinian singularities}\label[chapter]{kleinchapt}
We use some of the theory built up in \Cref{rschapt} to study singularity categories of certain noncommutative deformations of Kleinian singularities. A good general reference here is \cite{crawfthesis}. 

\section{The McKay correspondence}
For more on the McKay correspondence as it will be relevant to us, see either \cite[\S2]{derivedmckay} or \cite[\S1]{crawfthesis} and the references therein. The original reference is \cite{mckay}.

If $G$ is a subgroup of $\mathrm{GL}_2(\mathbb{C})$, it acts on the affine plane $\mathbb{C}^2$ and hence on its coordinate ring $R\coloneqq \mathbb{C}[x,y]$. The corresponding quotient is defined to be Spec of the invariant ring $R^G\coloneqq \{r\in R: rg=g \text{ for all }g\in G\}$.

\begin{defn}
    A \textbf{Kleinian singularity} is a singularity of the form $\spec(R^G)$ where $G$ is a nontrivial finite subgroup of $\mathrm{SL}_2(\mathbb{C})$.
\end{defn} 
Every Kleinian singularity is isomorphic to a hypersurface in $\mathbb{C}^3$, given by some polynomial in three variables. Since the finite subgroups of $\mathrm{SL}_2(\mathbb{C})$ are classified, so are the corresponding singularities, and we tabulate them below. 

Associated to every finite subgroup $G$ of $\mathrm{SL}_2(\mathbb{C})$ is a \textbf{Dynkin graph}, drawn in black in the table below. The graph is constructed solely from the representation theory of $G$: the vertices correspond to the nontrivial irreducible representations of $G$ and there is an edge between $i$ and $j$ if $V_j$ shows up in the decomposition of $V\otimes V_i$, where $V$ is the natural two-dimensional representation. Since $G$ is a subgroup of $\mathrm{SL}_2(\mathbb{C})$ this is in fact a symmetric condition. If we admit the trivial representation too we obtain an \textbf{extended Dynkin graph}; the extending vertex is labeled by a red zero in the table. Sometimes these are also called \textbf{affine Dynkin graphs}.

\begin{adjustbox}{center}
\begin{tabular}{ c|c|c|c}
 Type & Group & Polynomial & Dynkin graph\\ \hline
 $A_n$, \ $n\geq 1$ & cyclic & $x^2 + y^2 + z^{n+1}$ & \begin{tikzcd}[column sep=1em, row sep=1em]
 &&&{\textcolor{red}{0}}\ar[rrd, red, no head, bend left=15]\ar[llld, red, no head, bend right=15]&&\\
     1\ar[r, no head]&2\ar[r, no head]&3\ar[r, no head]& \cdots\ar[r, no head] & n-1\ar[r, no head] & n
 \end{tikzcd}\\\hline
 $D_n$, \ $n\geq 4$ & binary dihedral & $x^2+y^2z+z^{n-1}$ & \begin{tikzcd}[column sep=1em, row sep=0.7em]
     {\textcolor{red}{0}}\ar[dr, red, no head]&&&&&n-1\\
     &2\ar[r, no head]&3\ar[r, no head]&\cdots\ar[r, no head]&n-2\ar[ur, no head]\ar[dr, no head]&\\
     1\ar[ur, no head]&&&&&n
 \end{tikzcd}\\\hline
 $E_6$ & binary tetrahedral & $x^2+y^3+z^4$ & \begin{tikzcd}[column sep=1em, row sep=1em]
     &&{\textcolor{red}{0}}&&\\
     &&1\ar[u, red, no head]&&\\
     2\ar[r, no head]&3\ar[r, no head]&4\ar[r, no head]\ar[u, no head]&5\ar[r, no head]&6
 \end{tikzcd}\\\hline
 $E_7$ & binary octahedral & $x^2+y^3+yz^3$ & \begin{tikzcd}[column sep=1em, row sep=1em]
     &&&7&&\\
     {\textcolor{red}{0}}\ar[r, red, no head]&1\ar[r, no head]&2\ar[r, no head]&3\ar[r, no head]\ar[u, no head]&4\ar[r, no head]&5\ar[r, no head]&6
 \end{tikzcd}\\\hline
 $E_8$ & binary icosahedral & $x^2+y^3+z^5$ & \begin{tikzcd}[column sep=1em, row sep=1em]
 &&&&&8&&\\
     {\textcolor{red}{0}}\ar[r, red, no head]&1\ar[r, no head]&2\ar[r, no head]&3\ar[r, no head]&4\ar[r, no head]&5\ar[r, no head]\ar[u, no head]&6\ar[r, no head]&7
 \end{tikzcd}
\end{tabular}
\end{adjustbox}

\begin{exer}
    Check using the table that a Kleinian singularity is isolated, with a unique singular point at the origin.
\end{exer}

Let $Y \to \spec(R^G)$ be the minimal resolution of a Kleinian singularity (i.e.\ the resolution through which no other resolution factors - it is a nontrivial fact that minimal resolutions exist and are unique in this setting). The exceptional divisor is a union of irreducible rational curves, and if two distinct curves meet then they intersect at a single nodal point. The \textbf{dual graph} of the resolution has vertices the irreducible components of the exceptional divisor, and two curves are linked by an edge precisely when they intersect. The original and most fundamental form of the McKay correspondence is the following fact:
\begin{thm}[McKay]
 The dual graph of the minimal resolution of $\spec(R^G)$ is precisely the Dynkin graph of $G$.
\end{thm}

\begin{rmk}
    Ito and Nakamura showed that the minimal resolution of $R^G$ is given by the \textbf{$G$-Hilbert scheme} of $R$, which, roughly, parameterises the $G$-invariant zero-dimensional subschemes of $R$. They used this to give a very explicit version of the McKay correspondence \cite{in1, in2}. Kapranov and Vasserot used this to give an equivalence $$D_G(\mathbb{C}^2) \simeq D^b(Y)$$where the left-hand side denotes the derived category of $G$-equivariant sheaves \cite{kapvas}. Bridgeland--King--Reid obtained similar results for $\mathrm{SL}_3(\mathbb{C})$, Ishii and Ueda for $\mathrm{GL}_2(\mathbb{C})$, and Kawamata for $\mathrm{GL}_3(\mathbb{C})$ \cite{bkr, ishii, ishiiueda, kawmckay}.
\end{rmk}

\section{Preprojective algebras}
We keep notation as before. The \textbf{skew group ring} $R*G$ is $R\otimes_{\mathbb{C}}\mathbb{C}G$ with multiplication defined by $$(r,g)(s,h)=(r.gs,gh)$$
\begin{thm}[Auslander \cite{auspure, ausrat}]
The map $R*G \to \enn_{R^G}(R)$ is an isomorphism. Moreover, $R*G$ is a noncommutative resolution of $R^G$, i.e.\ $R*G$ has finite global dimension, and $R$ is a MCM $R^G$-module admitting a free summand.    
\end{thm}
In particular, $R*G$ comes with an idempotent $e$ such that $e(R*G)e \cong R^G$. Reiten and Van den Bergh gave a Morita equivalent version of $R*G$ that can be constructed purely from the McKay graph.
\begin{defn}
    Let $Q$ be a finite quiver. The \textbf{double quiver} $\bar Q$ of $Q$ is the quiver obtained by adding, for every arrow $a:i\to j$ of $Q$, an opposite arrow $a^*:j \to i$. The \textbf{preprojective algebra} of $Q$ is the algebra $\Pi(Q)$ obtained as the quotient of $\bbC\bar{Q}$ by the relation $\sum_a[a,a^*]=0$.
\end{defn}
\begin{exer}
    Let $T$ be a finite graph, and arbitrarily choose orientations of the edges to obtain a quiver $Q$. Show that both $\bbC\bar{Q}$ and $\Pi(Q)$ are independent of the choice of orientation. We henceforth write $\bbC\bar{T}$ and $\Pi(T)$ and in what follows we will blur the distinction between graphs and quivers.
\end{exer}

\begin{thm}[Reiten--Van den Bergh \cite{reitenvdb}]
Let $R^G$ be a Kleinian singularity, and $T$ the corresponding extended Dynkin graph. Then
\begin{itemize}
    \item $R*G$ is Morita equivalent to $\Pi(T)$.
    \item $e_0\Pi(T)e_0\cong R^G$.
\end{itemize}
\end{thm}
In particular, since $R*G$ has finite global dimension, so does $\Pi(T)$.
\begin{exer}
    The extended $A_1$ Dynkin graph is of the form $T=\begin{tikzcd}
        0 \ar[r, bend left=30, no head] & 1 \ar[l, bend left=30, no head]
    \end{tikzcd}$. Check that $e_0\Pi(T)e_0$ is the $A_1$ singularity.
\end{exer}

Let us use the preprojective algebra to define a noncommutative deformation of $R^G$. This was first done by  Crawley-Boevey and Holland \cite{cbh}.
\begin{defn}
    Let $Q$ be a finite quiver with vertex set $\{$1,\ldots, n$\}$. A \textbf{weight} is a vector $\lambda\in\bbC^n$; we think of a weight as a map $Q_0 \to \bbC$ assigning vertex $i$ a number $\lambda_i$. If $\lambda$ is a weight, the corresponding \textbf{deformed preprojective algebra} $\Pi^\lambda(Q)$ is the quotient of $\bbC\bar Q$ by the\textbf{ deformed preprojective relations} $\sum_ae_i[a,a^*]e_i=\lambda_ie_i$, one for each $i$.
\end{defn}
 Clearly $\Pi^0(Q)\cong \Pi(Q)$.
\begin{defn}
    If $Q$ is an extended Dynkin quiver and $\lambda$ is a weight, the corresponding \textbf{noncommutative deformation} of $R^G$ is the ring $\mathcal{O}^\lambda \coloneqq e_0\Pi^\lambda(Q)e_0$.
\end{defn}
If we take $\lambda =0$ then we get $\mathcal{O}^0\cong R^G$.
\begin{rmk}
    In what sense is $\mathcal{O}^\lambda$ a noncommutative deformation of $R^G$? We describe another method to obtain $\mathcal{O}^\lambda$, also appearing in \cite{cbh}. The McKay correspondence gives an isomorphism $Z(\bbC G)\cong \bbC^{n+1}$ between the centre of the group ring and the semisimple ring of weights on the extended Dynkin graph. Under this isomorphism, we put $\mathcal{S}^\lambda \coloneqq \frac{\bbC\langle x,y\rangle*G}{[x,y]-\lambda}$, so that $\mathcal{S}^0$ is the skew group ring $R*G$. If $e=\frac{1}{|G|}\sum_{g\in G}g$ denotes the average of the group elements, then we in fact have $e\mathcal{S}^\lambda e \cong \mathcal{O}^\lambda$. Moreover, both   $\mathcal{S}^\lambda$ and  $\mathcal{O}^\lambda$ are filtered by path length, and their associated graded algebras are $R*G$ and $R^G$ respectively. Hence we view $\mathcal{S}^\lambda$ and $\mathcal{O}^\lambda$ as \textbf{PBW deformations} of $R*G$ and $R^G$ respectively.
\end{rmk}
\begin{prop}
    $\Pi^\lambda(Q)$ is a noncommutative resolution of $\mathcal{O}^\lambda$.
\end{prop}
When $\lambda=0$ this reduces to Auslander's result on the skew group ring.

\section{Singularity category computations}
 We follow \cite[\S9]{kalckyang2}; the main result is originally due to Crawford \cite{crawford}. We will study the singularity category of $\mathcal{O}^\lambda$ via the method of \Cref{kycor}. Crawford proves that the singularity category of $\mathcal{O}^\lambda$ is idempotent complete, and hence \Cref{kycor} immediately yields a triangle equivalence $$\frac{\per(B)}{D_\mathrm{fd}(B)} \simeq D_\mathrm{sg}(\mathcal{O}^\lambda)$$where $B$ denotes the derived quotient $\Pi^\lambda(Q)/^{\mathbb{L}}\kern -2pt \Pi^\lambda(Q)e_0\Pi^\lambda(Q)$. In order to find such a $B$, we need to resolve $\Pi^\lambda(Q)$ as a dg algebra. Our resolution will be given by a derived version of the preprojective algebra.

\begin{defn}
   Let $Q$ be a finite quiver and $\lambda$ a weight on $Q$. Let $\tilde Q$ be the graded quiver consisting of $\bar Q$ placed in degree zero, along with a loop $t_i$ in degree $-1$ at each vertex $i$. The \textbf{derived preprojective algebra} is the dg algebra $\underline{\Pi}(Q)$ whose underlying graded algebra is $\bbC\tilde Q$, and whose differential satisfies $d(a)=d(a^*)=0$ and $d(t_i)=\sum_ae_i[a,a^*]e_i$. The \textbf{deformed derived preprojective algebra} $\underline{\Pi}^\lambda(Q)$ is the same, but the differential on $t_i$ is modified to be the sum $d(t_i)=\sum_ae_i[a,a^*]e_i-\lambda_ie_i$.
\end{defn}

\begin{rmk}
The dg algebra $\underline{\Pi}^\lambda(Q)$ is a deformed 2-Calabi--Yau completion of the path algebra $\bbC Q$, in the sense of Keller \cite{kellercy}. More generally, to construct the $n$-CY completion one puts $a^*$ in degree $2-n$ and $t_i$ in degree $1-n$; the differential remains the same. Deformed $n$-CY completions are more subtle; we remark that Ginzburg dg algebras of quivers with potential are examples of $3$-CY completions (the potential should be thought of as the deformation parameter).
	\end{rmk}
    Clearly we have an isomorphism $H^0(\underline{\Pi}^\lambda(Q))\cong {\Pi}^\lambda(Q)$, which extends to a map $\underline{\Pi}^\lambda(Q) \to {\Pi}^\lambda(Q)$.
    \begin{prop}
    If $Q$ is a finite quiver without cycles, and not a Dynkin quiver, then the map $\underline{\Pi}^\lambda(Q) \to {\Pi}^\lambda(Q)$ is a cofibrant resolution of dg algebras.
    \end{prop}
    \begin{proof}[Proof idea]
  It is easy to see that $\underline{\Pi}^\lambda(Q)$ is a cofibrant dg algebra, so we just need to check that the map is a quasi-isomorphism. When $\lambda=0$ this is a result of Hermes \cite{hermesginz}. In the general case, one puts a secondary Adams grading on $\underline{\Pi}^\lambda(Q)$ by putting $a$ and $a^*$ in degree $-1$, $e_i$ in degree $0$, and $t_i$ in degree $-2$. This secondary grading induces an increasing filtration on $\underline{\Pi}^\lambda(Q)$ by Adams path length whose associated graded is $\underline{\Pi}(Q)$ equipped with an extra grading. The associated spectral sequence computes both $H^*(\underline{\Pi}^\lambda(Q))$ and $H^*(\underline{\Pi}(Q))$, which shows that $H^*(\underline{\Pi}^\lambda(Q))$ is concentrated in degree zero, as required.
    \end{proof}
    \begin{rmk}
        The above spectral sequence argument is closely related to the fact that ${\Pi}^\lambda(Q)$ is a PBW deformation of $\Pi(Q)$.
    \end{rmk}
 	 	\begin{rmk}
 	 	If $Q$ is a finite quiver without cycles then $\Pi(Q)$ is finite dimensional precisely when $Q$ is of Dynkin type; since finite dimensional algebras cannot be 2CY, the conclusion of the proposition cannot hold when $Q$ is a Dynkin quiver. In this case, Hermes constructs a minimal model of $\underline{\Pi}(Q)$. If $Q$ is a finite quiver without cycles then $\Pi(Q)$ is Koszul precisely when $Q$ is not of Dynkin type; one can prove the proposition using this. 
 	 		\end{rmk}	
            
 	Let us return to the Kleinian situation. We are interested in the quotient

    $$B\coloneqq \underline{\Pi}^\lambda(Q)/(e_0)$$where $Q$ is an extended Dynkin quiver and $0$ the extending vertex. Let $Q'$ be the full subquiver on the vertices $1,\ldots, n$ and $\lambda'$ the corresponding weight on $Q'$. Since $B$ is obtained from $\underline{\Pi}^\lambda(Q)$ by killing all paths that pass through $0$, it follows that $B$ is isomorphic to $\underline{\Pi}^{\lambda'}(Q')$. We will now pass to an even smaller model for $B$.

	Let $Q'_{\lambda'}$ be the full subquiver of $Q'$ on those vertices $i$ with $\lambda'_i=0$. There is a natural surjection $\underline{\Pi}^{\lambda'}(Q') \to \underline{\Pi}(Q'_{\lambda'})$ which Crawford proves to be a quasi-isomorphism as long as $\lambda$ is \textbf{quasi-dominant}, meaning that each of the $\lambda'_i$ either has positive real part, or zero real part and nonnegative imaginary part. Since one may choose $\lambda$ to be quasi-dominant without affecting the isomorphism type of the noncommutative deformation, we may assume that this is the case.

    So we have a quasi-isomorphism $B\simeq \underline{\Pi}(Q'_{\lambda'})$. Observe that $Q'_{\lambda'}$ must be a disjoint union of Dynkin quivers $Q^1,\ldots, Q^s$; let $R^1, \ldots, R^s$ denote the corresponding Kleinian singularities.

    \begin{thm}\label{crawfthrm}
        There is a triangle equivalence$$D_\mathrm{sg}(\mathcal{O}^\lambda)\simeq \bigoplus_{j=1}^s D_\mathrm{sg}(R^j)$$
    \end{thm}
    \begin{proof}
        We certainly have triangle equivalences $$D_\mathrm{sg}(\mathcal{O}^\lambda)\simeq\frac{\per(\underline{\Pi}(Q'_{\lambda'}))}{D_\mathrm{fd}(\underline{\Pi}(Q'_{\lambda'}))}\simeq \bigoplus_{j=1}^s \frac{\per(\underline{\Pi}(Q^j))}{D_\mathrm{fd}(\underline{\Pi}(Q^j))}$$
        On the other hand, let $R$ be a Kleinian singularity, $Q$ the associated Dynkin graph, and $Q_+$ the extended Dynkin graph. Since $\Pi(Q_+)$ is a noncommutative resolution of $R$, we may compute the singularity category of $R$ using the derived quotient of $\Pi(Q_+)$ by $e_0$. As before, since the dg algebra $\underline \Pi(Q_+)$ resolves $\Pi(Q_+)$ and we have $\underline{\Pi}(Q_+)/(e_0) \simeq \underline{\Pi}(Q)$, the derived quotient is simply $\underline{\Pi}(Q)$, and hence we have a triangle equivalence $$D_\mathrm{sg}(R)\simeq \frac{\per(\underline{\Pi}(Q))}{D_\mathrm{fd}(\underline{\Pi}(Q))}$$Combining these equivalences proves the theorem.
    \end{proof}

    \chapter{Koszul duality}

In this chapter, we introduce the machinery of Koszul duality, before showing that it can be used to give structural results on singularity categories as well as provide a computational tool. We mostly follow \cite{positselski, leavitt, caldararutu, tu}.

\section{Bar and cobar constructions}
We provide a brief overview of the bar and cobar constructions and the module-comodule version of associative Koszul duality, roughly following the presentation in \cite{reflexivity}. For a more comprehensive treatment see \cite{positselski, lodayvallette}. In this section, $k$ will denote a commutative semisimple ring.
\begin{defn}
A \textbf{dg coalgebra} is a complex $C$ together with a \textbf{comultiplication} $\Delta:C\to C\otimes C$ and a \textbf{counit} $\eta:C \to k$ satisfying the following axioms:
\begin{enumerate}
    \item[]Coassociativity: $(\Delta\otimes \mathrm{id_C})\Delta \ = \  (\mathrm{id_C}\otimes \Delta)\Delta$ as maps $C\to C^{\otimes 3}$. 
    \item[]Counitality: $(\eta\otimes \mathrm{id_C})\Delta \ = \  (\mathrm{id_C}\otimes \eta)\Delta = \mathrm{id}_C$.
\end{enumerate}
 We also ask that the comultiplication and counit be chain maps. A dg coalgebra is \textbf{coaugmented} if $\eta$ admits a retract $k\to C$ which is a morphism of dg coalgebras. In this case, $\bar C\coloneqq \ker(\eta)$ is a noncounital subcoalgebra of $C$. A dg coalgebra $C$ is \textbf{conilpotent} if it is coaugmented, and for every $c\in \bar C$ there is an $N$ such that $\Delta^N(c)=0$.
\end{defn}
\begin{exer}
    Write explicit equations expressing that $\Delta$ and $\eta$ are chain maps (the first says that $d$ is a \textbf{coderivation} for $\Delta$).
\end{exer}
\begin{exer}\label[exer]{ldcogex}\hfill
\begin{enumerate}
    \item  Let $(C,\Delta,\eta)$ be a dg coalgebra. Show that its linear dual $C^*$ is a dg algebra with multiplication $\Delta^*$ and unit $\eta^*$.
    \item Let $A$ be a finite dimensional dg algebra. Show that $A^*$ is a dg coalgebra in the same way. Show that $A$ is augmented precisely when $C$ is coaugmented. In this situation, show that the corresponding maximal ideal of $A$ is nilpotent precisely when $C$ is conilpotent.
    \item Give an example of a dg algebra $A$ such that the multiplication on $A$ does not dualise to a comultiplication on $A^*$.
\end{enumerate}
\end{exer}
If $V$ is a complex, its \textbf{tensor coalgebra} is the dg coalgebra $$T^c(V)\coloneqq k\oplus V \oplus V^{\otimes 2}\oplus \cdots$$ with comultiplication given by the \textbf{deconcatenation coproduct}, which sends a tensor $v_1\otimes\cdots\otimes v_n$ to the sum $\sum_i(v_1 \otimes \cdots \otimes v_i)\otimes(v_{i+1}\otimes\cdots\otimes v_n)$. 
\begin{exer}\hfill
\begin{enumerate}
    \item Check that $T^cV$ is a conilpotent dg coalgebra.
    \item Show that the functor $T^c$ is right adjoint to the forgetful functor from conilpotent dg coalgebras to complexes.
\end{enumerate}
   
\end{exer}
\begin{defn}
    Let $A$ be an augmented dg algebra. The \textbf{bar construction} $BA$ is the dg coalgebra whose underlying graded coalgebra is the same as that of $T^c(\bar A [1])$. The differential $d$ on $BA$ is the sum $d=d_I+d_E$, where $d_I$ denotes the usual `internal' differential of $T^c(\bar A [1])$ and $d_E$ denotes the `external' differential, defined by sending $a_1\otimes\cdots \otimes a_n$ to the signed sum of the tensors $a_1\otimes \cdots \otimes a_ia_{i+1}\otimes\cdots\otimes a_n$.
\end{defn}
\begin{exer}
    *Use the Koszul sign rule to work out the correct signs appearing in $d_E$, and use this to check that $d_I+d_E$ is a differential.
\end{exer}

\begin{defn}
    Write $A^!\coloneqq (BA)^*$ for the linear dual of the bar construction. By \Cref{ldcogex}, $A^!$ is a dg algebra. We call $A^!$ the \textbf{Koszul dual} of $A$.
\end{defn}

\begin{exer}\label[exer]{SqZEx}
Let $A$ be the square-zero extension $k[\varepsilon]/\varepsilon^2$, with $\epsilon$ in degree $1-n$. Show that $A^!\simeq k\llbracket x \rrbracket$ with $x$ placed in degree $n$. (Hint: if $n\neq 0$ then $k\llbracket x \rrbracket\cong k[x]$.)
\end{exer}

The bar construction has a dual construction, the \textbf{cobar construction}, which sends conilpotent dg coalgebras to dg algebras. Briefly, let $C$ be a conilpotent dg coalgebra. The cobar construction of $C$ is the dg algebra $\Omega C$ whose underlying graded algebra is $T(\bar C[-1])$, the tensor algebra on the shifted coaugmentation ideal of $C$. The differential combines the natural internal differential on the tensor algebra with the comultiplication on $C$.

\begin{thm}[Koszul duality, cf.\ {\cite{positselski}}]\hfill
\begin{enumerate}
\item $\Omega$ and $B$ form an adjunction $\Omega\colon \mathbf{dgCog}^\mathrm{conil} \longleftrightarrow \mathbf{dgAlg}^\mathrm{aug}\colon B$.
\item Let $A$ be an augmented dg algebra. The counit $\Omega BA \to A$ is a quasi-isomorphism of algebras.
\item Let $C$ be a conilpotent dg coalgebra. The unit $C \to B\Omega C$ is a weak equivalence of coalgebras (i.e.\ is sent to a quasi-isomorphism by $\Omega$).
\end{enumerate}
\end{thm}

\begin{rmk}
    In fact, Positselski proves that there exist model structures on the categories of augmented dg algebras and conilpotent dg coalgebras making $\Omega \dashv B$ into a Quillen equivalence. A weak equivalence of dg algebras is precisely a quasi-isomorphism. The weak equivalences of coalgebras are created by $\Omega$; every weak equivalence is a quasi-isomorphism but the converse is not true.
\end{rmk}

If $C$ is a dg coalgebra, a (right) \textbf{dg-$C$-comodule} is a complex $V$ together with a coaction map $\rho\colon V\to V\otimes C$ such that $(\mathrm{id}_V\otimes \Delta)\rho = (\rho\otimes \mathrm{id}_C)\rho$. A \textbf{$C$-colinear map} between two dg-$C$-comodules $U,V$ is a map of complexes $U\to V$ which is compatible with the coactions. These define an abelian dg category $C\mathbf{-Comod}$ of dg-$C$-comodules. We let $\mathrm{Hot}(C)\coloneqq H^0(C\mathbf{-Comod})$ denote the corresponding homotopy category; it is a triangulated category in the usual way. The subcategory $\mathbf{CoAcy}(C)$ of \textbf{coacyclic} $C$-comodules is then defined to be the smallest thick subcategory of $\mathrm{Hot}(C)$ containing the totalisations of short exact sequences of $C$-comodules and closed under all coproducts. The \textbf{coderived category} of $C$ is the Verdier quotient $\dco(C)\coloneqq \mathrm{Hot}(C)/\mathbf{CoAcy}(C)$. Via taking dg quotients one can also enhance $\dco(C)$ to a pretriangulated dg category, and we will often regard it as such.
\begin{rmk}
One can also recover $\dco(C)$ as the homotopy category of a model structure on the category of dg-$C$-comodules; the cofibrations are the injections and the weak equivalences are the maps with coacyclic cone. In particular every comodule is cofibrant. 
\end{rmk}

\begin{rmk}
A weak equivalence between dg comodules is a quasi-isomorphism, but the converse is not true. One can check that there is a well-defined functor $\dco(C)^\mathrm{op} \to {\mathcal{D}}(C^*)$ which sends a comodule $N$ to its linear dual $N^*$. In general this functor need not be full, faithful, or essentially surjective.
\end{rmk}

\begin{thm}[Module-comodule Koszul duality, cf.\ {\cite{positselski}}]\label[thm]{ModCoMod}
Let $A$ be an augmented dg algebra and $C$ a conilpotent dg coalgebra. There there are quasi-equivalences 
\begin{align*}
    D(A) &\simeq \dco(BA)\\
    \dco(C)&\simeq D(\Omega C)
\end{align*}
The first sends $A$ to $k$ and $k$ to $BA$. The second sends $C$ to $k$ and $k$ to $\Omega C$.
\end{thm}
Module-comodule Koszul duality can be used to show that the dual bar construction computes derived endomorphisms:
\begin{prop}\label[prop]{REndProp}
    Let $A$ be an augmented dg algebra. Then there is a quasi-isomorphism $A^!\simeq \R\hom_A(k,k)$ of dg algebras.
\end{prop}
\begin{proof}[Proof sketch]
    Put $C=BA$. By module-comodule Koszul duality, there are quasi-isomorphisms of dg algebras
    $$\R\mathrm{End}_A(k) \simeq \R\mathrm{End}_{\dco(C)}(C)$$
    Since $C$ is an injective $C$-comodule, it is fibrant, and we can compute its derived endomorphisms as $\R\mathrm{End}_{\dco(C)}(C) \simeq \mathrm{End}_C(C) \cong C^*$, as desired.
\end{proof}

\begin{ex}
    Use \Cref{REndProp} along with \Cref{SqZEx} to show that, if $k[x]$ denotes the graded polynomial ring, with $x$ placed in arbitrary degree, then we have $k[x]^{!!}\simeq k\llbracket x \rrbracket$.
\end{ex}

        \section{Finite dimensional algebras}
        The purpose of this section is to give a description, due to Keller and Wang, of the singularity category of a finite dimensional algebra. We follow \cite{leavitt}.
 		
 		\begin{thm}\label{kwtheorem}
        Let $\mathbb{F}$ be a perfect field and let $A$ be a finite dimensional $\mathbb{F}$-algebra. Let $k$ be the maximal semisimple quotient of $A$. Then, working over the base ring $k$, there is an equivalence $$D_\mathrm{sg}(A^\mathrm{op})^\mathrm{op} \simeq \frac{\per(\Omega (A^*))}{\mathbf{thick}(k)}\simeq \frac{\per(A^!)}{\mathbf{thick}(k)}.$$
  			\end{thm}
  		\begin{proof}[Proof idea]
        Since $\mathbb{F}$ was perfect, $A$ is a $k$-algebra by the Wedderburn--Malcev theorem. Put $C=A^*$ for brevity; since $A$ was nilpotent over $k$ it follows that $C$ is conilpotent. There is a natural map $\Omega(C)\to A^!$ and one can check that it is a quasi-isomorphism since $C$ is finite dimensional. So in what follows we focus on the first equivalence.

        Since $A$ is finite dimensional, there is a functor $A\text{-}\mathbf{Mod} \to C\text{-}\mathbf{Comod}$: an action map $M\otimes A \to M$ corresponds across the hom-tensor adjunction to a coaction map $M \to M\otimes C$ making $M$ into a $C$-comodule. In fact, this extends to an equivalence $K(\mathrm{Inj}(A))\to \dco(C)$. If we choose a fully faithful injective resolution functor $D^b(A) \into K(\mathrm{Inj}(A))$, by composition we obtain a fully faithful triangle functor $\iota:D^b(A) \into \dco(C)$. Hence $\iota$ induces triangle equivalences $$\frac{D^b(A)}{\mathbf{thick}(A^*)}\simeq\frac{\mathbf{thick}_{D^b(A)}(k)}{\mathbf{thick}_{D^b(A)}(A^*)}\simeq \frac{\mathbf{thick}_{\dco(C)}(\iota(k))}{\mathbf{thick}_{\dco(C)}(\iota(A^*))}\simeq \frac{\mathbf{thick}_{\dco(C)}(k)}{\mathbf{thick}_{\dco(C)}(C)}$$
        Since linear duality is an equivalence $D^b(A^\mathrm{op}) \to D^b(A)^\mathrm{op}$, it follows that the left hand side is exactly $D_\mathrm{sg}(A^\mathrm{op})^\mathrm{op}$. For the right hand side, \Cref{ModCoMod} provides a triangle equivalence $$\frac{\mathbf{thick}_{\dco(C)}(k)}{\mathbf{thick}_{\dco(C)}(C)}\simeq \frac{\mathbf{thick}_{D(\Omega C)}(\Omega C)}{\mathbf{thick}_{D(\Omega C)}(k)}\simeq \frac{\per(\Omega C)}{\mathbf{thick}(k)}$$ and putting the above equivalences together yields the required statement.
  			\end{proof}
            \begin{rmk}
                One can drop the perfect field requirement at the cost of assuming that $A$ is nilpotent over a finite dimensional semisimple $\mathbb{F}$-algebra; for example path algebras of quivers with admissible relations are of this form.
            \end{rmk}

 	\section{Hochschild homology via Koszul duality}
    Matrix factorisations can be viewed as certain modules over a \textit{curved} dg algebra. In this section, we explain this perspective, before using some results in curved Koszul duality to give an alternate proof of Dyckerhoff's result (\Cref{dythm}) on the Hochschild homology of matrix factorisation categories. We mostly follow the papers \cite{caldararutu, tu}.

 	\begin{defn}
 		A \textbf{curved dg algebra} is the data of
 		\begin{itemize}
 			\item A graded algebra $A$
 			\item A degree $1$ derivation $d:A \to A$
 			\item An element $h\in A^2$
 		\end{itemize}
 		such that
 		\begin{itemize}
 			\item $d(h)=0$
 			\item $d^2(x)=[h,x]=hx-xh$ for all $x\in A$.
 		\end{itemize}
 		We call $d$ the \textbf{differential} and $h$ the \textbf{curvature element}.
 	\end{defn}
 	\begin{ex}
 		If $A$ is a dg algebra, then it is naturally a curved dg algebra with zero curvature. 
 	\end{ex}
 	\begin{ex}
 		If $A$ is a graded algebra and $h\in A$ is a degree $2$ central element, then the pair $(A,h)$ is a curved dg algebra with zero differential (a.k.a.\ a \textbf{curved graded algebra}).
 	\end{ex}

 A \textbf{morphism} of curved algebras $A \to B$ is a pair $(f,b)$ where $f:A \to B$ is a map of graded algebras, and $b \in B$ is a degree 1 element satisfying the formulas
 \begin{itemize}
 	\item $f(da)=d(fa) + [b,fa]$
 	\item $f(h_A) = h_B + db + b^2$.
 \end{itemize} 
 Morphisms compose by putting $(g,b)(f,a)=(gf, b+g(a))$. Observe that the inclusion of dg algebras into curved dg algebras is faithful but not full.

 A \textbf{dg module} over a curved dg algebra $A$ is a graded $A$-module $M$ with a differential $d$ of degree $1$ satisfying the Leibniz rule with respect to the $A$-action, and such that $d^2(m)=hm$. Note that $A$ need not be a module over itself. DG modules assemble into a dg category $\mathbf{Mod-}A$, with the hom-complexes defined exactly as in the uncurved case. We let $\mathbf{Tw}(A)$ denote the full subcategory on the \textbf{finitely generated twisted modules}, i.e.\ those $A$-modules whose underlying graded $A$-modules are free of finite rank over the underlying graded algebra of $A$. This is a full pretriangulated dg subcategory.

     	\begin{ex}\label[ex]{curvrmk}
 		Let $A$ be a commutative ring and let $\sigma\in A$ be any element. Define a curved $\Z/2$-graded algebra $A_\sigma$ by:
 		\begin{itemize}
 			\item The underlying graded algebra of $A_\sigma$ is $A$, concentrated in even parity.
 			\item $A_\sigma$ has zero differential and the curvature element is given by $\sigma$.
 		\end{itemize}
 		Then $\mathbf{Tw}(A_\sigma)$ is precisely the $\Z/2$-graded dg category $\mathbf{MF}^\mathrm{dg}(A,\sigma)$. This gives a quick proof that the homotopy category $\mathbf{MF}(A, \sigma)$ is in fact a triangulated category. 
 	\end{ex}
    There is a completely parallel definition of \textbf{curved dg coalgebra}: a curved dg coalgebra is a graded coalgebra $C$ with a coderivation $d$ and a curvature functional $h:C \to k$ satisfying some axioms. Similarly, there is a notion of \textbf{finitely generated twisted comodule} over a curved dg coalgebra. Just like in the uncurved case, the linear dual of a curved dg coalgebra is a curved dg algebra, and one has the following proposition:
    \begin{prop}
        Let $C$ be a curved dg coalgebra. Then there is a quasi-equivalence $$\mathbf{Tw}(C)^\mathrm{op}\simeq \mathbf{Tw}(C^*).$$
    \end{prop}
    \begin{ex}\label[ex]{curvrmk2}
        Let $A=k\llbracket x_1,\ldots, x_n\rrbracket$ be the power series ring, so that $A$ is the linear dual of the cosymmetric coalgebra $C$ on $n$ variables. If $\sigma\in A$ then $\sigma$ dualises to a functional on $C$, which gives $C$ the structure of a curved $\Z/2$-graded coalgebra which we denote $C_\sigma$. We have $C_\sigma^*\cong A_\sigma$ and hence using \Cref{curvrmk} we obtain a quasi-equivalence $\mathbf{Tw}(C_\sigma)^\mathrm{op}\simeq \mathbf{MF}^\mathrm{dg}(A,\sigma)$.
    \end{ex}
A curved dg coalgebra $C$ has a \textbf{cobar construction} $\Omega C$, which is a (not necessarily augmented) dg algebra. The definition of $\Omega C$ is the same as the usual cobar construction, but the differential acquires an extra term from the curvature functional. If $C$ is uncurved then $\Omega$ reduces to the usual cobar construction, and in particular $\Omega C$ is augmented. Similarly, if $A$ is a dg algebra, it has a bar construction $BA$ which is a conilpotent curved dg coalgebra, and $\Omega BA \simeq A$. If $A$ was augmented then $BA$ is uncurved. Defining a bar construction for curved dg algebras is possible, but more subtle, since one can no longer expect $BA$ to be conilpotent. See e.g.\ \cite{GKD} for a more thorough discussion.

\begin{thm}[Curved Koszul duality]
    Let $C$ be a coaugmented curved dg coalgebra. Then there is a quasi-equivalence of pretriangulated dg categories $\mathbf{Tw}(C) \simeq \mathbf{Tw}(\Omega C)$. When $C$ is in addition conilpotent, there is a quasi-equivalence $\mathbf{Tw}(\Omega C)\simeq \per(\Omega C)$. In particular, when $A$ is a dg algebra there is a quasi-equivalence $\per(A)\simeq \mathbf{Tw}(BA)$.
\end{thm}
\begin{ex}\label[ex]{curvrmk3}
    Let $C_\sigma$ be the curved $\Z/2$-graded coalgebra of \Cref{curvrmk2}. Curved Koszul duality yields a quasi-equivalence $\mathbf{MF}^\mathrm{dg}(A,\sigma)\simeq \per(\Omega C_\sigma)^\mathrm{op}$. One can in fact show that $\Omega C_\sigma\simeq P_\sigma^\mathrm{op}$, the (opposite of the) dg algebra of polynomial differential forms constructed by Dyckerhoff. In particular, we have $P_\sigma^!\simeq A_\sigma$. In this sense, the equivalence of \Cref{dyccor} is in fact an example of curved Koszul duality.
\end{ex}

 	If $A$ is a curved dg algebra, it has a Hochschild complex $\mathrm{HH}^\bullet(A)$. As a graded vector space, $\mathrm{HH}^\bullet(A)$ is the product $\prod_{i=0}^\infty\hom(A^{\otimes i},A)$. The differential sends a degree $k$ cochain $f$ to the sum
 	
 	$$(a_1\ldots,a_l)\mapsto \sum_{\stackrel{j}{k\leq l}} (-1)^{j+\vert a_1\vert +\cdots +\vert a_j\vert} m_{l-k+1}(a_{1}, \ldots ,f(a_{j+1},\ldots, a_{j+k}),\ldots, a_{l}) $$
 	$$+ \sum_i (-1)^{\vert f\vert+\vert a_1\vert +\cdots +\vert a_i\vert} f(a_1,\ldots,m_{l-k+1}(a_{i+1},\ldots, a_{i+l-k+1}),\ldots, a_l)$$
 	where $m_0()=h$, $m_1(a)=da$, and $m_2(a,b)=ab$. Note that this is like the usual Hochschild differential, but also incorporates a curvature term. There is a similar definition for $\mathrm{HH}_\bullet(A)$ as the totalisation of the bigraded vector space $A\otimes A^i$ equipped wih a similar differential.
    \begin{prop}
 		Let $A$ be a curved graded algebra with nonzero curvature. Then both $\mathrm{HH}^*(A)$ and $\mathrm{HH}_*(A)$ vanish.
 	\end{prop}
 	\begin{proof}[Proof idea]
 		A spectral sequence argument reduces to the case where $A$ is a curved algebra with zero multiplication, and one can compute $\mathrm{HH}(A)$ explicitly in this situation.
 	\end{proof}
    Hence the usual notion of Hochschild co/homology is not a good invariant for our current purposes. We introduce a slightly modified version.
    \begin{defn}[\cite{pphoch}]
 		Let $A$ be a curved dg algebra. The \textbf{compactly supported Hochschild cohomology complex} is the graded vector space $\mathrm{HH}^\bullet_c(A)\coloneqq \bigoplus_i \hom(A^{\otimes i}, A)$ equipped with the above Hochschild differential. Similarly, the  \textbf{Borel--Moore Hochschild homology complex} is the graded vector space $\mathrm{HH}_\bullet^{\mathrm{BM}}(A)\coloneqq \prod_i (A \otimes A^i)$ equipped with the usual Hochschild differential.
 	\end{defn}

    We now describe an alternate method to compute $\mathrm{HH}_*(\mathbf{MF}^\mathrm{dg}(A,\sigma))$. Following \Cref{curvrmk3} we obtain a quasi-isomorphism 
    $$\mathrm{HH}_\bullet(\mathbf{MF}^\mathrm{dg}(A,\sigma))\simeq \mathrm{HH}_\bullet(\Omega C_\sigma)$$
    using the derived Morita invariance of $\mathrm{HH}_\bullet$. Tu then proves that $\mathrm{HH}_*(\Omega C_\sigma)$ is naturally isomorphic to the linear dual of $\mathrm{HH}_*^{\mathrm{BM}}(A_\sigma)$, using the intermediate notion of the Hochschild complex of a curved dg coalgebra. This reduces the computation of $\mathrm{HH}_*(\mathbf{MF}^\mathrm{dg}(A,\sigma))$ to the computation of $\mathrm{HH}_*^{\mathrm{BM}}(A_\sigma)$, which one can do directly, following \cite{caldararutu}. We remark that a generalisation of this computation appears in \cite{HHcurved}.

\backmatter

\chapter{Further reading}

 	In this final section, we provide a few small glimpses into the large range of results about singularity categories that we failed to mention in earlier chapters. No more than a brief summary will be given; in particular we do not claim to provide anything close to a comprehensive list of references.

\section*{Singularity categories for schemes}

Let $X$ be a quasi-separated noetherian $k$-scheme. Just as in the case of $k$-algebras, one can define $D_\mathrm{sg}(X) \coloneqq D^b(\mathrm{coh}X) / \per(X)$, and when $X$ is affine this agrees with $D_\mathrm{sg}(k[X])$. This definition -- as well as the name ``singularity category'' for this object -- is due to Orlov \cite{orlovtri}. One key fact about $D_\mathrm{sg}(X)$ is that it can be computed on a formal neighbourhood of $\mathrm{Sing}X$. In the following theorem, we refrain from giving the definition of an ELF scheme; see the original reference \cite{orlovcomplete} or the good overview \cite{symons} for the details. We will, however, note that a quasiprojective variety defined over a perfect field is ELF.
\begin{thm}[Orlov {\cite{orlovcomplete}}]
    Let $X,Y$ be ELF schemes and denote by $\mathfrak{X}, \mathfrak{Y}$ the formal completions of $X,Y$ at their singular loci. If $\mathfrak{X}\cong\mathfrak{Y}$ as formal schemes then there is a triangle equivalence $D_\mathrm{sg}(X)^\omega \simeq D_\mathrm{sg}(Y)^\omega$.
\end{thm}
We note that the theorem fails if one drops the idempotent completions; indeed for a counterexample one may take $X$ to be the cubic and $Y$ to be the coordinate axes in the plane. We obtain the useful corollary:
\begin{cor}
   Let $X$ be an ELF scheme with singular locus a finite set of isolated points. Then there are triangle equivalences $$D_\mathrm{sg}(X)^\omega\quad\simeq\quad\bigoplus_{p\in \mathrm{Sing}X}D_\mathrm{sg}(\mathcal{O}_{X,{p}})^\omega\quad\simeq\quad \bigoplus_{p\in \mathrm{Sing}X}D_\mathrm{sg}(\widehat{\mathcal{O}}_{X,p})^\omega$$and if in addition $X$ has complete local Gorenstein singularities then we also have a triangle equivalence $$D_\mathrm{sg}(X)^\omega\quad\simeq\quad\bigoplus_{p\in \mathrm{Sing}X}D_\mathrm{sg}(\widehat{\mathcal{O}}_{X,p}).$$
\end{cor}
Observe that similar `block decompositions' appear in \Cref{bkthrm} and \Cref{crawfthrm}. In general the difference between $D_\mathrm{sg}(X)$ and its idempotent completion is measured by $K_{-1}(\per X)$, as in \cite{orlovcomplete}. We note that other localisation results, as well as a noncommutative version, appear in \cite{ChenUnifying}.

One can often globalise other constructions of the singularity category: if $X$ is a Gorenstein scheme (or stack) then Buchweitz's description of the singularity category globalises, and one can obtain a description of $D_\mathrm{sg}(X)$ in terms of MCM sheaves \cite{pvstack, lugor}. When $X$ is the zero scheme of a section of a line bundle, one can obtain a similar description of $D_\mathrm{sg}(X)$ in terms of global matrix factorisations \cite{OrlovMFglobal, LDMFglobal, epsings}. We note that Efimov and Positselski also give a relative version \cite{epsings}.

Orlov showed that the bounded derived category of a Calabi--Yau complete intersection is the graded singularity category of its affine cone; in more general situations one category embeds in the other \cite{OrlovCY}. A noncommutative version for dg algebras was also given by Brown and Sridhar \cite{BSorlov}. Shipman, following work of Segal, was able to give a description of Orlov's theorem in terms of global matrix factorisations \cite{segalGIT, shipman}. This is closely related to a theorem of Isik, which states that if $Y$ is a variety given as the zero scheme of a section of a vector bundle $E$ on a smooth variety, then $Y$ is derived equivalent to the equivariant singularity category of a hypersurface in the total space of $E^*$ \cite{isik}.

The descriptions of the Hochschild theory of matrix factorisation categories, as well as the closely related ``K\"unneth theorems'' of \Cref{hochsect}, were globalised by Preygel \cite{preygel} (see also \cite{bznp}). A related result appears in \cite{tvkunneth} where it is applied to prove a version of the Bloch conductor conjecture using the connection between singularity categories and vanishing cohomology established in \cite{brtv}.

\section*{Singularity categories in mirror symmetry}

Equivariant singularity categories appear on both sides of the conjectural Homological Mirror Symmetry correspondence. A (gauged) \textbf{affine Landau--Ginzburg model} is the data of a polynomial $w\in k[x_1,\ldots, x_n]$ with an isolated singularity at the origin together with a group $G$ of symmetries of $w$; $G$ is known as the \textit{gauge group}. Associated to an affine LG model is an equivariant matrix factorisation category $\mathbf{MF}(\mathbb{A}^n, G,w)$ known as the \textbf{derived category of branes}; the connection with curved algebras appears early in the HMS literature \cite{konHMS, orlovtri, caldararutu, segalLGMod}. One can also consider non-affine LG models of the form $(X,G,w)$ where $X$ is a smooth variety, $G$ a reductive group acting on $X$, and $w$ a $G$-equivariant function on $X$; see e.g.\ \cite{FKfrac} for a good mathematical introduction. In this setting one typically thinks of $\mathbf{MF}(X, G,w)$ as a deformation of $D^b(X)$ along the `superpotential' $w$. We remark that the above notions are closely related to the concept of a \textbf{gauged linear sigma model} \cite{glsm}.

\textbf{Landau--Ginzburg mirror symmetry} predicts that, for a certain class of polynomials $w$ called \textit{invertible}, one can associate to a pair $(w,G)$ a mirror pair $(w^T,G^T)$ such that certain enumerative invariants (the FJRW  invariants) of $\mathbf{MF}(\mathbb{A}^n, G,w)$ agree with certain other invariants (the Saito--Givental theory) of $\mathbf{MF}(\mathbb{A}^n, G^T,w^T)$. One expects $w^T$ to be the combinatorially defined \textit{Berglund--H\"ubsch transpose} of $w$. The related \textbf{LG/CY correspondence} states that, in certain settings, $\mathbf{MF}(\mathbb{A}^n, G,w)$ should actually be equivalent to some $A_\infty$-category associated to the Calabi--Yau orbifold $[X_w/G]$, where the hypersurface $X_w\coloneqq \{w=0\}$ sits inside a weighted projective space. Whether the $A_\infty$-category is a derived category of sheaves or a Fukaya category of Lagrangians depends on whether we are on the A- or B-sides of the mirror; on the A-side the relevant enumerative invariant is the \textit{orbifold Gromov--Witten theory}. Moreover, the LG/CY correspondence should interchange LG mirror symmetry and the more classical CY mirror symmetry. See e.g.\ \cite{lgcy, lumoduli, habermann} or the survey \cite{fjrwsur} and the references therein for a far more detailed introduction.

 		\section*{Big singularity categories}

        Singularity categories do not usually have infinite co/products, since bounded derived categories do not. The first to consider a `big' singularity category seems to have been Krause, who associated to any noetherian Grothendieck abelian category $\mathscr{A}$ its \textbf{stable derived category} $K_\mathrm{ac}(\mathrm{Inj}\mathscr{A})$, the homotopy category of acyclic complexes of injective objects.
        \begin{thm}[Krause {\cite{krausestab}}]
            There is a recollement (cf.\ \Cref{recoll})
             $$\begin{tikzcd}[column sep=huge]
K_\mathrm{ac}(\mathrm{Inj}\mathscr{A}) \ar[r]& K(\mathrm{Inj}\mathscr{A})\ar[l,bend left=10]\ar[l,bend right=10]\ar[r] & D(\mathscr{A})\ar[l,bend left=10]\ar[l,bend right=10]
\end{tikzcd}$$which yields, after localisation, a triangle equivalence $$K_\mathrm{ac}(\mathrm{Inj}\mathscr{A})^c \simeq D_\mathrm{sg}(\mathrm{noeth}\mathscr{A})$$ identifying the compact objects in $K_\mathrm{ac}(\mathrm{Inj}\mathscr{A})$ with the singularity category of noetherian objects in $\mathscr{A}$. 
        \end{thm}
        When $\mathscr{A}=\mathbf{Mod-}X$ for $X$ a noetherian separated scheme or ring, then $D_\mathrm{sg}(\mathrm{noeth}\mathscr{A})$ is the usual singularity category of $X$. Note the similarity between $K_\mathrm{ac}(\mathrm{Inj}\mathscr{A})$ and the category $K_\mathrm{ac}(\mathrm{proj}A)$ appearing in the proof of \Cref{buchthm}. Observe also that when $A$ is a finite dimensional algebra, the category $K(\mathrm{Inj}A)$ already appears in the proof of \Cref{kwtheorem}, as the coderived category of the coalgebra $A^*$. In fact, any ring $R$ has a coderived category $\dco(R)$, and Positselski gives an equivalence $K_\mathrm{ac}(\mathrm{Inj}R)\simeq \dco(R)$ for any noetherian $R$ \cite{positselski}. Becker lifts these equivalences to the level of model structures \cite{becker}.

        Positselski's arguments hold in greater generality, in particular for many kinds of dg rings \cite[3.7]{positselski}, and allows one to construct `singularity categories' for them as categories of the form $\dco(R)^c$. We note that for unbounded dg rings, $\per R \not\subseteq D^b(R)$, so that the traditional definition need not make sense. Singularity categories for dg rings, and, more generally, ring spectra, were considered in \cite{gsmorita} which also established a connection with Koszul duality.

        \section*{Alternate models for singularity categories}
If $A$ is a homologically smooth dg algebra, Amiot's \textbf{generalised cluster category} $\mathrm{cl}(A)$ is defined to be the quotient $\frac{\per (A)}{D_\mathrm{fd}(A)}$ \cite{amiotcluster}. Note that the definition is formally `inverse' or `opposite to' that of the singularity category $\frac{D_\mathrm{fd}(B)}{\per (B)}$ of a finite dimensional (dg) algebra $B$, and indeed the two are often permuted by Koszul duality. So one can often view singularity categories as generalised cluster categories, and vice versa - see e.g.\ \cite{AIRcluster, gsmorita, huakeller, me, bsingcats, bgcosg} or the proof of \Cref{kwtheorem} (which is ultimately from \cite{leavitt}). In fact, $\mathrm{cl}(A)$ is sometimes known as the \textit{cosingularity category} of $A$. When $A$ is the Ginzburg dg algebra associated to a Jacobi-finite quiver with potential, Amiot additionally proves that $\mathrm{cl}(A)$ is $2$-Calabi--Yau \cite{amiotcluster}.

Some explicit models for singularity categories, constructed using noncommutative crepant resolutions, were given in \cite{dtdvvdb}; we note here also a link to the results of \Cref{rschapt}. This also allowed a new proof of the main result of \cite{AIRcluster}. Related constructions have been given in the non-isolated setting in \cite{kalckyang2}, and recently in the higher-dimensional setting in \cite{liumckay}. Other explicit models, using dg Leavitt path algebras, appear in  \cite{leavitt}.

The explicit models of \cite{dtdvvdb} were used in \cite{tabuadaKtheory} to compute the $K$-theory of the singularity categories of certain isolated quotient singularities. Whilst there do not seem to be many computations of $K_*(D_\mathrm{sg}(X))$ in the literature, some also appear in \cite{bdk} and \cite{PavicShinder}; homotopy $K$-theory is treated in \cite{GratzStevenson}. We remark that since $K$-theory is localising, for any noetherian ring $A$ we obtain a fibre sequence $K(A) \to G(A) \to K(D_\mathrm{sg}(A))\to$, and moreover since $G(A)$ is connective we see that $K_{-n}(D_\mathrm{sg}(A))\cong K_{-n-1}(A)$ for $n>0$.

 If $A$ is a Gorenstein ring, Buchweitz proves that the singularity category of $A$ can be obtained from the stable module category $\underline{\mathbf{mod}}(A)$ by stabilising (in the sense of \cite{HellerStab}) the syzygy functor $\Omega$ \cite{buchweitz}. This has been significantly generalised by subsequent authors \cite{KellerVossieck, Grandis, beligiannis, ChenWangNew}; see \cite{ChenSurvey} for a survey.

 			\begin{footnotesize}
		\bibliographystyle{alpha}
		\bibliography{sgnotes.bib}
	\end{footnotesize}
\end{document}